\documentclass[11pt]{article}

\usepackage[T1]{fontenc}
\usepackage[utf8]{inputenc}
\usepackage{lmodern}
\usepackage[margin=1in]{geometry}
\usepackage{amsmath,amsthm,amssymb,amsfonts,amscd}
\usepackage{mathtools}
\usepackage{graphicx}
\usepackage{xcolor}
\usepackage[mathscr]{euscript}
\usepackage{microtype}
\usepackage[hidelinks]{hyperref}
\usepackage[nameinlink,noabbrev]{cleveref}

\newtheorem{theorem}{Theorem}[section]
\newtheorem{lemma}[theorem]{Lemma}
\newtheorem{proposition}[theorem]{Proposition}
\newtheorem{corollary}[theorem]{Corollary}

\theoremstyle{remark}
\newtheorem{remark}[theorem]{Remark}

\providecommand{\widebar}[1]{\overline{#1}}
\numberwithin{equation}{section}

\definecolor{grey}{rgb}{0.5,0.5,0.5}
\definecolor{lightgrey}{rgb}{0.9,0.9,0.9}
\definecolor{darkgreen}{rgb}{0,0.6,0}
\definecolor{orange}{rgb}{1,0.5,0}
\definecolor{lightpink}{rgb}{1,0.714,0.757}
\definecolor{lightorange}{rgb}{1,0.855,0.725}

\def\bbA{{\Delta}}
\def\bbD{{\mathbb D}}

\def\bbN{{\mathbb N}}
\def\bbR{{\mathbb R}}
\def\bbS{{\mathbb S}}
\def\bbT{{\mathbb T}}

\def\U{{\mathcal U}}

\def\rD{{\rm D}}

\def\rI{{\rm I}}

\def\bN{{\mathbf{N}}}
\def\bfe{{\mathbf{e}}}
\def\bfu{{\text{\footnotesize$\mathbf{U}$}}}
\def\bfv{{\text{\footnotesize$\mathbf{V}$}}}
\def\bfw{{\text{\footnotesize$\mathbf{W}$}}}

\def\f{\text{\bf\emph{f}}}

\def\u{\text{\bf\emph{u}}}
\def\v{\text{\bf\emph{v}}}
\def\w{\text{\bf\emph{w}\hspace{1pt}}}

\def\y{\text{\bf\emph{y}}}
\def\z{\text{\bf\emph{z}}}
\def\rg{{\mathrm g}}

\def\sbF{{\text{\bf\emph{F}\hspace{-1pt}}}}

\def\div{{\operatorname{div}}}
\def\Def{{\operatorname{Def}}}

\def\contsubset{\hspace{1pt}{\hookrightarrow}\hspace{1pt}}

\def\n{{3}}
\def\Forall{\forall\hspace{2pt}}
\def\dfrac#1#2{\text{\small$\displaystyle{}\frac{#1}{#2}$}}
\def\smallint#1#2{{\text{\small$\displaystyle{}\int_{#1}^{#2}$\hspace{1pt}}}}

\def\dx{{\hspace{1pt}dx}}
\def\dy{{\hspace{1pt}dy}}

\def\dt{{\hspace{1pt}dt}}
\def\dS{{\hspace{1pt}dS}}
\def\scl{{\hspace{1pt};}}
\def\cm{{\hspace{1pt},}}
\def\pd{{\hspace{1pt}.}}

\def\bdy{{\partial}}
\def\Gin{{\Gamma_{\!\mathrm{in}}}}
\def\Gout{{\Gamma_{\!\text{out}}}}

\def\uB{{\u_{{}_B}\hspace{-1pt}}}
\def\rB{{\varrho_{{}_B}}}
\def\tB{{\vartheta_{\!{}_B}}}
\def\uBl{{\underline{\uB}}}
\def\uBu{{\widebar{\uB}}}
\def\rou{{\overline{\varrho_{{}_0}}}}
\def\rBu{{\widebar{\rB}}}
\def\rol{{\underline{\varrho_{{}_0}}}}
\def\rBl{{\underline{\rB}}}
\def\hpt{{\hspace{1pt}}}

\def\bD{{\widebar{\rD}}}
\def\bp{{\widebar{p}}}
\def\bu{{\widebar{\u}}}

\def\brho{{\widebar{\varrho}\hpt}}
\def\btheta{{\widebar{\vartheta}}}
\def\Varphi{{\boldsymbol{\varphi}}}

\def\sbV{{\!\text{\bf\emph{V}}\!}}

\title{Blow-up criterion to the compressible full Navier--Stokes equations
with inflow--outflow boundary conditions}

\author{C. H. Arthur Cheng\thanks{Department of Mathematics, National Central
University, Taoyuan, Taiwan. Email: \texttt{cchsiao@math.ncu.edu.tw}. The work
of this author was partially supported by the National Science and Technology
Council (NSTC), Taiwan, under grants NSTC 113-2115-M-008-004-MY2 and NSTC
115-2115-M-008-003-MY2.}
\and Young-Sam Kwon\thanks{Department of Mathematics, Dong-A University,
Busan, Korea. Email: \texttt{ykwon7210@gmail.com}.}
\and Cheng-Fang Su\thanks{Corresponding author. Department of Applied
Mathematics, National Yang Ming Chiao Tung University, Hsinchu, Taiwan.
Email: \texttt{scf1204@nycu.edu.tw}.}}

\date{}

\hypersetup{
  pdftitle={Blow-up criterion to the compressible full Navier-Stokes equations with inflow-outflow boundary conditions},
  pdfauthor={C. H. Arthur Cheng, Young-Sam Kwon, and Cheng-Fang Su}
}

\begin{document}

\maketitle

\begin{abstract}
We study the three-dimensional compressible Navier--Stokes--Fourier system
with genuine inflow--outflow boundary conditions. We establish local
existence of strong solutions in higher-order Sobolev spaces under basic
initial--boundary compatibility conditions. For time-independent boundary
data, we also prove a continuation criterion: a strong solution can be
continued as long as the velocity gradient is integrable in time in
$L^\infty$ and the maximum temperature has suitable time integrability. The
analysis hinges on boundary estimates for differentiated continuity
equations that exploit the sign of the normal velocity on the inflow
boundary.
\end{abstract}



\section{Introduction}

Let $\Omega\subset\mathbb{R}^{3}$ be a bounded domain with sufficiently smooth
boundary. We study the initial--boundary value problem for a viscous,
compressible, and heat-conducting fluid with density $\varrho$, velocity $\u$,
and absolute temperature $\vartheta$. In the normalized ideal-gas setting used
throughout this paper, the pressure is $p=\varrho\vartheta$, the viscosity
operator is $\Delta\u$, and the heat conductivity $\kappa$ is a positive
constant. The governing system is
\begin{subequations}\label{general_equations}
\begin{alignat}{2}
\varrho_t+\div(\varrho\u)&=0
&&\text{in }\Omega\times(0,T),\label{density_eq}\\
(\varrho\u)_t+\div(\varrho\u\otimes\u)
&=-\nabla p+\Delta\u+\varrho\f
&&\text{in }\Omega\times(0,T),\label{momentum_eq}\\
\varrho(\vartheta_t+\u\cdot\nabla\vartheta)
&=\kappa\Delta\vartheta+|\nabla\u|^2-p\,\div\u+\varrho g\quad
&&\text{in }\Omega\times(0,T),\label{temperature_eq}\\
(\varrho,\u,\vartheta)&=(\varrho_{{}_0},\u_{{}_0},\vartheta_0)
&&\text{on }\Omega\times\{t=0\},\label{initial_condition1}\\
(\u,\vartheta)&=(\uB,\tB)
&&\text{on }\partial\Omega\times(0,T),\label{boundary_condition1}\\
\varrho&=\rB
&&\text{on }\Gin\times(0,T).\label{boundary_condition2}
\end{alignat}
\end{subequations}
Here $\f$ and $g$ are the external force and heat source, respectively.

The system is open in the sense that mass may cross the boundary. We assume
that $\partial\Omega$ is the disjoint union of compact sets $\Gin$ and $\Gout$
and that, with $\bN$ denoting the outward unit normal,
\[
 \uB\cdot\bN<0\quad\text{on }\Gin,
 \qquad
 \uB\cdot\bN\geq0\quad\text{on }\Gout.
\]
Accordingly, the density is prescribed only on the inflow part $\Gin$, whereas
the velocity and temperature are prescribed on the whole boundary. Unless
stated otherwise, the boundary data may depend on time. We extend $\uB$ and
$\tB$ from $\partial\Omega$ to $\Omega$ by solving, at each time,
\begin{subequations}\label{extension_of_uBtB}
\begin{alignat}{2}
\Delta(\widetilde{\u},\widetilde{\vartheta})&={\bf0}
&&\text{in }\Omega,\\
(\widetilde{\u},\widetilde{\vartheta})&=(\uB,\tB)\quad
&&\text{on }\partial\Omega.
\end{alignat}
\end{subequations}
By a harmless abuse of notation, these extensions are again denoted by $\uB$
and $\tB$.

The inflow boundary creates a difficulty that is absent in an impermeable
domain. The continuity equation is hyperbolic in the normal direction, so the
density must be prescribed on $\Gin$ but not on $\Gout$. Moreover,
differentiating the continuity equation produces boundary fluxes involving
spatial derivatives of $\varrho$. Their signs and traces must be controlled
without imposing a density condition on the outflow boundary. This issue is
central both to the construction of high-order strong solutions and to their
continuation.

The classical theory of local strong solutions for compressible,
heat-conducting fluids goes back to Nash~\cite{Nash1962}, Matsumura and
Nishida~\cite{MatsumuraNishida1983}, and Valli and
Zajaczkowski \cite{VaZa1986}. Most of this theory concerns the Cauchy problem
or impermeable boundaries. Of particular relevance here, Valli and
Zajaczkowski treated a time-dependent non-isothermal system allowing genuine
inflow and outflow in a Hilbert-space framework. Their strong-solution theory
provides high-order spatial and temporal regularity, but requires
equation-generated compatibility conditions at the initial boundary. We
recall the relevant statement below.

The subsequent existence theory for open compressible systems developed
predominantly at the level of weak solutions. For the time-dependent
barotropic system, Novo~\cite{Novo2005} proved the global existence of weak
solutions with inflow and outflow. Choe, Novotn\'y, and
Yang~\cite{ChoeNovotnyYang2018}, and Chang, Jin, and
Novotn\'y~\cite{ChangJinNovotny2019}, subsequently established global weak
solutions for general inflow--outflow data, including large prescribed
boundary velocity and inflow density. Kwon and
Novotn\'y~\cite{KwonNovotny2021} constructed dissipative weak solutions and
proved their stability and weak--strong uniqueness.

For the full Navier--Stokes--Fourier system, Feireisl and
Novotn\'y~\cite{FeireislNovotny2021} introduced global finite-energy weak
solutions for physically realistic in/out-flow boundary conditions and
established weak--strong uniqueness. Chaudhuri and
Feireisl~\cite{ChaudhuriFeireisl2022} likewise proved the global existence of
weak solutions, together with weak--strong uniqueness, for nonhomogeneous
Dirichlet data for velocity and temperature and the associated inflow
condition for density. We emphasize that \emph{all} the existence results
cited in these two paragraphs are weak-solution or dissipative-solution results. In particular, weak--strong uniqueness asserts coincidence with a strong solution during the lifespan of the latter; it does not itself construct a strong solution.

The situation is markedly different for strong solutions. Restricting
attention to three-dimensional, time-dependent, non-isothermal compressible
systems with genuine mass inflow, to the best of our knowledge the only
published strong-existence theories directly relevant to the present problem
are the Hilbert-space result of Valli and
Zajaczkowski~\cite{VaZa1986} and the $L^p$--$L^q$ result of
Meliani~\cite{Meliani2026}. Piasecki and
Pokorn\'y~\cite{PiaseckiPokorny2014} also proved the existence of strong
solutions to the \emph{stationary} Navier--Stokes--Fourier system in a
three-dimensional channel with slip--inflow boundary conditions and data near
a nontrivial constant flow. Their theorem is a steady result, however, and
does not provide a local theory for the evolutionary initial--boundary value
problem considered here.

Neither of the two time-dependent strong-existence theories yields the local
existence theorem proved in this paper. We recall their precise statements in
Theorems~\ref{thm:existence} and~\ref{thm:meliani-existence} below; see
Remark~\ref{rem:strong-existence-comparison} for the detailed comparison with
our result.

Finite-time singularity results must be distinguished from continuation
criteria. Xin~\cite{Xin1998Blowup} proved finite-time breakdown for a class of
nontrivial smooth solutions of the compressible Navier--Stokes equations whose
initial density has compact support. This demonstrates that viscosity alone
does not preclude finite-time breakdown, but it does not identify a
continuation criterion. Such a criterion instead asks which lower-order
quantities must lose control when a local strong solution reaches the end of
its maximal interval. The classical model is the Beale--Kato--Majda criterion
for incompressible Euler flow~\cite{BealeKatoMajda1984}. In the compressible
setting, the density must also be controlled and, for the full system, the
temperature enters essentially.

For barotropic compressible flow, representative BKM-type criteria were
obtained by Huang and Xin~\cite{HuXi2010} and by Huang, Li, and
Xin~\cite{HuangLiXin2011}, the latter allowing vacuum states. For the full
heat-conducting system, blow-up or continuation criteria for strong or
classical solutions involving the density, temperature, velocity gradient, or
Serrin-type norms were developed in
\cite{Hu2009,FanJiangOu2010,SunWangZhang2011,HuangLiWang2013,HuangLi2013,
JiuWangYe2021,XieXuZhang2026}. Feireisl, Novotn\'y, and
Sun~\cite{FeireislNovotnySun2014}, by contrast, formulated a regularity
criterion for finite-energy weak solutions and used weak--strong uniqueness
together with estimates for the associated local strong solution. A
complementary direction seeks amplitude criteria in the spirit of Nash's
regularity program~\cite{Nash1958}.

Basari\'c, Feireisl, and Mizerov\'a~\cite{BaFeMi2023} established conditional
regularity for strong solutions of the Navier--Stokes--Fourier system under
boundedness of the basic thermodynamic fields and velocity. Their boundary
velocity is tangential, $\uB\cdot\bN=0$, so there is no mass inflow. Related
results include the work of Feireisl, Wen, and
Zhu~\cite{FeireislWenZhu2024} and the general-equation-of-state extension of
Abbatiello, Basari\'c, and
Chaudhuri~\cite{AbbatielloBasaricChaudhuri2025}. Abbatiello, Basari\'c,
Chaudhuri, and
Feireisl~\cite{AbbatielloBasaricChaudhuriFeireisl2026} further developed local
well-posedness and conditional regularity for inhomogeneous
Dirichlet--Neumann data. Their argument assumes that the relevant material
derivative of the velocity vanishes on the boundary and therefore does not
cover a genuine mass-inflow boundary.

The closest continuation result for a genuinely open strong flow is the work
of Abbatiello and Meliani~\cite{AbbatielloMeliani2026}. They considered the
three-dimensional \emph{barotropic} compressible Navier--Stokes system and the
local strong solution in the $L^p$--$L^q$ class described above. They proved
that, if the maximal existence time is finite, then control must be lost over
$\varrho^{-1}$, $\u$, or a suitable norm of $\nabla\varrho$. This is a
strong-solution continuation criterion, not a weak-solution existence result.
Their analysis also identifies why the standard material-derivative argument
used for $\uB\cdot\bN=0$ fails at a genuine inflow--outflow boundary. Since the
barotropic model contains no temperature equation, it does not address the
coupling among pressure work, viscous heating, and heat conduction. The
extension of this open-boundary criterion to the Navier--Stokes--Fourier system
is therefore a separate problem.

The contributions of the present paper are as follows. First, we construct a
local strong solution to~\eqref{general_equations} in the higher-order Sobolev
class
\[
 \varrho\in C_tH^2_x,\qquad
 (\u,\vartheta)\in L^\infty_tH^2_x\cap L^2_tH^3_x,
\]
from $H^2$ initial data satisfying the natural initial--boundary trace
conditions, including $\varrho_{{}_0}=\rB(0)$ on $\Gin$. In comparison with
the classical Hilbert-space construction recalled from~\cite{VaZa1986}, we do
not impose the additional $H^1_0$ compatibility conditions on the first time
derivatives generated by the equations; this reduction is accompanied by a
weaker temporal regularity class. Second, for time-independent boundary data,
we prove a continuation criterion involving
$\nabla\u\in L^1_tL^\infty_x$ and a time-integrated maximum norm of the
temperature. Third, we establish boundary estimates for the differentiated
continuity equation that explicitly exploit the sign of $\uB\cdot\bN$ on
$\Gin$. To the best of our knowledge, this is the first continuation criterion
for the full Navier--Stokes--Fourier dynamics with genuine inflow and outflow.

For comparison, we recall the following classical local strong-solution result,
which is a reformulation of Theorem~2.5 in~\cite{VaZa1986}.

\begin{theorem}[Reformulation of Theorem 2.5 in \cite{VaZa1986}]
\label{thm:existence}
Let $\partial\Omega$ be of class ${\mathscr{C}}^3$. Suppose that the external
force $\f$ and the heat source $g$ satisfy
\begin{align*}
 \f&\in L^2_{\mathrm{loc}}(0,\infty;H^1(\Omega)^3),\qquad
 \f_t\in L^2_{\mathrm{loc}}(0,\infty;L^2(\Omega)^3),\\
 g&\in L^2_{\mathrm{loc}}(0,\infty;H^1(\Omega)),\qquad
 g_t\in L^2_{\mathrm{loc}}(0,\infty;L^2(\Omega)).
\end{align*}
Assume that $\varrho_{{}_0}\in H^2(\Omega)$ and
$\varrho_{{}_0}>0$ in $\Omega$, that $\vartheta_0\in H^2(\Omega)$ and
$\vartheta_0>0$ in $\Omega$, and that
\[
 (\u_{{}_0}-\uB(0),\vartheta_0-\tB(0))\in H^1_0(\Omega)^4.
\]
Suppose, in addition, that
\begin{align}
&\frac{1}{\varrho_{{}_0}}
 \left[\div\bbS(\bbD\u_{{}_0})
       -\nabla p(\varrho_{{}_0},\vartheta_0)\right]
 -\uB_t(0)-\u_{{}_0}\cdot\nabla\u_{{}_0}+\f(0)
 \in H^1_0(\Omega)^3,\nonumber\\
&\frac{1}{c_v(\varrho_{{}_0},\vartheta_0)\varrho_{{}_0}}
 \left[\kappa\Delta\vartheta_0
       -\vartheta_0p_\vartheta(\varrho_{{}_0},\vartheta_0)\div\u_{{}_0}
       +\bbS(\bbD\u_{{}_0}):\nabla\u_{{}_0}
       +\varrho_{{}_0}g(0)\right]\nonumber\\
&\hspace{4cm}
 -\tB_t(0)-\u_{{}_0}\cdot\nabla\vartheta_0
 \in H^1_0(\Omega).
\label{initial_condition}
\end{align}
Then there exists $T>0$ and a strong solution
$(\varrho,\u,\vartheta)$ satisfying
\begin{enumerate}
\item[\rm1.]
$\varrho\in {\mathscr C}^0([0,T];H^2(\Omega))$ and
$\varrho_t\in {\mathscr C}^0([0,T];H^1(\Omega))$;
\item[\rm2.]
$(\u,\vartheta)\in L^2(0,T;H^3(\Omega))^4$,
$(\u_t,\vartheta_t)\in L^2(0,T;H^2(\Omega))^4$, and
$$
(\u_{tt},\vartheta_{tt})\in L^2(0,T;L^2(\Omega))^4;
$$
\item[\rm3.]
$\varrho>0$ and $\vartheta>0$ on $\widebar{\Omega}\times[0,T]$.
\end{enumerate}
\end{theorem}

For completeness, we also record the following zero-magnetic-field
specialization of Meliani's $L^p$--$L^q$ existence theorem.

\begin{theorem}[Specialization of Theorem 3.1 in \cite{Meliani2026}]
\label{thm:meliani-existence}
Let $\Omega\subset\mathbb{R}^3$ be a bounded domain with boundary of class
${\mathscr C}^2$, and suppose that $\Gin$ is a closed surface. Let
$p,q\in(1,\infty)$ satisfy
\begin{equation}\label{Meliani_pq_condition_intro}
 q>3,\qquad
 p>\max\left\{\frac{2q}{q-1},\frac{2q}{2q-3}\right\}.
\end{equation}
Assume that the initial and boundary density data satisfy
\begin{align*}
 \rB&>0,\qquad
 \rB\in W^{1,q}(0,T;L^q(\Gin))
       \cap L^q(0,T;W^{1,q}(\Gin)),\\
 \varrho_{{}_0}&>0,\qquad
 \varrho_{{}_0}\in W^{1,q}(\Omega),\qquad
 \varrho_{{}_0}=\rB(0)\quad\text{on }\Gin.
\end{align*}
Suppose further that
\begin{align*}
 \u_{{}_0},\vartheta_0&\in B^{2(1-1/p)}_{q,p}(\Omega),\\
 \tB>0,\qquad \uB,\tB&\in F^{1-1/(2q)}_{p,q}(0,T;L^q(\partial\Omega))
 \cap L^p(0,T;W^{2-1/q,q}(\partial\Omega)),
\end{align*}
with
\[
 \u_{{}_0}=\uB(0),\qquad \vartheta_0=\tB(0)
 \quad\text{on }\partial\Omega,
 \qquad
 -\uB\cdot\bN\geq c>0
 \quad\text{on }[0,T]\times\Gin.
\]
Assume also that the gravitational potential satisfies
$G\in L^p(0,T;W^{1,q}(\Omega))$. Then there exists $T_*>0$ such that the
corresponding open Navier--Stokes--Fourier subsystem considered
in~\cite{Meliani2026}, with force $\nabla G$ and obtained by taking the magnetic
initial and boundary data equal to zero, admits a unique strong solution
satisfying
\begin{align*}
 \varrho&\in L^\infty(0,T_*;W^{1,q}(\Omega))
       \cap W^{1,\infty}(0,T_*;L^q(\Omega)),\\
 (\u,\vartheta)&\in L^p(0,T_*;W^{2,q}(\Omega))
       \cap W^{1,p}(0,T_*;L^q(\Omega)).
\end{align*}
\end{theorem}

\begin{remark}[Comparison of the strong-existence results]
\label{rem:strong-existence-comparison}
The theorem of Valli and Zajaczkowski requires the first time derivatives
generated by the momentum and temperature equations to satisfy additional
$H^1_0$ compatibility conditions at $t=0$. Theorem~\ref{thm:main} assumes
$H^2$ initial data and only their basic trace compatibility with the prescribed
velocity and temperature on the boundary. It therefore admits initial data
excluded by Theorem~\ref{thm:existence}, at the cost of working with a weaker
temporal regularity class.

Theorem~\ref{thm:meliani-existence} establishes local well-posedness
for a time-dependent open full system. The system couples the
Navier--Stokes--Fourier equations with magnetohydrodynamics and is
formulated in the $L^p$--$L^q$ framework. If the initial and boundary
data for the magnetic field vanish, the corresponding
Navier--Stokes--Fourier subsystem is recovered.

The restriction on $(p,q)$ in
\eqref{Meliani_pq_condition_intro} is also imposed in Assumption~2
of Abbatiello and Meliani~\cite{AbbatielloMeliani2026}. By contrast,
Theorem~\ref{thm:main} yields the higher-order Sobolev regularity
\begin{align*}
 \varrho
 &\in C([0,T];H^2(\Omega)),\\
 (\u,\vartheta)
 &\in L^\infty(0,T;H^2(\Omega))
      \cap L^2(0,T;H^3(\Omega)).
\end{align*}
In three dimensions, the Sobolev embedding further gives
$$
 H^3(\Omega)
 \hookrightarrow W^{2,6}(\Omega),\qquad
 (\u,\vartheta)
 \in L^2(0,T;W^{2,6}(\Omega)).
$$
At $q=6$, however,
\eqref{Meliani_pq_condition_intro} requires $p>12/5$ and therefore excludes
the endpoint pair $(p,q)=(2,6)$. Moreover, for $q=6$ and $p>12/5$, the initial
trace space $B^{2(1-1/p)}_{6,p}(\Omega)$ required in the $L^p$--$L^q$
theory is not guaranteed by $H^2(\Omega)$ initial data. Thus Meliani's theorem
does not imply either our $C_tH^2_x$ estimate for the density or our
$L^2_tH^3_x$ estimate for velocity and temperature, and the present existence
theorem is not a corollary of that result.

For clarity, on a finite time interval,
\[
 L^\infty(0,T;W^{2,6}(\Omega))
 \hookrightarrow L^p(0,T;W^{2,6}(\Omega)),\qquad 1<p<\infty.
\]
This does not mean that the $L^p$--$L^q$ theorem produces an
$L^\infty_tW^{2,6}_x$ bound. Nor is $L^\infty_tW^{2,6}_x$ equivalent to
$L^2_tH^4_x$. Conversely, interpolation shows that a solution already
constructed in our Sobolev class belongs to certain admissible $L^p$--$L^q$
spaces with $3<q<4$; this one-way embedding does not recover the higher-order
Sobolev estimates from the $L^p$--$L^q$ existence theorem. The novelty of our
existence result is therefore the construction and propagation of this
particular higher-order Hilbert-space class under only the basic
initial--boundary compatibility, rather than the introduction of a disjoint
notion of strong solution.
\end{remark}

\subsection{The main results}

We now state the main results of this paper. The first part concerns the local-in-time existence of strong solutions to \eqref{general_equations}. The second part gives a continuation criterion, or equivalently a blow-up criterion, for such strong solutions. For this latter part, we will additionally assume that the boundary data are independent of time.

\begin{theorem}\label{thm:main}
Let $\partial\Omega$ be of class ${\mathscr{C}}^3$. Suppose that the external forcing $\f$ and heat source $g$ satisfy
\begin{align*}
\f \in L^2(0,T;H^1(\Omega))^3,\quad &\f_{\!t} \in L^2(0,T;H^{-1}(\Omega))^3,\\
g \in L^2(0,T;H^1(\Omega)),\quad &g_t \in L^2(0,T;H^{-1}(\Omega))
\end{align*}
and $\uB$, $\tB$ are functions denoting the boundary data satisfying
\begin{alignat*}{2}
(\partial_t^k \uB, \partial_t^k \tB) &\in L^2(0,T;H^{3-2k}(\Omega))^4 &&\text{for $k=0,1,2$}, \\
(\partial_t^k \uB, \partial_t^k \tB) &\in L^\infty(0,T;H^{2-2k}(\Omega))^4 \qquad&&\text{for $k=0,1$.}
\end{alignat*}
Assume also that $\rB$ is strictly positive and has the regularity required by
the inflow transport problem and the boundary estimates below, with
$\varrho_{{}_0}=\rB(0)$ on $\Gin$. Let $(\varrho_{{}_0}, \u_{{}_0}, \vartheta_0) \in H^2(\Omega)^5$, and $\varrho_{{}_0} > 0$ on $\widebar{\Omega}$, $\vartheta_0 > 0$ on $\widebar{\Omega}$. If it holds the first order compatibility conditions
$$
\u_{{}_0} = \uB \text{ \ and \ } \vartheta_0 = \tB \qquad\text{on}\quad \bdy\Omega\cm
$$
then there exists $0 < T^* \le T$ and functions $(\varrho,\u,\vartheta)$ satisfying
\begin{enumerate}
\item[\rm1.] $\varrho \in {\mathscr{C}}([0,T^*];H^2(\Omega))$, $\varrho_t \in {\mathscr{C}}([0,T^*];H^1(\Omega))$;
\item[\rm2.] $(\partial_t^k \u, \partial_t^k \vartheta) \in L^2(0,T^*;H^{3-2k}(\Omega))^4$ for $k=0,1,2$ and\\ $(\partial_t^k \u, \partial_t^k \vartheta) \in L^\infty(0,T^*;H^{2-2k}(\Omega))^4$ for $k=0,1$;
\item[\rm3.] $\varrho(x,t) > 0$ and $\vartheta(x,t) > 0$ in $\widebar{Q_{T^*}}$;
\end{enumerate}
such that $(\varrho,\u,\vartheta)$ is a solution to \eqref{general_equations}.
\end{theorem}

The point of comparison with~\cite{VaZa1986} is the compatibility structure,
not a blanket weakening of every hypothesis. In the classical construction,
the traces of the equation-generated quantities
\[
\u_t(0)=\frac{1}{\varrho_{{}_0}}
\big[\Delta\u_{{}_0}-\nabla(\varrho_{{}_0}\vartheta_0)\big]
-\u_{{}_0}\cdot\nabla\u_{{}_0}+\f(0)
\]
and
\[
\vartheta_t(0)=\frac{1}{\varrho_{{}_0}}
\big[\kappa\Delta\vartheta_0+|\nabla\u_{{}_0}|^2
-\varrho_{{}_0}\vartheta_0\div\u_{{}_0}\big]
-\u_{{}_0}\cdot\nabla\vartheta_0+g(0)
\]
must match the time derivatives of the boundary data in a stronger trace
sense. Theorem~\ref{thm:main} retains the $H^2$ initial regularity but requires
only the displayed trace compatibility of $\u_{{}_0}$ and $\vartheta_0$.

We now formulate the continuation result using a maximal existence time; this
avoids ambiguity between a local existence time and a genuine blow-up time.

\begin{remark}
The comparison with Theorem~\ref{thm:existence} concerns compatibility, not
only the nominal Sobolev order of the initial data. Although
Theorem~\ref{thm:main} assumes
\[
 (\varrho_0,\u_0,\vartheta_0)\in H^2(\Omega)^5,
\]
the result of Valli and Zajaczkowski additionally requires the first time
derivatives generated by the momentum and temperature equations to satisfy
$H^1_0$ boundary compatibility conditions at $t=0$. By contrast,
Theorem~\ref{thm:main} requires only the basic trace conditions on
$\u_0$ and $\vartheta_0$. The Sobolev embeddings
$H^2(\Omega)\hookrightarrow W^{1,6}(\Omega)\hookrightarrow L^\infty(\Omega)$
ensure that the equation-generated initial time derivatives belong to
$L^2(\Omega)$, which is the level used in our construction.
\end{remark}
The following is the blow-up criterion result as the second result.
\begin{theorem}\label{thm:main2}
Given assumptions in Theorem \ref{thm:main}, we further assume that the boundary data are independent of time. If in addition
\[
\f \in L^1(0,T;L^\infty(\Omega)^3),
\]
the solution $(\varrho,\u,\vartheta)$ blows up at $t = T^*$ if and only if there exists $\delta > 0$ such that
\begin{equation}\label{our_blow-up_criterion}
\lim_{T \nearrow T^*} \int_0^T \Big[\|\nabla \u(t)\|_{L^\infty(\Omega)} + \|\vartheta(t)\|_{L^\infty(\Omega)}^{5+\delta}\Big] \, dt = \infty\pd
\end{equation}
\end{theorem}

We make three remarks on the blow-up criterion in Theorem \ref{thm:main2} regarding three comments on improvement in this paper and why the methods used in previous results are not available.

\begin{remark}
The criterion \eqref{our_blow-up_criterion} should be understood as a Beale–Kato–Majda type condition adapted to the in-flow/out-flow boundary. In several blow-up criteria for the compressible Navier--Stokes--Fourier system, the continuation argument is carried out under the assumption
\begin{equation*}
(\varrho,\u,\vartheta)\in L^\infty(0,T;L^\infty(\Omega)).
\end{equation*}
However, the boundary conditions of these studies are fundamentally different from those of this study.  In particular, the condition
\begin{equation*}
\u\cdot\bN=0
\quad\text{on }\partial\Omega
\end{equation*}
plays a crucial role in the estimates. This boundary structure allows one to use test functions of the type
\begin{equation*}
D_t\u-\u\cdot\nabla\uB
\end{equation*}
in the higher order energy estimates. We shall point out the precise estimate later; see Section~\ref{sec:L2H1_estimate_for_Dtu} below. This mechanism breaks down for the in-flow/out-flow boundary condition considered here, since in general
\begin{equation*}
\u\cdot\bN\neq 0
\quad\text{on }\partial\Omega.
\end{equation*}
Consequently,
\begin{equation*}
D_t\u-\u\cdot\nabla\uB
\end{equation*}
does not satisfy the homogeneous boundary condition needed for the corresponding test-function argument. In fact, this obstruction already appears in the simpler pure out-flow case. Therefore, the continuation argument based only on
\begin{equation*}
(\varrho,\u,\vartheta)\in L^\infty(0,T;L^\infty(\Omega))
\end{equation*}
is not available in the present setting. Instead, we impose the  Beale–Kato–Majda type condition
\begin{equation*}
\nabla\u\in L^1(0,T;L^\infty(\Omega)),
\end{equation*}
which directly controls the deformation of the flow and allows us to close the higher order estimates in the presence of in-flow and out-flow boundaries.
\end{remark}
\begin{remark}
The temperature condition in \eqref{our_blow-up_criterion},
\begin{equation*}
\vartheta\in L^{5+\delta}(0,T;L^\infty(\Omega)),
\qquad \delta>0,
\end{equation*}
is not meant to be a structural restriction of the model. One could replace it by the stronger and more standard assumption
\begin{equation*}
\vartheta\in L^\infty(0,T;L^\infty(\Omega)),
\end{equation*}
which immediately implies the above condition on every finite time interval. The purpose of using the integrability condition
\begin{equation*}
\int_0^T \|\vartheta(t)\|_{L^\infty(\Omega)}^{5+\delta},dt<\infty
\end{equation*}
is to formulate the continuation criterion in a more quantitative way. Namely, our estimates require only a certain amount of time integrability of the maximum temperature, rather than a uniform-in-time upper bound. Thus the exponent $5+\delta$ reflects the level of temporal integrability needed by the present argument. In this sense, the condition is a technical refinement of the usual $L^\infty_tL^\infty_x$ assumption, and it measures how far one can weaken the temperature control while still retaining continuation of the strong solution.
\end{remark}

\begin{remark}

In Theorem \ref{thm:main2} we assume that the initial density is strictly positive,
\begin{equation*}
\varrho_0>0
\quad\text{on }\overline{\Omega}.
\end{equation*}
This is stronger than the assumptions in some blow up criterion  results, such as \cite{HuangLiXin2011}, where vacuum is allowed initially. The reason is again related to the in-flow/out-flow boundary condition. In the present problem, the density is prescribed on the in-flow part of the boundary, and the higher order estimates for $\varrho$ must be derived directly from the continuity equation. When differentiating the density equation, one encounters boundary contributions of the form
\begin{equation*}
\int_{\Gamma_{\mathrm{in}}}
(\u\cdot\bN)
|\nabla^k\varrho|^2\,dS,
\end{equation*}
together with lower order commutator terms. These boundary terms have no analogue in the impermeable case $\u\cdot\bN=0$, and their control is one of the main difficulties caused by the in-flow boundary. In our approach, the estimates for these terms rely on the strict positivity of the density, together with the prescribed positive in-flow density, in order to propagate the required lower and upper bounds for $\varrho$. Therefore, the assumption
\begin{equation*}
\inf_{\overline{\Omega}}\varrho_0>0
\end{equation*}
is not merely a technical convenience; it is tied to the use of the density equation at the boundary and to the in-flow/out-flow structure of the problem.

\end{remark}

 To summarize, the criterion \eqref{our_blow-up_criterion} differs from previous $L^\infty_tL^\infty_x$-type continuation criteria because the present in-flow/out-flow boundary condition destroys the boundary cancellation used in the impermeable case. The replacement of the velocity bound by the BKM-type condition on $\nabla\u$, the quantitative time-integrability condition on $\vartheta$, and the strict positivity assumption on $\varrho_0$ are all consequences of this boundary structure.

\subsection{Outlines}
The rest of this paper is organized as follows. In Section 2 we
introduce the basic assumptions and derive several preliminary a priori estimates, including the blow-up criterion, estimates for the density under the inflow-outflow boundary condition, and the linear parabolic regularity result used in the local existence argument. In Section 3 we prove the local-in-time existence of strong solutions by a fixed-point argument. Section 4 is devoted to the proof of the continuation criterion. Finally, the appendices collect the boundary estimates on the inflow part and the auxiliary technical estimates used in the local construction.

\section{Basic Assumptions and Some A-Priori Bounds}
Let $\sbV$ solve the Poisson equation
\begin{alignat*}{2}
\bbA \sbV &= \bbA \uB \qquad&&\text{in}\quad\Omega\cm\\
\sbV &= {\bf 0} &&\text{on}\quad\bdy\Omega\cm
\end{alignat*}
and set $\widetilde{\u}_{{}_B} = \uB - \sbV$. Then $\bbA \widetilde{\u}_{{}_B} = {\bf 0}$ in $\Omega$ and $\widetilde{\u}_{{}_B} = \uB$ on $\bdy\Omega$. Moreover, $\partial_t^k \widetilde{\u}_{{}_B}$ has the same regularity as $\partial_t^k \uB$; thus that the statement in the main theorem will not change when $\uB$ is replaced by $\widetilde{\u}_{{}_B}$. Therefore, without loss of generality, we can assume that $\bbA \uB = {\bf 0}$. Similarly, we also assume that $\Delta \tB = 0$ in $\Omega$.

In the following, we introduce the quantities used in this paper:
\begin{align*}
\rol &= \inf_\Omega \varrho_{{}_0}, \
\rBl = \inf_{\Gin} \rB, \
\rou = \sup_\Omega \varrho_{{}_0}, \
\rBu = \sup_{\Gin} \rB, \\
\uBu &= \sup_{\partial\Omega} |\uB|, \ \uBl = \inf_{\Gin} (-\uB\cdot \bN)\pd
\end{align*}
The quantities $\rou$, $\rBu$, $\uBu$ are assumed to be positive, and the inflow assumption
\begin{equation}\label{inflow-outflow_BC}
\uB \cdot \bN < 0 \qquad\text{on}\quad \Gin
\end{equation}
and the compactness of $\Gin$ show that $\uBl$ is also positive.

\subsection{Basic assumptions}

Throughout this section, we work under the assumptions and notations introduced above.
In particular, we regard $\uB$ and $\tB$ as their harmonic extensions to $\Omega$, and we use the quantities
\[
\varrho_0,\ \varrho_B,\ \overline{\varrho}_0,\ \overline{\varrho}_B,\ \overline{u}_B,\ \underline{u}_B
\]
to denote the bounds associated with the initial and boundary data. Moreover, the positivity of $\varrho_{{}_0}$ and $\varrho_B$, together with the inflow
condition on $\Gamma_{\mathrm{in}}$, will be used repeatedly in the a priori estimates below.

\subsection{Preliminary results}

\begin{lemma}\label{lem:lower_upper_bound_for_rho}
It holds that
\begin{align}\label{Linfty_bound_for_rho}
&\min\big\{\rol,\rBl\big\} \exp\Big(\!-\! \int_0^t \|\div \u\|_{L^\infty(\Omega)}\dt'\Big) \nonumber\\
&\quad\le \varrho(t,x) \le \max\big\{\rou,\rBu\big\} \exp\Big(\int_0^t \|\div \u\|_{L^\infty(\Omega)}\dt'\Big)\pd
\end{align}
\end{lemma}
\begin{proof}
Let $t\in [0,T]$ be fixed. Let $X = X(s,x)$ be the reverse flow map of the velocity $\u$ satisfying
$$
\partial_s X(s,x) = \u(s,X(s,x))\cm\qquad X(t,x) = x \in \Omega\pd
$$
The map $X$ is used to trace the material point $y$ from which the flow map carries $y$ to $x$ at time $t$.
Then for a function $f = f(t,x)$, using the notation $(f \circ X)(s,x) = f(t-s,X(t-s,x))$, the equation of continuity \eqref{density_eq} implies that
$$
\partial_s (\varrho \circ X) = (\varrho \circ X) (\div \u)\circ X
$$
Therefore,
\begin{equation}\label{ODE_for_rho}
\partial_s \log |\varrho \circ X| = (\div \u)\circ X\pd
\end{equation}
\begin{enumerate}
\item {\bf The case that $X(s,x) \in \Omega$ for all $0 \le s \le t$}: Integrating \eqref{ODE_for_rho} from $t$ to $0$, we obtain
    $$
    \varrho(t,x) = \varrho(t,X(t,x)) = \varrho_{{}_0}(X(0,x)) \exp\Big(\int_t^0 \big[(\div \u)\circ X\big](x,t') dt'\Big)\pd
    $$
    Therefore,
    the inequality above shows that in the case $X(s,x) \in \Omega$ for all $s\in [0,t]$,
    \begin{equation}\label{rho_bounds1}
    \rol \exp\Big(\!-\! \int_0^t \|\div \u\|_{L^\infty(\Omega)}\dt'\Big) \le \varrho(t,x) \le \rou \exp\Big(\int_0^t \|\div \u\|_{L^\infty(\Omega)}\dt'\Big)\pd
    \end{equation}
\item {\bf The case that $X(s,x) \in \Gin$ for some $0 \le s < t$ while $X(t',x) \in \Omega$ for all $s < t' \le t$}: Integrating \eqref{ODE_for_rho} from $t$ to $s$, we obtain
    \begin{align*}
    \varrho(t,x) &= \varrho(t,X(t,x)) = \varrho(s,X(s,x)) \exp\Big(\int_t^s \big[(\div \u)\circ X\big](x,t') dt'\Big)\pd \\
    &= \rB(X(s,x)) \exp\Big(\int_t^s \big[(\div \u)\circ X\big](x,t') dt'\Big)\pd
    \end{align*}
    Therefore,
    the inequality above shows that in the case $X(s,x) \in \Gin$ for some $s\in [0,t)$ while $X(t',x) \in \Omega$ for all $t' \in (s,t]$,
    \begin{equation}\label{rho_bounds2}
    \rBl \exp\Big(\!-\! \int_0^t \|\div \u\|_{L^\infty(\Omega)}\dt'\Big) \le \varrho(t,x) \le \overline{\rB} \exp\Big(\int_0^t \|\div \u\|_{L^\infty(\Omega)}\dt'\Big)\pd
    \end{equation}
\end{enumerate}
The final result follows from combining \eqref{rho_bounds1} and \eqref{rho_bounds2}.
\end{proof}

The following Lemma \ref{lem:moser_iteration} is extracted from Proposition 2.2 in \cite{HuXi2010}. We single this part out
\begin{lemma}\label{lem:moser_iteration}
Let $c_1,c_2,r > 1$ be constants, and $C_0$ be a positive constant. If $A:\bbN \to \bbR^+$ is a sequence satisfying
\begin{equation}\label{iterative_inequality}
A(k+1) \le c_1 r^{\,c_2 k} \big[A(k)^{\,r} + 1 \big] + C_0^{\hpt r^k} \qquad\forall\,k \ge k_0\cm
\end{equation}
then there exists positive constant $C = C(c_1,c_2,r,C_0, A(k_0))$ such that $A(k) \le C^{\,r^k}$ for all $k\ge k_0$.
\end{lemma}

\begin{proposition}\label{prop:2.2}
Let $\varrho,\u$ satisfy the momentum equation \eqref{momentum_eq} in the pointwise sense, and $\u=\uB$ on $\partial\Omega$. If
\[
\|\varrho\|_{L^\infty(Q_T)}<\infty
\qquad\text{and}\qquad
\|p(\varrho,\vartheta)\varrho^{-1}\|_{L^{5+\delta}(0,T;L^\infty(\Omega))}<\infty
\quad\text{for some }\,\delta > 0\cm
\]
then
\begin{equation}\label{momentum_Linfty_bound}
\begin{array}{rl}
\|\varrho \u\|_{L^\infty(Q_T)}
\le C\big(\hspace{-9pt}&\|\varrho_{{}_0}\|_{L^\infty(\Omega)}, \|\u_{{}_0}\|_{L^\infty(\Omega)}, \|\uB\|_{W^{1,\infty}(\Omega)}, \|\varrho\|_{L^\infty(Q_T)}, \vspace{.15cm}\\
&\|p/\varrho\|_{L^{5+\delta}(0,T;L^\infty(\Omega))}, \|\f\|_{L^1(0,T;L^\infty(\Omega))},\mu,\lambda,\Omega,T,\delta\big)\pd
\end{array}
\end{equation}
\end{proposition}
\begin{proof}
Let $q\ge0$ be a fixed integer number. Since $\|\varrho_{{}_0}\|_{L^\infty(\Omega)}$, $\|\u_{{}_0}\|_{L^\infty(\Omega)}$, and $\|\uB\|_{L^\infty(\Omega)}$ are finite, we have
\[
\int_\Omega \varrho_{{}_0} |\u_{{}_0}-\uB|^{q+2}\,dx \le c_0^{\,q+2}<\infty
\]
for some $c_0=c_0(\|\varrho_{{}_0}\|_{L^\infty(\Omega)},\|\u_{{}_0}\|_{L^\infty(\Omega)},\|\uB\|_{L^\infty(\Omega)})$.\vspace{.1cm}

Set $\w=\u-\uB$. With \eqref{extension_of_uBtB}, the momentum equation \eqref{momentum_eq} can be written as
\[
\varrho(\w_t+\u\cdot\nabla\w)+\nabla p(\varrho,\vartheta)
= \operatorname{div}\bbS(\nabla\w)-\varrho\,\w\cdot\nabla\uB-\varrho\,\uB\cdot\nabla\uB+\varrho\,\f.
\]
Testing this equation against $|\w|^q\w$ and using $|\w|^q\w\big|_{\partial\Omega}={\bf 0}$, we obtain
\begin{align*}
&\frac{1}{q+2}\frac{d}{dt}\int_\Omega \varrho |\w|^{q+2}\dx - \int_\Omega \operatorname{div}\bbS(\nabla \w)\cdot |\w|^q\w\dx \\
&\qquad = \int_\Omega p(\varrho,\vartheta) \div (|\w|^q \w) \dx - \int_\Omega \varrho |\w|^q \w \cdot (\w \cdot \nabla\uB) \dx \\
&\qquad\quad -\int_\Omega \varrho |\w|^q \w \cdot (\uB \cdot \nabla\uB)\dx + \int_\Omega \varrho \f \cdot |\w|^q \w \dx \pd
\end{align*}

Let $\alpha=\dfrac{q}{q+2}$\vspace{.1cm} and define $\z=|\w|^{q/2}\w$, so that $|\w|^{q+2}=|\z|^2$ and $|\w|^q\w=|\z|^\alpha\z$.
Let $\bfe=\z/|\z|$ on $\{\z\neq0\}$, and set $P_\pm=I\pm\alpha\,\bfe\otimes\bfe$.
A direct computation yields
\[
\nabla\w = |\z|^{-\alpha}P_-\,\nabla\z,
\qquad
\nabla(|\w|^q\w)=\nabla(|\z|^\alpha\z)=|\z|^{\alpha}P_+\,\nabla\z.
\]
Hence
\begin{align*}
-\int_\Omega \operatorname{div}\bbS(\nabla\w)\cdot |\w|^q\w\,dx
&=\int_\Omega \bbS(\nabla\w):\nabla(|\w|^q\w)\,dx \\
&=\int_\Omega \Big(2\mu\,\operatorname{sym}(P_- \nabla\z)+\lambda\,\operatorname{tr}(P_- \nabla\z)\,I\Big):(P_+ \nabla\z)\,dx.
\end{align*}
Since $\mu>0$ and $2\mu+3\lambda>0$, the Lam\'e form is coercive on symmetric matrices:
\[
\big(2\mu X+\lambda\,\operatorname{tr}(X)I\big):X \ge c_L\,|X|^2
\qquad \forall\, X=X^\top,
\qquad c_L:=\min\{2\mu,2\mu+3\lambda\}.
\]
Moreover, $P_-P_+=I-\alpha^2\bfe\otimes\bfe$ implies
\[
(P_-M):(P_+M)\ge (1-\alpha^2)|M|^2
\qquad \forall\, M\in\bbR^{3\times3}.
\]
Combining these facts and using Korn's inequality, we infer that there exists $c=c(\mu,\lambda,\Omega)>0$ such that
\[
-\int_\Omega \operatorname{div}\bbS(\nabla\w)\cdot |\w|^q\w\,dx
\ge c\,(1-\alpha^2)\,\|\nabla\z\|_{L^2(\Omega)}^2.
\]
Therefore, it follows that
\begin{align}
&\frac{1}{q+2}\frac{d}{dt}\int_\Omega \varrho |\z|^2 \dx + c\,(1-\alpha^2)\,\|\nabla \z\|_{L^2(\Omega)}^2 \nonumber\\
&\qquad \le \int_\Omega p(\varrho,\vartheta) \div (|\z|^\alpha\z) \dx - \int_\Omega \varrho \z \cdot (\z \cdot \nabla\uB) \dx \nonumber\\
&\qquad\quad -\int_\Omega \varrho |\z|^\alpha\z \cdot (\uB \cdot \nabla\uB)\dx + \int_\Omega \varrho \f \cdot |\z|^\alpha\z \dx \pd \label{u_Linfty_bound_ineq_temp}
\end{align}

We next estimate the right-hand side of \eqref{u_Linfty_bound_ineq_temp}. First, using \[
\operatorname{div}(|\z|^\alpha\z)=\operatorname{tr}(\nabla(|\z|^\alpha\z))\] and the pointwise bound
\[
|\nabla(|\z|^\alpha\z)| \le C\,|\z|^\alpha|\nabla\z|,
\]
Young's inequality gives rise to
\begin{align*}
\int_\Omega p(\varrho,\vartheta)\,\operatorname{div}(|\z|^\alpha\z)\,dx
&\le \frac{c(1-\alpha^2)}{2}\|\nabla\z\|_{L^2(\Omega)}^2
+ \frac{C}{(1-\alpha^2)}\int_\Omega p(\varrho,\vartheta)^2 |\z|^{2\alpha}\,dx.
\end{align*}
Since $2\alpha=2q/(q+2)$ and $\|\varrho\|_{L^\infty(Q_T)}<\infty$, H\"older's inequality yields
\[
\int_\Omega p(\varrho,\vartheta)^2 |\z|^{2\alpha}\,dx
\le C_1 \big\|p(\varrho,\vartheta)/\varrho\big\|_{L^\infty(\Omega)}^2
\Big(\int_\Omega \varrho|\z|^2\,dx\Big)^\alpha,
\]
for some $C_1=C_1(\|\varrho\|_{L^\infty(Q_T)},\Omega)$.\vspace{.1cm}

Next, since $\|\nabla\uB\|_{L^\infty(\Omega)}<\infty$ and $\|\varrho\|_{L^\infty(Q_T)}<\infty$, we have
\begin{align*}
\left|\int_\Omega \varrho|\z|^\alpha \z\cdot(\w\cdot\nabla\uB)\,dx\right| &
\le C_2 \Big(\int_\Omega \varrho|\z|^2\,dx\Big)^{\frac{q+1}{q+2}}, \\
\left|\int_\Omega \varrho|\z|^\alpha \z\cdot(\uB\cdot\nabla\uB)\,dx\right| &
\le C_2 \Big(\int_\Omega \varrho|\z|^2\,dx\Big)^{\frac{q+1}{q+2}},
\end{align*}
with $C_2=C_2(\|\varrho\|_{L^\infty(Q_T)},\|\uB\|_{W^{1,\infty}(\Omega)},\Omega)$.
Finally, the force term is bounded by
\[
\left|\int_\Omega \varrho\,\f\cdot|\z|^\alpha\z\,dx\right|
\le \|\f\|_{L^\infty(\Omega)}\int_\Omega \varrho|\z|^2\,dx.
\]

Collecting the above estimates, using $(q+2)(1-\alpha^2)=\dfrac{4(q+1)}{q+2}\ge 2$, and absorbing the gradient term, we arrive at
\begin{equation}\label{ddt_of_u_in_Lq+2_gen_new}
\frac{d}{dt}\int_\Omega \varrho|\z|^2\,dx + c\|\nabla\z\|_{L^2(\Omega)}^2
\le C_3\Big[(q+2)^2\phi(t) + (q+2)\psi(t)\Big]\Big(1+\int_\Omega \varrho|\z|^2\,dx\Big),
\end{equation}
where $\phi(t):=\|p(\varrho,\vartheta)/\varrho\|_{L^\infty(\Omega)}^2$, $\psi(t):=1+\|\f(t)\|_{L^\infty(\Omega)}$, and $C_3$ depends only on $c$, $C_1$, and $C_2$.

Let $E(t):= \smallint{\Omega}{} \varrho(t)|\z(t)|^2\,dx$ and define the integrating factor by\vspace{-.15cm}
\[
\Lambda(t):=\exp\Big(-C_3(q+2)\int_0^t\psi(s)\,ds\Big).
\]
Multiplying \eqref{ddt_of_u_in_Lq+2_gen_new} by $\Lambda(t)$, we obtain
\begin{equation}\label{ddt_of_u_in_Lq+2_gen}
\frac{d}{dt}\big[\Lambda(t)(1+E(t))\big] + c\,\Lambda(t)\|\nabla\z(t)\|_{L^2(\Omega)}^2
\le C_3(q+2)^2\phi(t)\,\Lambda(t)(1+E(t)).
\end{equation}
The Gronwall's inequality also deduces
\begin{equation}\label{consequence_of_Grownwall}
\Lambda(t)(1+E(t))
\le \exp\Big(C_3\int_0^t (q+2)^2\phi(s)\,ds\Big)\Big(1+\int_\Omega \varrho_{{}_0}|\u_{{}_0}-\uB|^{q+2}\,dx\Big),
\end{equation}
and integrating \eqref{ddt_of_u_in_Lq+2_gen} over $(0,T)$ yields
\begin{equation}\label{2.15_gen}
\begin{array}{l}
\displaystyle
\sup_{t\in[0,T]}\Lambda(t)(1+E(t)) + c\int_0^T \Lambda(t)\|\nabla\z(t)\|_{L^2(\Omega)}^2\,dt\\[2mm]
\displaystyle\qquad
\le 1+c_0^{q+2} + C_3(q+2)^2\int_0^T \phi(t)\,\Lambda(t)(1+E(t))\,dt.
\end{array}
\end{equation}

Let us define $A(k)$ by
\[
A(k):=\int_0^T\!\int_\Omega \varrho\,|\lambda(t)\w|^{r^k}\,dx\,dt
\]
To apply the result of the following Lemma \ref{lem:moser_iteration}, we must show that the assumption of A(k) in (\ref{iterative_inequality}) is satisfied.

 Denote by $\lambda(t) = \exp(-C_3 \Psi(t))$. So $\Lambda(t)=\lambda(t)^{q+2}$.
Recall that $\z=|\w|^{q/2}\w$. Then
\[
|\lambda \w|^{q+2}=\lambda^{q+2}|\w|^{q+2}=\Lambda|\z|^2 = |\Lambda^{1/2}\z|^2,
\qquad
|\lambda\w|^{\frac{5}{3}(q+2)} = |\Lambda^{1/2}\z|^{\frac{10}{3}}.
\]
By H\"older's inequality,
\begin{align}
\int_0^T\!\int_\Omega \varrho\,|\lambda\w|^{\frac{5}{3}(q+2)}\,dx\,dt
&=\int_0^T\!\int_\Omega \varrho\,|\Lambda^{1/2}\z|^{\frac{10}{3}}\,dx\,dt
=\int_0^T\!\int_\Omega \varrho\,|\Lambda^{1/2}\z|^{\frac{4}{3}}|\Lambda^{1/2}\z|^2\,dx\,dt\nonumber\\
&\le \int_0^T \Big(\int_\Omega (\varrho|\Lambda^{1/2}\z|^{\frac{4}{3}})^{\frac{3}{2}}dx\Big)^{\frac{2}{3}}
\Big(\int_\Omega |\Lambda^{1/2}\z|^6 dx\Big)^{\frac{1}{3}} dt\nonumber\\
&\le \|\varrho\|_{L^\infty(Q_T)}^{\frac13}\int_0^T
\Big(\int_\Omega \varrho|\Lambda^{1/2}\z|^2 dx\Big)^{\frac{2}{3}}
\|\Lambda^{1/2}\z\|_{L^6(\Omega)}^2\,dt.\label{c2}
\end{align}
Using the Sobolev embedding $H^1(\Omega)\hookrightarrow L^6(\Omega)$ and Poincar\'e inequality, we have
\[
\|\Lambda^{1/2}\z\|_{L^6(\Omega)}^2 \le C(\Omega)\|\nabla(\Lambda^{1/2}\z)\|_{L^2(\Omega)}^2.
\]
Since $\Lambda=\Lambda(t)$ depends only on $t$,
\[
\nabla(\Lambda^{1/2}\z)=\Lambda^{1/2}\nabla\z,
\]
which implies
\[
\|\nabla(\Lambda^{1/2}\z)\|_2^2=\Lambda(t)\|\nabla\z\|_2^2.
\]
Therefore, the inequality \eqref{c2} together with the above computations deduces
\begin{align*}
&\int_0^T\!\int_\Omega \varrho\,|\lambda\w|^{\frac{5}{3}(q+2)}\,dx\,dt \\
&\qquad \le C_4\Big[\sup_{t\in[0,T]}\int_\Omega \varrho(t)|\Lambda^{1/2}(t)\z(t)|^2\,dx\Big]^{\frac{2}{3}}
\int_0^T \Lambda(t)\|\nabla\z(t)\|_{L^2(\Omega)}^2\,dt,
\end{align*}
where $C_4=C_4(\|\varrho\|_{L^\infty(Q_T)},\Omega)$.
Invoking \eqref{2.15_gen} and recalling $|\Lambda^{1/2}\z|^2=|\lambda\w|^{q+2}$, we obtain
\begin{equation}\label{general_situation}
\begin{array}{l}
\displaystyle{} \int_0^T\!\int_\Omega \varrho\,|\lambda\w|^{\frac{5}{3}(q+2)}\,dx\,dt \\
\displaystyle{}\qquad \le C_5\Big[1+c_0^{q+2}+C_3(q+2)^2\int_0^T \phi(t)\,\Lambda(t)\big(1+E(t)\big)\,dt\Big]^{\frac{5}{3}},
\end{array}
\end{equation}
with $C_5=C_4/c$.
Let us fix $0<s<1$. We also use H\"older's inequality again and $\Lambda\le 1$ to obtain
\begin{align}
\int_0^T \phi(t)\,\Lambda(t)\big(1+E(t)\big)\,dt
&\le \|\phi\|_{L^{\frac{1}{1-s}}(0,T)}\Big(\int_0^T \big[\Lambda(t)(1+E(t))\big]^{\frac{1}{s}}\,dt\Big)^s \nonumber\\
&\le 2 \|\phi\|_{L^{\frac{1}{1-s}}(0,T)}\,T^s
+ 2 \|\phi\|_{L^{\frac{1}{1-s}}(0,T)}\,\Big(\int_0^T \big[\Lambda(t)E(t)\big]^{\frac{1}{s}}\,dt\Big)^s,\label{c3}
\end{align}
where we have used an elementary inequality
\[
(a+b)^\frac{1}{s}\leq 2^\frac{1}{s}(a^\frac{1}{s}+b^\frac{1}{s}).
\]
Since $\Lambda(t)E(t)= \smallint{\Omega}{} \varrho |\Lambda^{1/2}\z|^2dx=\smallint{\Omega}{} \varrho|\lambda\w|^{q+2}dx$, we further have
\begin{equation}\label{c4}
\int_0^T \big[\Lambda(t)E(t)\big]^{\frac{1}{s}}dt
\le \|\varrho\|_{L^\infty(Q_T)}^{\frac{1-s}{s}}
\int_0^T\!\int_\Omega \varrho\,|\lambda\w|^{\frac{q+2}{s}}\,dx\,dt.
\end{equation}
Substituting \eqref{c3},\eqref{c4} into \eqref{general_situation} yields
\begin{equation}\label{general_Moser_ineq}
\hspace{-6pt}\begin{array}{l}
\displaystyle
\int_0^T\!\int_\Omega \varrho\,|\lambda\w|^{\frac{5}{3}(q+2)}\,dx\,dt
\le 16 C_5\Big[1+c_0^{\frac{5}{3}(q+2)}
+ C_3^{\frac53}(q+2)^{\frac{10}{3}}\|\phi\|_{L^{\frac{1}{1-s}}(0,T)}^{\frac53} T^{\frac53 s} \\[2mm]
\displaystyle\hspace{30pt}
+ C_3^{\frac53}(q+2)^{\frac{10}{3}}
\|\phi\|_{L^{\frac{1}{1-s}}(0,T)}^{\frac53}\,\|\varrho\|_{L^1(\Omega)}^{\frac{5}{3}(1-s)}
\Big(\int_0^T\!\int_\Omega \varrho\,|\lambda\w|^{\frac{q+2}{s}}\,dx\,dt\Big)^{\frac{5s}{3}}
\Big].
\end{array}
\end{equation}

Let $r=\dfrac{5}{3}s$ and choose $q=q(k)$ by $q+2=s r^k$. Then \eqref{general_Moser_ineq} implies
\[
A(k+1)\le C_0^{r^k}+C\,r^{\frac{10}{3}k}\big(1+A(k)^r\big),
\]
where $C_0$ and $C$ depend on the data as follows:
\begin{align*}
C_0 &= C_0(\|\varrho_{{}_0}\|_{L^\infty(\Omega)},\|\u_{{}_0}\|_{L^\infty(\Omega)},\|\uB\|_{W^{1,\infty}(\Omega)},
\|\varrho\|_{L^\infty(Q_T)},\mu,\lambda,\Omega,s),\\
C &= C(\|\varrho\|_{L^\infty(Q_T)},\|\uB\|_{W^{1,\infty}(\Omega)},
\|p/\varrho\|_{L^{\frac{2}{1-s}}(0,T;L^\infty(\Omega))},\mu,\lambda,\Omega,T,s).
\end{align*}
Here we used $\|\phi\|_{L^{\frac{1}{1-s}}(0,T)}=\|p/\varrho\|_{L^{\frac{2}{1-s}}(0,T;L^\infty(\Omega))}^2$.

We may assume  $C>1$. If $s>\dfrac{3}{5}$, then $r>1$ and thus \eqref{iterative_inequality} holds (with suitable $c_1,c_2,k_0$).
Moreover, \eqref{consequence_of_Grownwall} (with $q+2=r^2$) yields $A(2)<\infty$.
Therefore, Lemma~\ref{lem:moser_iteration} gives a constant $C_7>0$ such that
\[
\int_{Q_T}\varrho\,|\lambda\w|^{r^k}\,dx\,dt \le C_7^{\,r^k}\qquad\mathrm{~for~all~}\,k\ge2.
\]

Finally, since $\psi(t)=1+\|\f(t)\|_{L^\infty(\Omega)}\in L^1(0,T)$, we have the uniform lower bound
\[
\lambda(t)=\exp\left(-C_3\int_0^t\psi(s)\,ds\right)\ge \exp\Big(\!-\!C_3\|\psi\|_{L^1(0,T)}\Big)=:\lambda_*>0.
\]
Hence we get $|\w|\le \lambda_*^{-1}|\lambda\w|$ and consequently
\[
\int_{Q_T}\varrho\,|\w|^{r^k}\,dx\,dt
\le \lambda_*^{-r^k}\int_{Q_T}\varrho\,|\lambda\w|^{r^k}\,dx\,dt
\le (\lambda_*^{-1}C_7)^{r^k}=:C_8^{\,r^k}\qquad\mathrm{~for~all~}\,k\ge2,
\]
where $C_8$ depends on the data, $\|\varrho\|_{L^\infty(Q_T)}$,
$\|p/\varrho\|_{L^{\frac{2}{1-s}}(0,T;L^\infty(\Omega))}$,
$\|\f\|_{L^1(0,T;L^\infty(\Omega))}$, $\mu$, $\lambda$, $\Omega$, $T$, and $s$.

For each $q>1$, choose $k=k(q)$ and $\theta=\theta(q)\in(0,1]$ such that $r^{k-1}\le q<r^k$ and
\[
\|\varrho\u\|_{L^q(Q_T)}\le \|\varrho\u\|_{L^{r^{k-1}}(Q_T)}^\theta \|\varrho\u\|_{L^{r^k}(Q_T)}^{1-\theta}.
\]
Using \eqref{Linfty_bound_for_rho} and the bounds above, we obtain $\|\varrho\u\|_{L^q(Q_T)}\le C$ uniformly in $q$.
Passing to the limit $q\to\infty$ gives rise to\vspace{.1cm} the desired $L^\infty$-bound. The final conclusion follows since $r>1$ iff $s>\dfrac{3}{5}$, and
$\dfrac{2}{1-s}>5$ whenever $s>\dfrac{3}{5}$.
\end{proof}

Combining with Lemma \ref{lem:lower_upper_bound_for_rho}, we have the following
\begin{corollary}\label{cor:u_Linfty_bound}
Let $\varrho,\u$ satisfy \eqref{momentum_eq} in the pointwise sense, and $\u=\uB$ on $\partial\Omega$. Assume
\[
\|\varrho\|_{L^\infty(Q_T)}<\infty, \ \ \int_0^T \|\div\u\|_{L^\infty(\Omega)} < \infty,
\]
and for some $\delta>0$,
\[
\|p(\varrho,\vartheta)\varrho^{-1}\|_{L^{5+\delta}(0,T;L^\infty(\Omega))}<\infty\pd
\]
Then
\begin{equation}\label{Linfty_bound_for_u}
\begin{array}{rl}
\|\u\|_{L^\infty(Q_T)} \le C\big(\hspace{-9pt}&\|\varrho_{{}_0}\|_{L^\infty(\Omega)}, \|\u_{{}_0}\|_{L^\infty(\Omega)},
\|\uB\|_{H^{2.5}(\partial\Omega)}, \|p/\varrho\|_{L^{5+\delta}(0,T;L^\infty(\Omega))}, \vspace{.1cm}\\
&\|\div\u\|_{L^1(0,T;L^\infty(\Omega))}, \rou,\rBu,\rol,\rBl,\Omega,T,\delta\big).
\end{array}
\end{equation}
\end{corollary}

For later use in the local existence argument, we introduce the Banach space
\begin{align*}
X_{T_*} := \Bigl\{ w \,\Big|\,& \partial_t^k w \in L^2(0,T_*;H^{3-2k}(\Omega)) \cap L^\infty(0,T_*;H^{2-2k}(\Omega)) \text{ for } k=0,1, \vspace{.1cm}\\
& w_{tt}\in L^2(0,T_*;H^{-1}(\Omega)) \Bigr\},
\end{align*}
equipped with the norm
\begin{align*}
\|w\|_{X_{T_*}}^2
& :=
\sup_{t\in[0,T_*]}
\Bigl(
\|w(t)\|_{H^2(\Omega)}^2+\|w_t(t)\|_{L^2(\Omega)}^2
\Bigr) \\
&\quad
+ \int_0^{T_*}
\Bigl(
\|w\|_{H^3(\Omega)}^2
+
\|w_t\|_{H^1(\Omega)}^2
+
\|w_{tt}\|_{H^{-1}(\Omega)}^2
\Bigr)\,dt .
\end{align*}
Here $w$ may be either scalar-valued or vector-valued.

The following result applies to both scalar-valued and vector-valued unknowns; in the
latter case, the equation and the variational identity are understood componentwise.

\begin{theorem}[Linear parabolic regularity]
\label{thm:linear_parabolic_blackbox}
Let $\Omega$ be a $C^3$-domain, let $0<T^\ast\le T$, and let $\nu>0$. Assume that
\[
\partial_t^k \widebar{\varrho} \in L^\infty(0,T^\ast;H^{2-k}(\Omega)),
\qquad k=0,1,2,
~
\widebar{\varrho}(x,t)\ge c_0>0
~\text{in }\Omega\times(0,T^\ast),
\]
\[
\widebar{\u},\ \sbV \in X_{T^\ast},
~
\widebar{a}\in L^\infty(0,T^\ast;H^1(\Omega)),
~
\partial_t \widebar{a}\in L^2(0,T^\ast;L^2(\Omega)),
\]
and
\[
\sbF\in L^\infty(0,T^\ast;L^2(\Omega))
\cap L^2(0,T^\ast;H^1(\Omega)),
~
\sbF_t\in L^2(0,T^\ast;H^{-1}(\Omega)).
\]
Suppose moreover that
\[
\v_0\in H^2(\Omega),
\sbV(\cdot,0)=:\sbV_0\in H^2(\Omega),~
\sbV_t(\cdot,0)=:\sbV_1\in L^2(\Omega),~
\sbF(\cdot,0)=:\sbF_0\in L^2(\Omega),
\]
and that the compatibility condition
\[
\v_0=\sbV_0
\quad\text{on }\partial\Omega
\]
holds. Consider the linear parabolic problem
\begin{equation}
\label{eq:linear_parabolic_problem}
\left\{
\begin{aligned}
\widebar{\varrho}\bigl(\v_t+\widebar{\u}\cdot\nabla \v\bigr) - \nu \Delta \v
    &= \widebar{\varrho}\hspace{1pt} \widebar{a} \v + \sbF
    &&\text{in }\Omega\times(0,T^\ast),\\
\v &= \v_0
    &&\text{on }\Omega\times\{t=0\},\\
\v &= \sbV
    &&\text{on }\partial\Omega\times(0,T^\ast).
\end{aligned}
\right.
\end{equation}
Then there exists a unique strong solution $\v\in X_{T^\ast}$ to
\eqref{eq:linear_parabolic_problem}, and
\begin{align}
&\|\v\|_{X_{T^\ast}}^2
\le C \Biggl[
\|\v_0\|_{H^2(\Omega)}^2
\hspace{-1pt}+\hspace{-1pt} \|\sbV_0\|_{H^2(\Omega)}^2
\hspace{-1pt}+\hspace{-1pt} \|\sbV_1\|_{L^2(\Omega)}^2
\hspace{-1pt}+\hspace{-1pt} \|\sbF_0\|_{L^2(\Omega)}^2
\hspace{-1pt}+\hspace{-1pt} \|\sbF\|_{L^\infty(0,T^\ast;L^2(\Omega))}^2
\notag\\
&\quad
\hspace{-1pt}+\hspace{-1pt} \int_0^{T^\ast}
\Bigl(
\|\sbF\|_{H^1(\Omega)}^2
\hspace{-1pt}+\hspace{-1pt} \|\sbF_t\|_{H^{-1}(\Omega)}^2
\hspace{-1pt}+\hspace{-1pt} \|\sbV\|_{H^3(\Omega)}^2
\hspace{-1pt}+\hspace{-1pt} \|\sbV_t\|_{H^1(\Omega)}^2
\hspace{-1pt}+\hspace{-1pt} \|\sbV_{tt}\|_{H^{-1}(\Omega)}^2
\Bigr)\,dt
\Biggr],
\label{eq:linear_parabolic_estimate}
\end{align}
where the constant $C$ depends only on
\[
\nu,\ c_0,\ T^\ast,\ \Omega,\
\|\widebar{\varrho}\|_{L^\infty(0,T^\ast;H^2(\Omega))},\
\|(\widebar{\varrho})_t\|_{L^\infty(0,T^\ast;H^1(\Omega))},
\]
\[
\|\widebar{\u}\|_{X_{T^\ast}},\
\|\widebar{a}\|_{L^\infty(0,T^\ast;H^1(\Omega))},\
\|\partial_t \widebar{a}\|_{L^2(0,T^\ast;L^2(\Omega))}.
\]
Moreover, for a.e. $t\in(0,T^\ast)$ it holds that for all $\phi\in H_0^1(\Omega)$,
\begin{align}
\label{eq:linear_parabolic_weak}
\langle \v_{tt},\widebar{\varrho}\hspace{1pt}\phi\rangle
+ \int_\Omega \widebar{\varrho}_t \v_t \phi\,dx
&+ \int_\Omega \partial_t\!\bigl(\widebar{\varrho} \widebar{\u}\cdot\nabla \v\bigr)\,\phi\,dx
+ \nu \int_\Omega \nabla \v_t \cdot \nabla\phi\,dx\nonumber\\
&=
\int_\Omega \partial_t(\widebar{\varrho}\hspace{1pt} \widebar{a} \v)\,\phi\,dx
+ \langle \sbF_t,\phi\rangle\,,
\end{align}
where the variational identity is understood componentwise in the vector-valued case.
\end{theorem}

\begin{remark}[Extensions of the linear regularity result]
\label{rem:linear_parabolic_extensions}
For our current purposes, the statement of Theorem ~\ref{thm:linear_parabolic_blackbox} suffices. Since its proof relies on a separate parabolic regularity theory, it is omitted here. We also note that, with appropriate modifications to the weak formulation, compatibility conditions, and underlying elliptic regularity, the same argument extends to more general parabolic operators of the form
\[
\nu_1 \Delta v + \nu_2 \nabla \div v,
\qquad
\nu_1>0,\quad \nu_2\ge 0,
\]
as well as to Navier-type boundary conditions. Since these extensions are not needed in the subsequent proofs of the present paper, we do not state them here in full generality.
\end{remark}

\section{Proof of Theorem~\ref{thm:main}: the well-posedness of the strong solution}
In this section, we prove the local existence part of Theorem~\ref{thm:main}. The linear parabolic regularity needed in the construction has already been summarized in Theorem~\ref{thm:linear_parabolic_blackbox}. We therefore only present the estimates that are specific to the fixed-point scheme for \eqref{general_equations}.

\subsection{A boundary-adapted estimate for the linearized system}
The fixed-point argument below uses a material-derivative estimate for a linear
parabolic equation with nonhomogeneous boundary data.  We state the estimate in a
form adapted to the momentum equation; the algebraic derivation is collected in
Appendix~\ref{app:linearized_material_derivative_estimate}.

\begin{lemma}[Boundary-adapted material-derivative estimate]
\label{lem:linearized_material_derivative_estimate}
Let $\v\in X_T$ solve
\begin{subequations}\label{parabolic_eq_w_inhom_bdy_condition_with_f_statement}
\begin{alignat}{2}
\brho \bD_t \v &= \nu \Delta \v + \sbF \qquad&&\text{in}\quad\Omega\times (0,T),
\label{parabolic_eq_w_inhom_bdy_condition_with_f_statement.1}\\
\v &= \v_0 &&\text{on}\quad\Omega \times \{t=0\},\\
\v &= \sbV &&\text{on}\quad\bdy\Omega \times (0,T),
\end{alignat}
\end{subequations} with $\bD_t \v=\partial_t\v+\overline{\u}\cdot\nabla\v$
where the coefficients and data have the regularity described in
Theorem~\ref{thm:linear_parabolic_blackbox}.  Assume in addition that
\[
\brho_t+\div(\brho\bu)=0,
\qquad
\sbF=-\nabla\bp+\brho\f,
\]
with \[\bp\in L^2(0,T;H^2(\Omega)), \bp_t\in L^2(0,T;H^1(\Omega)),
\f\in L^2(0,T;H^1(\Omega)),\f_t\in L^2(0,T;H^{-1}(\Omega)).\]
Let $\bfu,\bfv,\bfw$ be chosen so that
\[
\w=\bD_t\v-\bfu-\bfv\cdot\nabla\bfw
\in L^2(0,T;H^1_0(\Omega)),
~
\w_t\in L^2(0,T;H^{-1}(\Omega)).
\]
Then the following estimate  holds for a.e. $t\in(0,T)$.
\begin{align}
& \frac{1}{2} \frac{d}{dt} \int_\Omega \brho |\w|^2 \dx + \frac{\nu}{2} \|\nabla \w\|^2_{L^2(\Omega)} \le C \Big[\big(\|\div \bu\|^2_{L^4(\Omega)} + \|\Def \bu\|^2_{L^4(\Omega)}\big) \|\nabla \v\hpt\|^2_{L^4(\Omega)}  \nonumber\\
& + \|\bfu\|^2_{H^1(\Omega)} +\big(\|\nabla \bfv\|^2_{L^3(\Omega)}+\|\bfv\|^2_{L^\infty(\Omega)}\big) \|\bfw\|^2_{H^2(\Omega)} + \|\brho\|^2_{L^\infty(\Omega)} \|\bD_t \bfv\|^2_{L^2(\Omega)} \|\nabla \bfw\|^2_{L^3(\Omega)} \nonumber\\
& + \big(\|\nabla \brho\|^2_{L^3(\Omega)} + \|\brho\|^2_{L^\infty(\Omega)}\big) \big(\|\bD_t \bfu\|^2_{H^{-1}(\Omega)} + \|\f_t\|^2_{H^{-1}(\Omega)} \,+ \|\f\|^2_{H^1(\Omega)} \big) \nonumber\\
& + \big\|\bp_t + \div (\bp\bu)\big\|^2_{L^2(\Omega)} + \|\bp\|^2_{L^\infty(\Omega)} \|\nabla \u\|^2_{L^2(\Omega)} + \|\brho\|^2_{L^\infty(\Omega)} \|\bu\|^2_{L^\infty(\Omega)} \|\w\|^2_{L^2(\Omega)} \Big] \nonumber\\
&\qquad\qquad - \int_\Omega (\bfv \cdot \nabla \bfw_t) \cdot \w \dx \pd\label{key_estimate_for_u}
\end{align}
\end{lemma}

The point of the Lemma~\ref{lem:linearized_material_derivative_estimate} is that all boundary contributions are absorbed by the
choice of $\w$.  In the local construction we apply it with
$(\bfu,\bfv,\bfw)=(\sbV_t,\sbV,\u)$, so that
$\w=\bD_t^s\u-\sbV_t-\sbV\cdot\nabla\u$ vanishes on $\partial\Omega$.

\subsection{Construction of the strong solution}
\subsubsection{The fixed-point scheme}
We apply a fixed-point scheme to construct a strong solution to \eqref{general_equations}.
For $T^* \in (0,T]$ and $M > 0$, let $C_{T^*}(M)$ be a subset of $X_{T^*}$ given by
$$
\begin{array}{rl}
C_{T^*}(M) = \Big\{\w \in X_{T^*} \,\Big|\hspace{-7pt}& \|\w\|_{X_{T^*}} \le M, \w|_{t=0} = \u_{{}_0} , \w_t|_{t=0} = \u_1 \text{ in }\Omega \cm\vspace{.1cm}\\
& \w = \sbV \text{ on }\bdy\Omega \times (0,T^*)\Big\}\pd
\end{array}
$$
where
\begin{equation}\label{defn:u1}
\u_1 \equiv \frac{1}{\varrho_{{}_0}} \Big[\Delta \u_{{}_0} + \varrho_{{}_0} \f(0) - \u_{{}_0} \cdot \nabla \u_{{}_0} - \nabla (\varrho_{{}_0} \vartheta_0)\Big].
\end{equation}
For a given $\bu \in C_{T^*}(M)$ with $T^*$ and $M$ described later (usually $M$ is chosen first, and $T^*$ is chosen according to $M$), let $\brho$ be the solution to
\begin{subequations}\label{linear_density_eq}
\begin{alignat}{2}
\partial_t \brho + \div(\brho \bu) &= 0 &&\text{in}\quad \Omega \times (0,T^*)\cm\label{linear_density_eq.1}\\
\brho &= \varrho_{{}_0} \qquad&&\text{on}\quad \Omega \times \{t=0\}\cm\\
\brho &= \rB &&\text{on}\quad \Gin \times (0,T^*)\pd
\end{alignat}
\end{subequations}
Note that the solution $\brho$ to the equation above has to be obtained using method of characteristics due to the inflow boundary condition.

Having $\brho$ and $\bu$, let $\btheta$ be the solution to the linear temperature equation
\begin{subequations}\label{linear_temperature_eq}
\begin{alignat}{2}
\brho (\partial_t \btheta + \bu \cdot \nabla \btheta) - \kappa \Delta \btheta &= |\nabla \bu|^2 - \brho \btheta \div \bu + \brho g \qquad&&\text{in}\quad \Omega \times (0,T^*)\cm\\
\btheta &= \vartheta_0 \qquad&&\text{on}\quad \Omega \times \{t=0\}\cm\\
\btheta &= \tB &&\text{on}\quad \bdy\Omega \times (0,T^*)\pd
\end{alignat}
\end{subequations}
In other words, even we are given $\bu$ only, we indeed have a pair $(\brho,\bu,\btheta)$.

Suppose that $g \in L^2(0,T;H^1(\Omega))$ with $g_t \in L^2(0,T;H^{-1}(\Omega))$ and $\vartheta_0 \in H^2(\Omega)$ satisfies the compatibility condition
\begin{equation}\label{compatibility_condition2}
\vartheta_0 = \tB \qquad\text{on}\quad \bdy\Omega\pd
\end{equation}

The solution $\btheta$ to \eqref{linear_temperature_eq} belongs to $X_{T^*}$ by Theorem~\ref{thm:linear_parabolic_blackbox},
applied with
\[
a_s=-\div \bu,\qquad
F=|\nabla \bu|^2+\brho g,\qquad
V=\tB .
\]

Define $\bp = \brho \btheta$. Let $\u$ be the solution to the following linear parabolic equation
\begin{subequations}\label{linear_momentum_eq}
\begin{alignat}{2}
\brho (\u_t + \bu \cdot \nabla \u) + \nabla \bp &= \Delta \u + \brho \f \qquad&&\text{in}\quad \Omega \times (0,T^*)\cm \label{linear_momentum_eq.1}\\
\u &= \u_{{}_0} \qquad&&\text{on}\quad \Omega \times \{t=0\}\cm \label{linear_momentum_eq.2}\\
\u &= \sbV &&\text{on}\quad \partial\Omega \times (0,T^*)\pd \label{linear_momentum_eq.3}
\end{alignat}
\end{subequations}
We then establish a map
\[
\Phi:C_{T^*}(M)\to X_{T^*}, \qquad \Phi(\bu)=\u.
\]
The goal is to show that, under suitable choices of $T^*$ and $M$,
\[
\Phi:C_{T^*}(M)\to C_{T^*}(M)
\]
and that $\Phi$ admits a fixed point.

\subsubsection{The self-map estimate and the existence of fixed points}
The remaining step is to show that the map $\Phi$ constructed above is a self-map on
$C_{T^*}(M)$ for a suitable choice of $M$ and $T^*$.  We isolate the required
estimates in the following proposition.  Its proof is given in
Appendix~\ref{app:fixed_point_self_map_estimate}.

\begin{proposition}[Self-map estimate for the fixed-point scheme]
\label{prop:fixed_point_self_map_estimate}
Let $M>0$ and let $\bu\in C_{T^*}(M)$.  Let $\brho$, $\btheta$, and
$\u=\Phi(\bu)$ be defined by \eqref{linear_density_eq},
\eqref{linear_temperature_eq}, and \eqref{linear_momentum_eq}, respectively.
Then, after choosing $T^*=T^*(M)>0$ sufficiently small, the following hold:
\begin{enumerate}
\item There exist constants $0<c_0<c_1$, depending only on the initial and
inflow boundary density bounds, such that
\[
 c_0\le \brho(x,t)\le c_1 \qquad \text{in }\Omega\times(0,T^*).
\]
Moreover,
\[
 \brho\in L^\infty(0,T^*;H^2(\Omega)),
 \qquad
 \brho_t\in L^\infty(0,T^*;H^1(\Omega)).
\]
\item There exists a nondecreasing function $F=F(M)$ such that
\[
 \|\btheta\|_{X_{T^*}}^2\le F(M).
\]
\item If $T^*$ is further chosen so that
\[
 T^*\bigl(M+F(M)\bigr)\le 1,
\]
then there is a constant $C_*$, depending only on the data and independent of
$M$, such that
\[
 \|\Phi(\bu)\|_{X_{T^*}}=\|\u\|_{X_{T^*}}\le C_*.
\]
\end{enumerate}
\end{proposition}

Assuming Proposition~\ref{prop:fixed_point_self_map_estimate}, choose $M>0$ so
large that $C_*\le M$.  Then choose $T^*>0$ sufficiently small so that all the
smallness requirements in the proposition are satisfied.  Since the solution
$\u=\Phi(\bu)$ has the same initial value, initial time derivative, and boundary
value prescribed in the definition of $C_{T^*}(M)$, the estimate
$\|\u\|_{X_{T^*}}\le C_*\le M$ gives
\[
 \Phi:C_{T^*}(M)\to C_{T^*}(M).
\]

We endow $C_{T^*}(M)$ with the product weak/weak-$*$ topology induced by the
ambient spaces appearing in the definition of $X_{T^*}$.  The bounds above,
together with the Banach--Alaoglu theorem, imply that $C_{T^*}(M)$ is nonempty,
bounded, closed, convex, and compact in this topology.  The weak stability of the
transport equation and of the two linear parabolic problems defining $\brho$,
$\btheta$, and $\u=\Phi(\bu)$ implies that $\Phi$ is continuous with respect to
this topology.  Therefore, by the Schauder--Tychonoff fixed-point theorem,
$\Phi$ admits a fixed point $\bu\in C_{T^*}(M)$.  The corresponding triple
$(\brho,\bu,\btheta)$ is a strong solution to \eqref{general_equations} on
$\Omega\times(0,T^*)$.

\section{Proof of Theorem~\ref{thm:main2}: the blow-up criterion}
In this section, we prove the continuation criterion by deriving a priori estimates under the assumption that the quantities appearing in the blow-up condition remain finite on a time interval $(0,T)$. The argument proceeds in a hierarchical way. We first establish lower-order energy bounds for $\u$, $\vartheta$, and $\varrho$, then bootstrap these bounds to higher spatial regularity, and finally close the highest-order estimates needed to continue the strong solution beyond time $T$.
The key point is that, under the continuation assumption, all norms required by the local existence theory remain bounded up to time $T$.
\subsection{Blow-up criteria}

In this subsection, we introduce the a priori assumption used in the continuation argument. Recall that $\varrho_{{}_0}$ and $\varrho_B$ are assumed to be strictly positive. Moreover, we assume that
\begin{equation}\label{L1Linfty_bound_for_grad_u}
\int_0^T \Big[\|\nabla \u\|_{L^\infty(\Omega)} + \|\vartheta\|_{L^\infty(\Omega)}^{5+\delta} \dt \Big]\dt < \infty \quad\text{for some $\delta > 0$\pd}
\end{equation}
We note that, since the density satisfies
\[
\inf_{x \in \Omega} \varrho(x,t) > 0
\qquad \text{for all } t \in (0,T),
\]
Proposition \ref{prop:2.2}, together with Lemma \ref{lem:lower_upper_bound_for_rho} under assumption \eqref{L1Linfty_bound_for_grad_u}, implies that
\[
\|\u\|_{L^\infty(Q_T)} < \infty.
\]
\subsection{Basic energy estimates}
We begin with the basic energy level. The purpose of this subsection is to obtain the first uniform bounds for the gradients of $\u$ and $\vartheta$, together with a corresponding control of the density. These estimates provide the starting point for the whole continuation argument. In the inflow-outflow setting, a special difficulty comes from the boundary contributions generated by the continuity equation; therefore, besides the standard energy estimates for $\u$ and $\vartheta$, we must also keep track of the boundary terms arising in the density estimates.

\subsubsection{\texorpdfstring{$L^2_t H^1_x$}{L²ₜH¹ₓ} and \texorpdfstring{$L^\infty_t L^2_x$}{L²ₜL²ₓ} estimates for $\u$ and \texorpdfstring{$\vartheta$}{θ}}
We first test the momentum and temperature equations against $\u-\uB$ and
$\vartheta-\tB$, respectively. This choice is adapted to the nonhomogeneous Dirichlet boundary condition and avoids unwanted boundary contributions in the basic energy identity. At this level, the goal is not yet to close the highest-order norms, but rather to obtain the first coercive bounds that will later be combined with elliptic regularity and transport estimates.

Testing \eqref{momentum_eq} and \eqref{temperature_eq} against $\u - \uB$ and $\vartheta - \tB$ respectively, we immediately obtain the following
\begin{subequations}\label{collection_of_basic_estimates}
\begin{align}
&\frac{1}{2} \frac{d}{dt} \int_\Omega \varrho |\u \hspace{-.5pt}-\hspace{-.5pt} \uB|^2 \dx \hspace{-.5pt}+\hspace{-.5pt} \int_\Omega |\nabla \u|^2 \dx \nonumber\\
&\qquad\qquad\le C \Big[1 \hspace{-.5pt}+\hspace{-.5pt} \int_\Omega \varrho |\f|^2 \dx \hspace{-.5pt}+\hspace{-.5pt} \int_\Omega \varrho \big(|\u \hspace{-.5pt}-\hspace{-.5pt} \uB|^2 \hspace{-.5pt}+\hspace{-.5pt} |\vartheta \hspace{-.5pt}-\hspace{-.5pt} \tB|^2\big) \dx \Big]. \label{collection_of_basic_estimates.2}\\
&\frac{1}{2} \frac{d}{dt} \int_\Omega \varrho |\vartheta \hspace{-.5pt}-\hspace{-.5pt} \tB|^2 \dx \hspace{-.5pt}+\hspace{-.5pt} \kappa \int_\Omega |\nabla \vartheta|^2 \dx \nonumber\\&\qquad\qquad\le C \int_\Omega \big(1 \hspace{-.5pt}+\hspace{-.5pt} \|\nabla \u\|_{L^\infty(\Omega)} \big) \Big[1 \hspace{-.5pt}+\hspace{-.5pt} |\nabla \u|^2 \hspace{-.5pt}+\hspace{-.5pt} \varrho |\vartheta \hspace{-.5pt}-\hspace{-.5pt} \tB|^2 \Big] \dx\pd \label{collection_of_basic_estimate.4}
\end{align}
\end{subequations}

\subsubsection{\texorpdfstring{$L^\infty_t H^1_x$}{LᵒᵒₜH¹ₓ}-estimates for $\u$ and \texorpdfstring{$\vartheta$}{θ} and \texorpdfstring{$L^2_t L^2_x$}{L²ₜL²ₓ}-estimates for \texorpdfstring{$\rD_t \u$}{Dₜu} and \texorpdfstring{$\rD_t \vartheta$}{Dₜθ}}
The next step is to work with the material derivatives $D_t\u$ and $D_t\vartheta$, which are the natural quantities associated with the transport structure of the system. The specific test functions
\[
D_t\u-\uB\cdot\nabla\u
\qquad\text{and}\qquad
D_t\vartheta-\uB\cdot\nabla\vartheta
\]
are chosen so that the boundary values remain compatible with the inhomogeneous boundary condition. This gives a gain of one time derivative at the energy level and prepares the ground for the higher-order elliptic estimates used later.

Testing the temperature equation \eqref{temperature_eq} against
$\rD_t \vartheta - \uB \cdot  \nabla \vartheta \equiv \vartheta_t + (\u - \uB) \cdot \nabla \vartheta$, by the fact that $\vartheta_t + (\u - \uB) \cdot \nabla \vartheta = 0$ on $\partial\Omega$,~
we find that
\begin{align}
& \|\sqrt{\varrho}\, \rD_t \vartheta\|^2_{L^2(\Omega)} + \kappa \int_\Omega \nabla \vartheta \cdot \nabla \big[\vartheta_t + (\u - \uB) \cdot \nabla \vartheta\big] \dx \nonumber\\
&\qquad = \int_\Omega \varrho \rD_t \vartheta (\uB \cdot \nabla \vartheta) \dx + \int_\Omega |\nabla \u|^2 \rD_t \vartheta \dx - \int_\Omega \varrho\hpt\vartheta \div \u \rD_t \vartheta \dx + \int_\Omega \varrho g \rD_t \vartheta \dx \nonumber\\
&\qquad \le C_\varepsilon \Big[\|\sqrt{\varrho}\, \nabla \vartheta\|^2_{L^2(\Omega)} + \|\nabla \u\|^4_{L^4(\Omega)} + \|\vartheta \div \u\|^2_{L^2(\Omega)} + \|g\|^2_{L^2(\Omega)}\Big]  \nonumber\\&\qquad\qquad\qquad\qquad+ \varepsilon \|\sqrt{\varrho}\, \rD_t \vartheta\|^2_{L^2(\Omega)}\pd \label{collection_of_basic_estimate.5_temp}
\end{align}
Since
\begin{align*}
\int_\Omega \nabla \vartheta \cdot &\nabla \big[(\u - \uB) \cdot \nabla \vartheta\big] \dx = \int_\Omega (\nabla \vartheta \otimes \nabla \vartheta) : \nabla (\u - \uB) \dx\\& + \int_\Omega (\u - \uB) \cdot \nabla \Big(\frac{1}{2} \big|\nabla \vartheta\big|^2\Big) \dx \\
&= \int_\Omega (\nabla \vartheta \otimes \nabla \vartheta) : \nabla (\u - \uB) \dx - \frac{1}{2} \int_\Omega \div (\u - \uB) \big|\nabla \vartheta\big|^2 \dx\cm
\end{align*}
we have
$$
\left|\int_\Omega \nabla \vartheta \cdot \nabla \big[(\u - \uB) \cdot \nabla \vartheta\big] \dx \right| \le \frac{3}{2} \big(\|\nabla \uB\|_{L^\infty(\Omega)} + \|\nabla \u\|_{L^\infty(\Omega)}\big) \|\nabla \vartheta\|^2_{L^2(\Omega)}\cm
$$
so by choosing $\varepsilon > 0$ sufficiently small in \eqref{collection_of_basic_estimate.5_temp} then integrating in time,  we conclude that
\begin{align}\label{collection_of_basic_estimate.5}
 \|\nabla \vartheta(t)\|^2_{L^2(\Omega)} + \int_0^t \|\sqrt{\varrho}\, \rD_t \vartheta\|^2_{L^2(\Omega)} \dt' \preceq \|\nabla \vartheta_0\|^2_{L^2(\Omega)} + \|g\|^2_{L^2(Q_T)} + \int_0^t \|\nabla \u\|^4_{L^4(\Omega)} dt' \vspace{.15cm}\nonumber\\
  + \int_0^t \big(1 + \|\nabla \u\|_{L^\infty(\Omega)} + \|\vartheta\|^2_{L^\infty(\Omega)} \big) \big(\|\nabla \vartheta\|^2_{L^2(\Omega)} + \|\nabla \u\|^2_{L^2(\Omega)}\big) \dt',
\end{align}
where $C = C(\|\uB\|_{W^{1,\infty}(\Omega)},\kappa)$ and $A\preceq B$ means $A\leq C B$ for a constant number $C>0.$

Testing \eqref{momentum_eq} against $\rD_t \u - \uB\cdot \nabla \u$ (which vanishes on $\partial\Omega$), we obtain that
\begin{align}
& \int_\Omega \varrho |\rD_t \u|^2 dx + \frac{1}{2}\frac{d}{dt} \int_\Omega |\nabla \u|^2 \dx\nonumber\\
&\qquad\qquad = \int_\Omega \varrho \rD_t \u \cdot (\uB \cdot \nabla \u) \dx - \int_\Omega \nabla p \cdot (\rD_t \u - \uB \cdot \nabla \u) dx \nonumber\\
&\qquad\qquad + \int_\Omega \Delta \u \cdot \big[(\u - \uB) \cdot \nabla \u\big] \dx + \int_\Omega \varrho \f \cdot \big[\rD_t \u - \uB \cdot \nabla \u\big] \dx \nonumber\\
&\qquad\qquad \preceq \frac{1}{2} \int_\Omega \varrho |\rD_t \u|^2 \dx +  \int_\Omega \Big[|\nabla \u|^2 + \varrho |\f|^2 \Big] \dx - \int_\Omega \nabla p \cdot (\rD_t \u - \uB \cdot \nabla \u) dx \label{LinftyL2_estimate_for_u_temp}\\
&\qquad\qquad + \int_\Omega \Delta \u \cdot \big[(\u - \uB) \cdot \nabla \u\big] \dx \pd \nonumber
\end{align}
Note that
\begin{align*}
& - \int_\Omega \nabla p \cdot \big[\rD_t \u - \uB \cdot \nabla \u\big] \dx = \int_\Omega p\,\div \big[\rD_t \u - \uB \cdot \nabla \u\big] \dx \\
&\qquad = \frac{d}{dt} \int_\Omega p\,\div \u \dx - \int_\Omega p_t \div \u \dx + \int_\Omega p\,\div \big[(\u-\uB) \cdot \nabla \u\big] \dx \\
&\qquad = \frac{d}{dt} \int_\Omega p\,\div \u \dx - \int_\Omega \big[p_t + \div(p\u)\big] \div \u \dx + \int_\Omega \div (p\u) \div \u \dx \\
&\qquad\quad + \int_\Omega p\,\div \big[(\u-\uB) \cdot \nabla \u\big] \dx \cm
\end{align*}
where the last integral is further computed as follows:
\begin{align*}
& \int_\Omega p\,\div \big[(\u-\uB)\cdot \nabla \u\big] \dx
= \int_\Omega p\, \big[(\u - \uB)^k,_i \u^i,_k + (\u - \uB)^k \u^i,_{ki} \big] \dx \\
&\qquad = \int_\Omega \varrho (\vartheta - \tB) (\u - \uB)^k,_i \u^i,_k \dx + \int_\Omega \varrho\,\tB (\u - \uB)^k,_i \u^i,_k \dx \\
&\qquad\quad - \int_\Omega \big[(\u - \uB)\cdot \nabla p + p\,\div (\u - \uB) \big] \div \u \dx\pd
\end{align*}
Therefore, the Poincar\'e inequality $\|\vartheta-\tB\|_{L^2(\Omega)} \le C \|\nabla \vartheta\|_{L^2(\Omega)}$ implies that
\begin{align*}
 -& \int_\Omega \nabla p \cdot \big[\rD_t \u - \uB\cdot \nabla \u\big] \dx \\
&\quad = \frac{d}{dt} \!\int_\Omega p\,\div \u \dx -\! \int_\Omega \big[|\nabla \u|^2 \!- p\,\div \u+ \kappa \Delta (\vartheta - \tB) \big] \div \u \dx \\
&\quad\quad +\! \int_\Omega \varrho (\vartheta - \tB) (\u - \uB)^k,_i \u^i,_k dx  \\
&\quad\quad + \int_\Omega \varrho\,\tB (\u - \uB)^k,_i \u^i,_k \dx + \int_\Omega \big[ \uB \cdot \nabla p + p\,\div \uB) \big] \div \u \dx \\
&\quad \preceq  \frac{d}{dt} \int_\Omega p\,\div \u \dx+ \Big[ \big(1+\|\nabla \u\|_{L^\infty(\Omega)} \big) \big(1 + \|\nabla \u\|^2_{L^2(\Omega)} \!+ \|\nabla \vartheta\|^2_{L^2(\Omega)} \!+ \|\nabla \varrho\|^2_{L^2(\Omega)}\big)\\
&\quad\quad + \|\vartheta\|^2_{H^2(\Omega)}\Big].
\end{align*}
Moreover,
\begin{align*}
& \int_\Omega \Delta \u \cdot \big[(\u - \uB) \cdot \nabla \u\big] \dx
= - \int_\Omega \u^i,_k \big[(\u - \uB)^j,_k \u^i,_j + (\u - \uB)^j \u^i,_{jk} \big] \dx \\
&\qquad = - \int_\Omega \u^i,_k (\u - \uB)^j,_k \u^i,_j \dx - \frac{1}{2} \int_\Omega  (\u - \uB)^j (\u^i,_k \u^i,_k),_j \dx \\
&\qquad = - \int_\Omega \u^i,_k (\u - \uB)^j,_k \u^i,_j \dx + \frac{1}{2} \int_\Omega \div (\u - \uB) |\nabla \u|^2 \dx \\
&\qquad \preceq\big(1 + \|\nabla \u\|_{L^\infty(\Omega)}\big) \|\nabla \u\|^2_{L^2(\Omega)} \cm
\end{align*}
so from \eqref{LinftyL2_estimate_for_u_temp} we have
\begin{align*}
& \frac{1}{2} \int_\Omega \varrho |\rD_t \u|^2 \dx + \frac{1}{2} \frac{d}{dt} \int_\Omega |\nabla \u|^2\dx \preceq \frac{d}{dt} \int_\Omega p\,\div \u \dx + \Big[\|\f\|^2_{L^2(Q_T)} + \|\vartheta\|^2_{H^2(\Omega)} \\
&\quad\quad\ + \big(1+\|\nabla \u\|_{L^\infty(\Omega)} \big) \big(1 + \|\nabla \u\|^2_{L^2(\Omega)} + \|\nabla \vartheta\|^2_{L^2(\Omega)} + \|\nabla \varrho\|^2_{L^2(\Omega)}\big) \Big].
\end{align*}
where the constant depends on $\|\varrho\|_{L^\infty(Q_T)}$, $\|\uB\|_{W^{1,\infty}(\Omega)}$, $\|\tB\|_{W^{1,\infty}(\Omega)}$ and $\Omega$. Integrating in time, we obtain that for $0\le t \le T$,
\begin{align}
& \|\nabla \u(t)\|^2_{L^2(\Omega)} + \int_0^t \|\sqrt{\varrho}\, \rD_t \u\|^2_{L^2(\Omega)} \dt'\preceq\|\nabla \u_{{}_0}\|^2_{L^2(\Omega)} \nonumber\\
&\qquad  +  \int_\Omega p(t) \div \u(t)\dx -  \int_\Omega \varrho_{{}_0} \vartheta_0 \div \u_{{}_0} \dx + \left[\|\f\|^2_{L^2(Q_T)} +\int_0^t \|\vartheta\|^2_{H^2(\Omega)} \dt'\right. \label{LinftyL2_estimate_for_grad_u_temp2}\\
&\qquad\quad +\left.   \int_0^t \big(1+\|\nabla \u\|_{L^\infty(\Omega)} \big) \big(1 + \|\nabla \u\|^2_{L^2(\Omega)} + \|\nabla \vartheta\|^2_{L^2(\Omega)} + \|\nabla \varrho\|^2_{L^2(\Omega)}\big) dt' \right]. \nonumber
\end{align}
Since
\begin{align*}
\int_\Omega p\,\div \u \dx = \int_\Omega \varrho (\vartheta - \tB) \div \u \dx& + \int_\Omega \varrho \tB \div \u \dx \\
&\preceq \Big[1 + \int_\Omega \varrho |\vartheta - \tB|^2 \dx\Big] +\|\nabla \u\|^2_{L^2(\Omega)}
\end{align*}
for some constant depending on $\|\varrho\|_{L^\infty(Q_T)}$, we conclude from \eqref{collection_of_basic_estimate.4} and \eqref{LinftyL2_estimate_for_grad_u_temp2} that
\begin{align}
& \|\nabla \u(t)\|^2_{L^2(\Omega)} + \int_0^t \|\sqrt{\varrho}\, \rD_t \u\|^2_{L^2(\Omega)} \dt' \nonumber\\
&\qquad \preceq \left[\|\u_{{}_0}\|^2_{H^1(\Omega)} + \|\vartheta_0\|^2_{L^\infty(\Omega)} + \|\f\|^2_{L^2(Q_T)} + \int_0^t \|\vartheta\|^2_{H^2(\Omega)} \dt'  \right. \label{LinftyL2_estimate_for_grad_u_temp3}\\
&\qquad\qquad + \left.\int_0^t \big(1+\|\nabla \u\|_{L^\infty(\Omega)} \big) \big(1 + \|\nabla \u\|^2_{L^2(\Omega)} + \|\nabla \vartheta\|^2_{L^2(\Omega)} + \|\nabla \varrho\|^2_{L^2(\Omega)}\big) dt' \right],\nonumber
\end{align}
where the constant depends on $\|\varrho\|_{L^\infty(Q_T)}$, $\|\uB\|_{W^{1,\infty}(\Omega)}$, $\|\tB\|_{W^{1,\infty}(\Omega)}$ and $\Omega$.
On the other hand, from \eqref{temperature_eq} the elliptic estimate implies that
\begin{align}
\|\vartheta\|_{H^2(\Omega)} &\preceq\|\tB\|_{H^2(\Omega)} + \|\varrho \rD_t \vartheta\|_{L^2(\Omega)} + \big\||\nabla \u|^2\big\|_{L^2(\Omega)} + \|\varrho\hpt\vartheta \div\u\|_{L^2(\Omega)} \nonumber\\
&\preceq 1 + \|\varrho \rD_t \vartheta\|_{L^2(\Omega)} + \|\nabla \u\|_{L^4(\Omega)}^2 + \|\vartheta\|_{L^\infty(\Omega)} \|\nabla \u\|_{L^2(\Omega)}  \label{theta_H2_estimate}
\end{align}
so that
$$
\int_0^t \|\vartheta\|^2_{H^2(\Omega)} \dt' \preceq1 + \int_0^t \|\varrho \rD_t \vartheta\|^2_{L^2(\Omega)} \dt' + \int_0^t \|\nabla \u\|^4_{L^4(\Omega)} \dt' + \int_0^t \|\vartheta\|^2_{L^\infty(\Omega)} \|\nabla \u\|^2_{L^2(\Omega)} \dt'\scl
$$
thus \eqref{collection_of_basic_estimate.5} and \eqref{LinftyL2_estimate_for_grad_u_temp3} together yield
\begin{align}
& \|\nabla \u(t)\|^2_{L^2(\Omega)} \!+ \int_0^t \|\sqrt{\varrho}\, \rD_t \u\|^2_{L^2(\Omega)} \dt' \nonumber\\
&\quad \preceq\|\u_{{}_0}\|^2_{H^1(\Omega)} \!+ \|\vartheta_0\|^2_{L^\infty(\Omega)} \!+ \|\vartheta_0\|^2_{H^1(\Omega)} \!+ \|\f\|^2_{L^2(Q_T)} \!+ \|g\|^2_{L^2(Q_T)} \!+ \int_0^t \|\nabla \u\|^4_{L^4(\Omega)} \dt' \label{LinftyL2_estimate_for_grad_u_temp4}\\
&\qquad + \int_0^t \big(1+\|\nabla \u\|_{L^\infty(\Omega)} \!+ \|\vartheta\|^2_{L^\infty(\Omega)} \big) \big(1 + \|\nabla \u\|^2_{L^2(\Omega)} \!+ \|\nabla \vartheta\|^2_{L^2(\Omega)} \!+ \|\nabla \varrho\|^2_{L^2(\Omega)}\big) dt' . \nonumber
\end{align}

\begin{remark}
Suppose that instead of \eqref{L1Linfty_bound_for_grad_u} we assume only the usual Beale-Kato-Majda type blow-up criterion
$$
\smallint{0}{T} \|\nabla \u\|_{L^\infty(\Omega)}\dt < \infty
$$
without assuming any bounds for $\vartheta$. Then we have $\|\varrho\|_{L^\infty(Q_T)} < \infty$ by Lemma \ref{lem:lower_upper_bound_for_rho}. To proceed, we still need some bounds for $\vartheta$, and this can only be obtained by looking at the temperature equation \eqref{temperature_eq} which requires better estimate of the heat source $|\nabla \u|^2$ (which need not even be integrable in time if we only assume $\nabla \u \in L^1(0,T;L^\infty(\Omega))$). Therefore, we are forced to
assume the bound $\|\vartheta\|_{L^{5+\delta}(0,T;L^\infty(\Omega))} < \infty$ in \eqref{L1Linfty_bound_for_grad_u}.

On the other hand, the blow-up criterion
\begin{equation}\label{Huang's_assumption}
\smallint{0}{T} \|\nabla \u\|^2_{L^\infty(\Omega)}\dt < \infty
\end{equation}
in \cite{Hu2009} gives the least time integrability of the forcing $|\nabla \u|^2$ which directly leads to a bound for $\vartheta$ (Lemma 3.2). Note that this implies that our assumption \eqref{L1Linfty_bound_for_grad_u} is weaker than \eqref{Huang's_assumption}.
\end{remark}

\subsubsection{\texorpdfstring{$L^\infty_t H^1_x$}{LᵒᵒₜH¹ₓ}-estimates for \texorpdfstring{$\varrho$}{ρ}}\label{sec:H1_estimate_of_rho}

We now turn to the density. The main difficulty here is that differentiating the continuity equation produces boundary terms on $\Gamma_{\mathrm{in}}$ with the wrong sign, and these terms do not disappear under the inflow-outflow boundary condition. For this reason, the density estimate must be handled separately and requires a more careful use of the boundary information. In the model flat case, the structure can be seen directly from the formula for the normal derivative of $\varrho$ on the inflow boundary, while the general geometric case is postponed to the Appendix.

Differentiating the equation of continuity \eqref{density_eq} w.r.t. $x_k$
\begin{equation}\label{spatial_derivative_of_continuity_eq}
\partial_t \varrho,_k = - \div (\varrho,_k \u) - \div (\varrho \u,_k) = - \varrho,_k \div\u - \u \cdot \nabla \varrho,_k - \u,_k \cdot \nabla \varrho - \varrho\hpt \div \u,_k
\end{equation}
and then testing \eqref{spatial_derivative_of_continuity_eq} against $\varrho,_k \equiv \partial_{x_k} \varrho$, we find that
\begin{align}
& \frac{1}{2} \frac{d}{dt} \int_\Omega |\varrho,_k\!|^2 \dx \nonumber\\
& = - \frac{1}{2} \int_{\partial\Omega} (\u \cdot \bN) |\varrho,_k\!|^2 \dS - \int_\Omega \div \u |\varrho,_k\!|^2 \dx - \int_\Omega \varrho,_k \u,_k \cdot \nabla \varrho \dx - \int_\Omega \varrho \varrho,_k \div \u,_k \dx \nonumber\\
&\preceq - \frac{1}{2} \int_{\Gin} (\uB \cdot \bN) |\varrho,_k\!|^2 \dS
+ 2 \|\nabla \u\|_{L^\infty(\Omega)} \|\nabla \varrho\|^2_{L^2(\Omega)} + \|\varrho\|_{L^\infty(\Omega)} \|\nabla \varrho\|_{L^2(\Omega)} \|\u\|_{H^2(\Omega)}. \label{basic_H1_bound_for_rho}
\end{align}
The trouble term is the integral on the boundary which is a purely positive term since $\uB \cdot \bN < 0$ on $\Gin$. For the purpose of demonstration, let us assume that $\Omega = \bbT^2 \times (0,1)$, where $\Gin = \bbT^2 \times \{0\}$ and $\Gamma_{\text{out}} = \bbT^2 \times \{1\}$, and the proof of the general case is given in Section \ref{sec:bdy_integral_estimate1}. Then for $k=1,2$, $\varrho,_k$ is the derivative of $\rB$ in the tangential direction so we have
\begin{equation}\label{boundary_term_bound1}
- \frac{1}{2} \int_{\Gin} (\uB \cdot \bN) |\varrho,_k\!|^2 \dS \le \frac{1}{2} \uBu \|\rB\|^2_{H^1(\Gin)}\qquad\text{for $k=1,2$.}
\end{equation}
For $k=3$, we recall that since $\uB \cdot \bbN  < 0$ on $\Gin$, the compactness of $\Gin$ implies that
$$
\uBl \equiv \inf_{\Gin} (- \uB \cdot \bbN) > 0\pd
$$
Let $\u = (u_1,u_2,u_3)$. Then the equation of continuity \eqref{density_eq} implies that
\begin{equation}\label{normal_derivative_in_terms_of_tangential_derivative1}
\varrho,_{3} = - \frac{1}{u_3} \Big(\partial_t \varrho + u_1 \varrho,_{1} + u_2 \varrho,_{2} + \varrho\, \div \u\Big)
\end{equation}
whenever $u_3 \ne 0$.
Note that the fact that $- u_3 = \uB \cdot \bN < - \uBl$ implies that $u_3 > \uBl > 0$. Then
\begin{align*}
- \frac{1}{2} \int_{\Gin} (\uB \cdot \bN)& |\varrho,_3\!|^2 \dS\\
 &\preceq \int_{\Gin} \frac{1}{u_3} \Big(|\partial_t \rB| + |u_1| |\rB,_1| + |u_2| |\rB,_2| + \rB|\div \u|\Big)^2 \dS \\
&\preceq \frac{1}{\uBl} \int_{\Gin} \Big(|\partial_t \rB|^2 + |u_1|^2 |\rB,_1|^2 + |u_2|^2 |\rB,_2|^2 + \rB^2 |\div \u|^2\Big) \dS \\
&\preceq \frac{1}{\uBl} \Big(\|\partial_t \rB\|^2_{L^2(\Gin)} + \uBu^2 \|\rB\|^2_{H^1(\Gin)} + \|\rB\|^2_{L^\infty(\Gin)} \|\div \u\|^2_{L^2(\partial\Omega)} \Big).
\end{align*}
The interpolation inequality
\begin{equation}\label{interpolation_ineq}
\|f\|^p_{L^p(\partial\Omega)} \le C \|f\|^{\frac{q(4-p)}{6-q}}_{L^q(\Omega)} \|f\|^{\frac{6p-4q}{6-q}}_{H^1(\Omega)} \qquad\text{for $2\le p< 4$ and $2(p-1)\le q< 6$}
\end{equation}
for some constant $C$ \footnote{This $C$ is a pure constant depending on $p$ and $q$ because $\Omega = \bbT^2 \times (0,1)$, but in general $C = C(\Omega,p,q)$.}
further shows that
\begin{align}
- \frac{1}{2} \int_{\Gin} (\uB \cdot \bN) |\varrho,_3\!|^2 \dS &\preceq\|\partial_t \rB\|^2_{L^2(\Gin)} + \|\rB\|^2_{H^1(\Gin)} + \|\div \u\|_{L^2(\Omega)} \|\div \u\|_{H^1(\Omega)}  \nonumber\\
&\preceq\|\partial_t \rB\|^2_{L^2(\Gin)} + \|\rB\|^2_{H^1(\Gin)} + \|\u\|^\frac{1}{2}_{L^2(\Omega)} \|\u\|^\frac{3}{2}_{H^2(\Omega)} \label{boundary_term_bound2}
\end{align}
for some constant depending on $\|\div \u\|_{L^1(0,T;L^\infty(\Omega))},\rou,\rBu,\uBl,\uBu$.

Combining \eqref{basic_H1_bound_for_rho}, \eqref{boundary_term_bound1} and \eqref{boundary_term_bound2} and making use of \eqref{Linfty_bound_for_rho} and \eqref{Linfty_bound_for_u}, by Young's inequality we conclude that
\begin{equation}\label{collection_of_basic_estimates.1}
\frac{d}{dt} \|\nabla \varrho\|^2_{L^2(\Omega)} \preceq 1 + \|\nabla \u\|_{L^\infty(\Omega)} \|\nabla \varrho\|^2_{L^2(\Omega)} + \|\u\|^2_{H^2(\Omega)}\cm
\end{equation}
where the constant depends on $\|\div \u\|_{L^1(0,T;L^\infty(\Omega))}, \|\partial_t \rB\|_{L^2(\Gin)}, \|\rB\|_{H^1(\Gin)},\rou,\rBu,\\\uBl,\uBu$.

By elliptic estimates, the momentum equation \eqref{momentum_eq} leads to that for $1<p<\infty$,
\begin{align}
\|\u\|_{W^{2,p}(\Omega)} &\preceq\|\uB\|_{W^{2,p}\Omega)} + \|\varrho \rD_t\u \|_{L^p(\Omega)} + \|\nabla p\|_{L^p(\Omega)} + \|\varrho \f\|_{L^p(\Omega)}  \nonumber\\
&\preceq\|\uB\|_{W^{2,p}\Omega)} + \|\varrho \rD_t\u \|_{L^p(\Omega)} + \|\nabla (\varrho\hpt\vartheta)\|_{L^p(\Omega)} + \|\varrho \f\|_{L^p(\Omega)}  \nonumber\\
&\preceq1 + \|\rD_t \u\|_{L^p(\Omega)} + \|\nabla \vartheta\|_{L^p(\Omega)} + \|\vartheta\|_{L^\infty(\Omega)} \|\nabla \varrho\|_{L^p(\Omega)} + \|\f\|_{L^p(\Omega)} \label{u_W2p_estimate}
\end{align}
for some constant  depending on data, $\|\nabla \u\|_{L^1(0,T;L^\infty(\Omega))}$ and $p$, so \eqref{LinftyL2_estimate_for_grad_u_temp4} and \eqref{collection_of_basic_estimates.1} together imply that
\begin{align}
& \|\nabla \varrho(t)\|^2_{L^2(\Omega)}
\preceq\|\varrho_{{}_0}\|^2_{H^1(\Omega)} + \|\u_{{}_0}\|^2_{H^1(\Omega)} + \|\vartheta_0\|^2_{L^\infty(\Omega)} + \|\vartheta_0\|^2_{H^1(\Omega)} \nonumber\\
& \qquad+ \|\f\|^2_{L^2(Q_T)} + \|g\|^2_{L^2(Q_T)} + \int_0^t \|\nabla \u\|^4_{L^4(\Omega)} \dt' \label{LinftyL2_estimate_for_grad_rho_temp}\\
&\qquad + \int_0^t \big(1+\|\nabla \u\|_{L^\infty(\Omega)} + \|\vartheta\|^2_{L^\infty(\Omega)} \big) \big(1 + \|\nabla \u\|^2_{L^2(\Omega)} + \|\nabla \vartheta\|^2_{L^2(\Omega)} + \|\nabla \varrho\|^2_{L^2(\Omega)}\big) dt' .\nonumber
\end{align}

\subsection{Higher order estimates for $\u$ and \texorpdfstring{$\vartheta$}{θ}}\label{sec:higher_order_estimate}

To estimate $D_t\u$, we use a boundary adapted version of the standard material
derivative argument. The auxiliary quantity
\[
\w = D_t\u - V\cdot\nabla W
\]
is chosen so that $\w=0$ on $\partial\Omega$, which makes it possible to test the differentiated momentum equation without generating uncontrollable boundary terms. The choice of $V$ and $W$ will later be specialized to the present inflow-outflow setting. This is the key step that transfers the lower-order energy bound into a higher-order estimate for the velocity.

Once the lower-order energy bounds are available, we bootstrap them to higher regularity. The strategy is to first derive an equation for the material derivative of the momentum, estimate $D_t\u$ at one higher level, and then combine this with elliptic regularity to recover stronger spatial bounds for $\u$ and $\vartheta$. A central nonlinear term in this procedure is the quantity
\[
\int_0^T \|\nabla\u\|_{L^4(\Omega)}^4\,dt,
\]
whose control is needed in order to close the entire estimate.

\subsubsection{\texorpdfstring{$L^2_t H^1_x$}{L²ₜH¹ₓ}-estimates for \texorpdfstring{$\rD_t \u$}{Dₜu}}\label{sec:L2H1_estimate_for_Dtu}
Next we estimate $\rD_t\u$ at one higher level.  Applying the material derivative
to the momentum equation gives, in the $H^{-1}$-sense,
\begin{equation}\label{Hoff_eq}
\varrho \rD_t^2\u+\nabla p_t+\div(\nabla p\otimes \u)
=\Delta\u_t+\div(\Delta\u\otimes\u)+\varrho\rD_t\f .
\end{equation}
For boundary compatibility, we test this equation against
\[
\w=\rD_t\u-\bfv\cdot\nabla\bfw,
\]
where $\bfv$ and $\bfw$ are chosen so that $\w=0$ on $\partial\Omega$.  The general
identity obtained in this way is precisely the boundary-adapted estimate of
Lemma~\ref{lem:linearized_material_derivative_estimate}.

In the present Dirichlet setting the boundary data are independent of time, and we
choose
\[
\bfv=\uB,
\qquad
\bfw=\u,
\qquad
\w=\rD_t\u-\uB\cdot\nabla\u=\u_t+(\u-\uB)\cdot\nabla\u .
\]
Then $\w=0$ on $\partial\Omega$ while $\u_B$ independent of time $t$.  Testing \eqref{Hoff_eq} by this $\w$, using
$p_t+\div(p\u)=\varrho\rD_t\vartheta$, and absorbing the coercive terms by Young's
inequality, we obtain
\begin{align}\label{higher_order_estimate_equality}
\frac{d}{dt}\big\|\sqrt{\varrho}\,\w\big\|^2_{L^2(\Omega)}
&+\|\nabla\w\|^2_{L^2(\Omega)}\nonumber\\
&\preceq1+\|\rD_t\f\|^2_{L^2(\Omega)}
+\|\sqrt{\varrho}\,\rD_t\u\|^2_{L^2(\Omega)}
+\|\nabla\u\|^4_{L^4(\Omega)}
+\|\u\|^2_{H^2(\Omega)}\nonumber\\
&\qquad+\|\sqrt{\varrho}\,\rD_t\vartheta\|^2_{L^2(\Omega)}
+\|\vartheta\|^2_{L^\infty(\Omega)}\|\nabla\u\|^2_{L^2(\Omega)},
\end{align}
where the constant number depends on the data, $\|\varrho\|_{L^\infty(Q_T)}$, and
$\|\u\|_{L^\infty(Q_T)}$.  Since
\[
\|\sqrt{\varrho}\,\rD_t\u\|_{L^2(\Omega)}
\preceq \|\sqrt{\varrho}\,\w\|_{L^2(\Omega)}+\|\nabla\u\|^2_{L^2(\Omega)},\]
and \[
\|\nabla\rD_t\u\|_{L^2(\Omega)}
\preceq \|\nabla\w\|_{L^2(\Omega)}+\|\u\|_{H^2(\Omega)},
\]
we combine \eqref{higher_order_estimate_equality} with
\eqref{collection_of_basic_estimate.5}, \eqref{LinftyL2_estimate_for_grad_u_temp4},
and the elliptic estimate \eqref{u_W2p_estimate}.  After integration in time we get
\begin{align}
&\big\|\big(\sqrt{\varrho}\,\rD_t\u\big)(t)\big\|^2_{L^2(\Omega)}
+\int_0^t\|\nabla\rD_t\u\|^2_{L^2(\Omega)}\,dt' \nonumber\\
&\preceq1+\|\u_{{}_0}\|^2_{H^1(\Omega)}
+\|\vartheta_0\|^2_{H^2(\Omega)}
+\|\f\|^2_{H^1(Q_T)}+\|g\|^2_{L^2(Q_T)}
+\int_0^t\|\nabla\u\|^4_{L^4(\Omega)}\,dt'  \label{L2H1_estimate_for_ut_temp}\\
&
+\int_0^t\big(1+\|\nabla\u\|_{L^\infty(\Omega)}
+\|\vartheta\|^2_{L^\infty(\Omega)}\big)
\big(1+\|\nabla\vartheta\|^2_{L^2(\Omega)}
+\|\nabla\u\|^2_{L^2(\Omega)}
+\|\nabla\varrho\|^2_{L^2(\Omega)}\big)\,dt', \nonumber
\end{align}
where the constant depends on the data,
$\|\nabla\u\|_{L^1(0,T;L^\infty(\Omega))}$, and
$\|\vartheta\|_{L^{5+\delta}(0,T;L^\infty(\Omega))}$.

\begin{remark}
The essential difference from the corresponding estimate in \cite{BaFeMi2023} is
that \eqref{L2H1_estimate_for_ut_temp} contains
$\|\nabla\varrho\|_{L^2(\Omega)}$.  This term is caused by the inflow density
boundary condition and the nonhomogeneous Dirichlet data, and it is precisely why the
continuation argument must keep track of the density estimate together with the
velocity estimate.
\end{remark}

\subsubsection{\texorpdfstring{$L^4_t L^4_x$}{L⁴ₜL⁴ₓ}-estimates for \texorpdfstring{$\nabla \u$}{Du} and its consequences}
The purpose of this subsection is to control the nonlinear quantity
\[
\int_0^T \|\nabla\u\|_{L^4(\Omega)}^4\,dt,
\]
which appears repeatedly in the preceding estimates. This term is not part of the basic energy, but it can be recovered by combining the $L^\infty$-bound for $\u$, the estimate for $D_t\u$, and elliptic regularity for the momentum equation. Once this $L^4_{t,x}$-control is established, several previously open terms become integrable in time, and the Gronwall argument can be closed. We now provide a "linear estimate" for the term $\smallint{0}{t} \|\nabla \u\|^4_{L^4(\Omega)} \dt$ in \eqref{collection_of_basic_estimate.5}, \eqref{LinftyL2_estimate_for_grad_u_temp4},
\eqref{LinftyL2_estimate_for_grad_rho_temp} and \eqref{L2H1_estimate_for_ut_temp}.
By interpolation,
$$
\|\nabla \u\|^2_{L^4(\Omega)} \preceq \|\u\|_{L^\infty(\Omega)} \|\u\|_{H^2(\Omega)} \cm
$$
so the elliptic estimate \eqref{u_W2p_estimate} and Corollary \ref{cor:u_Linfty_bound} yield
\begin{equation}\label{Du_L4_estimate}
\|\nabla \u\|^4_{L^4(\Omega)}\preceq1 +  \|
\rD_t \u\|^2_{L^2(\Omega)} + \|\vartheta\|^2_{L^\infty(\Omega)} \|\nabla \varrho\|^2_{L^2(\Omega)} + \|\nabla \vartheta\|^2_{L^2(\Omega)} + \|\f\|^2_{L^2(\Omega)}
\end{equation}
for some constant
the constant depending on the data,
$\|\varrho\|_{L^\infty(Q_T)}$ and $\|\u\|_{L^\infty(Q_T)}$.
Define
\begin{align*}
{\rm X}(t) &= \|\nabla \vartheta(t)\|^2_{L^2(\Omega)} + \|\nabla \u(t)\|^2_{L^2(\Omega)} + \|\nabla \varrho(t)\|^2_{L^2(\Omega)} + \big\|(\sqrt{\varrho}\, \rD_t \u)(t)\big\|^2_{L^2(\Omega)} \cm\\
{\rm Y}(t) &= \big\|(\sqrt{\varrho}\, \rD_t \vartheta)(t)\big\|^2_{L^2(\Omega)} + \big\|(\sqrt{\varrho}\, \rD_t \u)(t)\big\|^2_{L^2(\Omega)} + \big\|(\nabla \rD_t \u)(t)\big\|^2_{L^2(\Omega)} \cm\\
{\rm K}(t) &= 1+\|\nabla \u(t)\|_{L^\infty(\Omega)} + \|\vartheta(t)\|^2_{L^\infty(\Omega)}\pd
\end{align*}
Combining \eqref{collection_of_basic_estimate.5}, \eqref{LinftyL2_estimate_for_grad_u_temp4}, \eqref{LinftyL2_estimate_for_grad_rho_temp}, \eqref{L2H1_estimate_for_ut_temp} and \eqref{Du_L4_estimate}, we conclude that
\begin{equation}\label{main_energy_ineq}
{\rm X}(t) + \int_0^t {\rm Y}(t')\dt' \preceq1 + \|\f\|^2_{H^1(Q_T)} + \|g\|^2_{L^2(Q_T)} + \int_0^t {\rm K}(t') \big(1 + {\rm X}(t')\big)\dt'
\end{equation}
for  the constant  depending the data, $\|\nabla \u\|_{L^1(0,T;L^\infty(\Omega))}$ and $\|\vartheta\|_{L^{5+\delta}(0,T;L^\infty(\Omega))}$. By the assumption \eqref{L1Linfty_bound_for_grad_u}, ${\rm K} \in L^1(0,T)$; thus the Gronwall inequality implies that
\begin{equation}\label{closed_estimate_formal_form}
{\rm X}(t) + \int_0^t {\rm Y}(t')\dt' \preceq1 + \|\f\|^2_{H^1(Q_T)} + \|g\|^2_{L^2(Q_T)}.
\end{equation}
Using \eqref{Du_L4_estimate}, the estimate above further imply that
\begin{equation}\label{Du_L4L4_estimate}
\int_0^T \|\nabla \u\|^4_{L^4(\Omega)} dt \preceq1 + \|\f\|^2_{H^1(Q_T)} + \|g\|^2_{L^2(Q_T)},
\end{equation}
and the combination of \eqref{closed_estimate_formal_form} and \eqref{Du_L4L4_estimate} shows that
\begin{equation}\label{closed_estimate}
\hspace{-20pt}\begin{array}{rl}
& \displaystyle{} \|\nabla \vartheta(t)\|^2_{L^2(\Omega)} +\,\displaystyle{} \|\nabla \u(t)\|^2_{L^2(\Omega)} + \|\nabla \varrho(t)\|^2_{L^2(\Omega)} + \big\|\rD_t \u(t)\big\|^2_{L^2(\Omega)} \vspace{.1cm}\\
&\qquad +\displaystyle{} \int_0^t \Big[\big\|\rD_t \vartheta\big\|^2_{L^2(\Omega)} + \big\|\rD_t \u\big\|^2_{H^1(\Omega)} + \|\nabla\u\|^4_{L^4(\Omega)} \Big] \dt' \preceq1 + \|\f\|^2_{H^1(Q_T)} + \|g\|^2_{L^2(Q_T)}. \end{array}
\end{equation}

\subsubsection{\texorpdfstring{$L^2_t H^2_x$}{L²ₜH²ₓ}-estimates for $\u$ and \texorpdfstring{$\vartheta$}{θ}}\label{sec:L2H2_estimate_of_utheta}
After the estimate for $\nabla\u$ in $L^4_tL^4_x$ has been obtained, the elliptic parts of the momentum and temperature equations can be exploited more efficiently. At this stage, the equations are used to convert control of the material derivatives into control of the full spatial $H^2$-norms. This is the first point where the parabolic-elliptic structure of the system is fully reflected in the estimates. Using \eqref{u_W2p_estimate} (with $p=2$) we immediately conclude that
\begin{align*}
\|\u\|^2_{H^2(\Omega)} &\preceq1 + \|
\rD_t \u\|^2_{L^2(\Omega)} + \|\nabla \vartheta\|^2_{L^2(\Omega)} + \|\vartheta\|^2_{L^\infty(\Omega)} \|\nabla \varrho\|^2_{L^2(\Omega)} + \|\f\|^2_{L^2(\Omega)}  \nonumber \\
&\preceq\big(1+\|\vartheta\|^2_{L^\infty(\Omega)}\big) \Big[1 + \|\f\|^2_{H^1(Q_T)} \!+ \|g\|^2_{L^2(Q_T)}\Big].
\end{align*}
Similarly, \eqref{theta_H2_estimate} and \eqref{Du_L4_estimate} show that
\begin{align*}
\|\vartheta\|^2_{H^2(\Omega)} &\preceq1 + \| \rD_t \vartheta\|^2_{L^2(\Omega)} + \|\rD_t \u\|^2_{L^2(\Omega)} \\&+ \|\vartheta\|^2_{L^\infty(\Omega)} \big(\|\nabla \varrho\|^2_{L^2(\Omega)} + \|\nabla \u\|^2_{L^2(\Omega)}\big) + \|\nabla \vartheta\|^2_{L^2(\Omega)}.
\end{align*}
The combination of the two estimates above and \eqref{closed_estimate} yield
\begin{equation}\label{utheta_L2H2_estimate}
\int_0^T \Big[\|\u\|^2_{H^2(\Omega)} + \|\vartheta\|^2_{H^2(\Omega)}\Big] \dt \preceq1 \hspace{-1pt}+\hspace{-1pt} \|\f\|^2_{H^1(Q_T)} \!+ \|g\|^2_{L^2(Q_T)}.
\end{equation}
Moreover,
$$
\|\u_t\|_{H^1(\Omega)} \preceq\|\rD_t \u\|_{H^1(\Omega)} + \|\u \cdot \nabla \u\|_{H^1(\Omega)} \preceq\|\rD_t \u\|_{H^1(\Omega)} + \|\u\|_{H^2(\Omega)} + \|\nabla \u\|^2_{L^4(\Omega)},
$$
so the $L^2_t H^1_x$-estimate for $\rD_t \u$ implies that
\begin{align}
\int_0^T \|\u_t\|^2_{H^1(\Omega)} \dt &\preceq\int_0^T \Big[\|\rD_t \u\|^2_{H^1(\Omega)} + \|\u\|^2_{H^2(\Omega)} + \|\nabla \u\|^4_{L^4(\Omega)} \Big] dt \nonumber\\
&\preceq1 + \|\f\|^2_{H^1(Q_T)} \!+ \|g\|^2_{L^2(Q_T)}.
\end{align}

\begin{remark}
The assumption $\|\vartheta\|_{L^{5+\delta}(0,T;L^\infty(\Omega))} < \infty$ indeed provides better time integrability of $\|\u\|_{H^2(\Omega)}$ and $\|\nabla \u\|_{L^4(\Omega)}$:
$$
\int_0^T \Big[\|\u\|^{5+\delta}_{H^2(\Omega)} + \|\nabla \u\|^{10+2\delta}_{L^4(\Omega)}\Big] dt \preceq\text{data},\|\f\|_{H^1(Q_T)}, \|g\|_{L^2(Q_T)}.
$$
Moreover, if the condition $\|\vartheta\|_{L^{5+\delta}(0,T;L^\infty(\Omega))} < \infty$ is replaced by $\|\vartheta\|_{L^\infty(Q_T)} < \infty$, then
$$
\sup_{t\in [0,T]} \Big[\big\|\u(t)\big\|^2_{H^2(\Omega)} + \|\nabla \u(t)\|^4_{L^4(\Omega)}\Big] \preceq1 + \|\f\|^2_{H^1(Q_T)} \!+ \|g\|^2_{L^2(Q_T)}.
$$
\end{remark}

\subsubsection{\texorpdfstring{$L^\infty_t L^2_x$}{LᵒᵒₜL²ₓ}-estimate for \texorpdfstring{$\varrho_t$}{ρₜ}}
Once both $\varrho$ and $\u$ are controlled at the previous level, the estimate for $\varrho_t$ follows directly from the continuity equation. Although this step is technically simple, it is important for the later treatment of the differentiated forcing terms and for the final application of the parabolic regularity theorem to the temperature equation. From the equation of continuity \eqref{density_eq}, using \eqref{closed_estimate} it is easy to see that
\begin{align}
\|\rho_t(t)\|_{L^2(\Omega)} &\preceq\|\u(t)\|_{L^\infty(\Omega)} \|\nabla \varrho(t)\|_{L^2(\Omega)} + \|\varrho(t)\|_{L^\infty(\Omega)} \|\div\u (t)\|_{L^2(\Omega)} \nonumber\\
&\preceq1 \hspace{-1pt}+\hspace{-1pt} \|\f\|^2_{H^1(Q_T)} \!+ \|g\|^2_{L^2(Q_T)}.\label{rho_t_L2L2_estimate}
\end{align}

\subsubsection{\texorpdfstring{$L^\infty_t L^3_x$}{LᵒᵒₜL³ₓ}-estimates for \texorpdfstring{$\nabla \varrho$}{Dρ}}\label{sec:L3_estimate_of_rho}
Next, we derive an $L^3$-estimates for $\nabla \varrho$ which is needed for even higher order estimates. Testing \eqref{spatial_derivative_of_continuity_eq} against $|\varrho,_k\!|\varrho,_k$, we find that
\begin{align*}
&\frac{1}{3} \frac{d}{dt} \|\varrho,_k\!\|^3_{L^3(\Omega)} \\&= - \frac{1}{3} \int_{\partial\Omega} (\u\cdot \bN) |\varrho,_k\!|^3\dS - \frac{2}{3} \int_\Omega \div \u |\varrho,_k\!|^3\dx - \int_\Omega \big(\u,_k\cdot \nabla \varrho + \varrho\hpt \div \u,_k) |\varrho,_k\!| \varrho,_k \dx
\end{align*}
so the interpolation inequality $\|f\|_{L^3(\Omega)} \preceq \|f\|^{\frac{1}{2}}_{L^2(\Omega)} \|f\|^{\frac{1}{2}}_{H^1(\Omega)}$ shows that
\begin{align}
\frac{d}{dt} \|\varrho,_k\!\|^3_{L^3(\Omega)} &\preceq- \int_\Gin (\u\cdot \bN) |\varrho,_k\!|^3\dS + \|\nabla \u\|_{L^\infty(\Omega)} \|\nabla \varrho\|^3_{L^3(\Omega)} +  \|\nabla^2 \u\|_{L^3(\Omega)} \|\nabla \varrho\|^2_{L^3(\Omega)} \nonumber\\
&\preceq- \int_\Gin (\u\cdot \bN) |\varrho,_k\!|^3\dS + \preceq\big(1 + \|\nabla \u\|_{L^\infty(\Omega)}\big) \|\nabla \varrho\|^3_{L^3(\Omega)} + \|\u\|^3_{W^{2,3}(\Omega)}. \label{Drho_LinftyL3_estimate_temp1}
\end{align}
Using \eqref{u_W2p_estimate}, the interpolation inequality $\|f\|_{L^3(\Omega)} \le C \|f\|^\frac{1}{2}_{L^2(\Omega)} \|f\|^\frac{1}{2}_{L^6(\Omega)}$ and the continuous embedding $H^1(\Omega) \contsubset L^6(\Omega)$ show that
\begin{align*}
&\|\u\|^3_{W^{2,3}(\Omega)}\\ &\qquad\preceq\|\uB\|^3_{W^{2,3}(\Omega)} + \|\sqrt{\varrho}\, \rD_t \u\|^3_{L^3(\Omega)} + \|\nabla \vartheta\|^3_{L^3(\Omega)} + \|\vartheta\|^3_{L^\infty(\Omega)} \|\nabla \varrho\|^3_{L^3(\Omega)} + \|\f\|^3_{L^3(\Omega)} \\
&\qquad\preceq1 + \|\rD_t \u\|^3_{L^3(\Omega)} + \|\nabla \vartheta\|^3_{L^3(\Omega)} + \|\vartheta\|^3_{L^\infty(\Omega)} \|\nabla \varrho\|^3_{L^3(\Omega)} + \|\f\|^3_{L^3(\Omega)}  \\
&\qquad\preceq1 + \|\rD_t \u\|^\frac{3}{2}_{L^2(\Omega)} \|\rD_t \u\|^\frac{3}{2}_{H^1(\Omega)} + \|\nabla \vartheta\|^\frac{3}{2}_{L^2(\Omega)} \|\vartheta\|^\frac{3}{2}_{H^2(\Omega)} + \|\f\|^\frac{3}{2}_{L^2(\Omega)} \|\f\|^\frac{3}{2}_{H^1(\Omega)} \\
&\qquad\quad +  \|\vartheta\|^3_{L^\infty(\Omega)} \|\nabla \varrho\|^3_{L^3(\Omega)}\pd
\end{align*}
By Young's inequality,
\begin{align*}
 \frac{d}{dt} \|\varrho,_k\!\|^3_{L^3(\Omega)}
&\preceq- \int_\Gin (\u\cdot \bN) |\varrho,_k\!|^3\dS +  \big(1 + \|\nabla \u\|_{L^\infty(\Omega)} + \|\vartheta\|^3_{L^\infty(\Omega)}\big) \|\nabla \varrho\|^3_{L^3(\Omega)} \\
&\qquad + 1 + \|\rD_t \u\|^6_{L^2(\Omega)} + \|\nabla \vartheta\|^6_{L^2(\Omega)}+ \|\f\|^6_{L^2(\Omega)} + \|\nabla \rD_t \u\|^2_{L^2(\Omega)}  \\&\qquad+ \|\vartheta\|^2_{H^2(\Omega)} + \|\f\|^2_{H^1(\Omega)},
\end{align*}
and estimates \eqref{closed_estimate}-\eqref{Drho_LinftyL3_estimate_temp1} further show that
\begin{equation}\label{Drho_LinftyL3_estimate_temp2}
\begin{array}{l}
\displaystyle{}\frac{d}{dt} \|\varrho,_k\!\|^3_{L^3(\Omega)} \preceq - \int_\Gin (\u\cdot \bN) |\varrho,_k\!|^3\dS +  \big(1 + \|\nabla \u\|_{L^\infty(\Omega)} \!+\hspace{-1pt} \|\vartheta\|^3_{L^\infty(\Omega)}\big) \|\nabla \varrho\|^3_{L^3(\Omega)} \vspace{.1cm}\\
\qquad +\, \Big[1 + \|\f\|^2_{H^1(Q_T)} + \|g\|^2_{L^2(Q_T)} \Big]^3 + \|\nabla \rD_t \u\|^2_{L^2(\Omega)} + \|\vartheta\|^2_{H^2(\Omega)} + \|\f\|^2_{H^1(\Omega)}.
\end{array}
\end{equation}
The integral $\smallint{\Gin}{} (\uB \cdot \bN) |\varrho,_k\!|^3 \dS$ for $k=1, 2$ (again, we show how the estimate is obtained in the case $\Omega = \bbT^2 \times (0,1)$, and the proof of the general case can be derived using partition-of-unity and straighten the inflow boundary, as in Section \ref{sec:bdy_integral_estimate1}) can be dominated by $C \|\rB\|^3_{L^3(\Gin)}$. For $k=3$, we use \eqref{normal_derivative_in_terms_of_tangential_derivative1} and the Jensen inequality to obtain that
\begin{align*}
- \int_{\Gin} (\uB \cdot \bN) &|\varrho,_3\!|^3 \dS \\&\preceq\int_{\Gin} \frac{1}{u_3^2} \Big(|\partial_t \rB| + |u_1| |\rB,_1| + |u_2| |\rB,_2| + \varrho |\div \u|\Big)^3 \dS \\
&\preceq \frac{16}{\underline{\u_{{}_B}^2}} \int_{\Gin} \Big(|\partial_t \rB|^3 + |u_1|^3 |\rB,_1|^3 + |u_2|^3 |\rB,_2|^3 + \varrho^3 |\div \u|^3\Big) \dS \\
&\preceq1 + \|\div \u\|^3_{L^3(\partial\Omega)}
\end{align*}
and \eqref{interpolation_ineq} (with $p=3, q=4$) further shows that
$$
- \int_{\Gin} (\uB \cdot \bN) |\varrho,_3\!|^3 \dS \preceq1 + \|\nabla \u\|^2_{L^4(\Omega)} \|\nabla \u\|_{H^1(\Omega)}\preceq1+\|\nabla \u\|^4_{L^4(\Omega)}
+ \|\u\|^2_{H^2(\Omega)}.
$$
Therefore, \eqref{Drho_LinftyL3_estimate_temp2} leads to the differential inequality
\begin{align*}
\frac{d}{dt}& \|\nabla \varrho\|^3_{L^3(\Omega)}\\ &\preceq \big(1 + \|\nabla \u\|_{L^\infty(\Omega)} + \|\vartheta\|^3_{L^\infty(\Omega)}\big) \|\nabla \varrho\|^3_{L^3(\Omega)} +\Big[1 + \|\f\|^2_{H^1(Q_T)} + \|g\|^2_{L^2(Q_T)} \Big]^3 \nonumber\\
&\quad +  \|\nabla \u\|^4_{L^4(\Omega)} + \|\u\|^2_{H^2(\Omega)} + \|\nabla \rD_t \u\|^2_{L^2(\Omega)} + \|\vartheta\|^2_{H^2(\Omega)} + \|\f\|^2_{H^1(\Omega)}.
\end{align*}
Since the assumption \eqref{L1Linfty_bound_for_grad_u} implies that
$$
1 + \|\nabla \u(\cdot)\|_{L^\infty(\Omega)} + \|\vartheta(\cdot)\|^3_{L^\infty(\Omega)} \in L^1(0,T),
$$
by the Gronwall inequality we conclude from \eqref{Du_L4_estimate}, \eqref{closed_estimate} and \eqref{utheta_L2H2_estimate} that
\begin{equation}\label{Drho_LinftyL3_estimate}
\|\nabla \varrho(t)\|_{L^3(\Omega)} \le C(\text{data}, \|\nabla \u\|_{L^1(0,T;L^\infty(\Omega))}, \|\vartheta\|_{L^{5+\delta}(0,T;L^\infty(\Omega))},\|\f\|_{H^1(Q_T)}, \|g\|_{L^2(Q_T)}).
\end{equation}

\subsection{Highest order estimates for \texorpdfstring{$\varrho$ and $\vartheta$}{ρ and θ}}\label{sec:even_higher_order_estimate}
We now enter the final stage of the continuation argument. The goal of this subsection is to close the highest-order norms needed to remain in the strong solution class up to time $T$. The most delicate part is the second-order estimate for the density, because differentiating the continuity equation twice produces boundary terms that are specific to the inflow-outflow setting. Once this difficulty is overcome, the highest-order estimates for $\u$ and $\vartheta$ follow from the
already established lower-order bounds together with elliptic and parabolic regularity.

\subsubsection{\texorpdfstring{$L^\infty_t H^2_x$}{LᵒᵒₜH²ₓ}-estimates for \texorpdfstring{$\varrho$}{ρ}}\label{sec:H2_estimate_of_rho}
Differentiate the equation of continuity \eqref{density_eq} w.r.t. $x_k$ and $x_j$
\begin{equation}\label{two_spatial_derivative_of_continuity_eq}
\varrho_t,_{kj} + \div (\varrho,_{kj} \u) + \div (\varrho,_k \u,_j) + \div (\varrho,_j \u,_k) + \div (\varrho \u,_{kj}) = 0
\end{equation}
and then testing \eqref{two_spatial_derivative_of_continuity_eq} against $\varrho,_{kj} \equiv \partial_{x_k} \partial_{x_j} \varrho$, we find that
\begin{align*}
\frac{1}{2} \frac{d}{dt}& \|\varrho,_{kj}\|^2_{L^2(\Omega)}\\ &= - \frac{1}{2} \int_{\partial\Omega} (\u \cdot \bN) |\varrho,_{kj}|^2 \dS + \frac{1}{2} \int_\Omega \div \u |\varrho,_{kj}|^2 \dx + 2 \int_\Omega (\u,_j \cdot \nabla \varrho,_k) \varrho,_{kj} \dx \\
&\quad + \int_\Omega (\u,_{kj} \cdot \nabla \varrho) \varrho,_{kj} \dx + \int_\Omega \varrho \varrho,_{kj} (\div \u),_{kj} \dx\pd
\end{align*}
It then follows from the continuous embedding $H^1(\Omega) \contsubset L^6(\Omega)$ and \eqref{Drho_LinftyL3_estimate} that
\begin{align}
& \frac{d}{dt} \|\varrho,_{kj}\|^2_{L^2(\Omega)}\preceq - \int_\Gin (\u \cdot \bN) |\varrho,_{kj}|^2 \dS  \nonumber\\
&\qquad\quad + \|\nabla\u\|_{L^\infty(\Omega)} \|\nabla^2 \varrho\|_{L^2(\Omega)}^2 + \|\nabla^2 \u\|_{L^6(\Omega)} \|\nabla^2 \varrho\|_{L^2(\Omega)} + \|\nabla^3 \u\|_{L^2(\Omega)} \|\nabla^2 \varrho\|_{L^2(\Omega)} \nonumber\\
&\qquad\preceq - \int_\Gin (\u \cdot \bN) |\varrho,_{kj}|^2 \dS +  \big(1 + \|\nabla \u\|_{L^\infty(\Omega)} \big) \|\nabla^2 \varrho\|^2_{L^2(\Omega)} + \|\u\|^2_{H^3(\Omega)}\label{D2rho_LinftyL2_estimate_temp}
\end{align}
for the constant depending on the data, $\|\u\|_{L^1(0,T;L^\infty(\Omega))}$, $\|\vartheta\|_{L^{5+\delta}(0,T;L^\infty(\Omega))}$,\\ $\|\f\|_{H^1(Q_T)}$ and $\|g\|_{L^2(Q_T)}$.

The boundary integral $- \smallint{\Gin}{} (\u \cdot \bN) |\varrho,_{jk}\!|^2\dS$\vspace{.1cm} is more complicated than those in Sections \ref{sec:H1_estimate_of_rho} and \ref{sec:L3_estimate_of_rho}. To keep the focus on the main estimate, we state the result
\begin{equation}\label{D2_inflow_boundary_estimate}
- \int_\Gin (\u \cdot \bN) |\nabla^2 \varrho|^2\dS \preceq1 + \|\nabla \u\|^4_{L^4(\Omega)} + \|\rD_t \u\|^2_{H^1(\Omega)} + \|\f\|^2_{H^1(Q_T)} + \|g\|^2_{L^2(Q_T)} + \|\u\|^2_{H^3(\Omega)}
\end{equation}
without proof here. The proof of \eqref{D2_inflow_boundary_estimate} is given in Section \ref{sec:bdy_integral_estimate2}.

Having estimate \eqref{D2_inflow_boundary_estimate}, we then can conclude from \eqref{D2rho_LinftyL2_estimate_temp} that
\begin{equation}\label{D2rho_L2_diff_ineq}
\frac{d}{dt} \|\nabla^2 \varrho\|^2_{L^2(\Omega)} \preceq1 + \|\nabla \u\|^4_{L^4(\Omega)} + \|\rD_t \u\|^2_{H^1(\Omega)} + \big(1 + \|\nabla \u\|_{L^\infty(\Omega)} \big) \|\nabla^2 \varrho\|^2_{L^2(\Omega)} + \|\u\|^2_{H^3(\Omega)}.
\end{equation}

Let us apply the elliptic estimate to the momentum equation \eqref{momentum_eq},
\begin{equation}\label{u_H3_estimate_temp}
\|\u\|_{H^3(\Omega)} \preceq\|\uB\|_{H^3(\Omega)} + \|\varrho \rD_t \u\|_{H^1(\Omega)} + \|\nabla p\|_{H^1(\Omega)} + \|\varrho \f\|_{H^1(\Omega)}.
\end{equation}
Note that
$$
\|\nabla p\|_{H^1(\Omega)} \preceq\|\vartheta\|_{L^\infty(\Omega)} \|\nabla^2 \varrho\|_{L^2(\Omega)} + \|\nabla \varrho\|_{L^3(\Omega)} \|\nabla \vartheta\|_{L^6(\Omega)} + \|\varrho\|_{L^\infty(\Omega)} \|\nabla^2 \vartheta\|_{L^2(\Omega)},
$$
and for any function $f$,
$$
\|\varrho f\|_{H^1(\Omega)}\preceq\|f\|_{H^1(\Omega)} + \|f\|_{L^6(\Omega)} \|\nabla \varrho\|_{L^3(\Omega)}.
$$
Using \eqref{Drho_LinftyL3_estimate} and the continuous embedding $H^1(\Omega) \contsubset L^6(\Omega)$, we have
\begin{equation}\label{prior_to_u_H3_estimate}
\|\varrho f\|_{H^1(\Omega)}
\preceq \|f\|_{H^1(\Omega)} \quad\text{ and }\quad \|\nabla p\|_{H^1(\Omega)} \preceq\|\vartheta\|_{L^\infty(\Omega)} \|\nabla^2 \varrho\|_{L^2(\Omega)} + \|\vartheta\|_{H^2(\Omega)}.
\end{equation}
Therefore, \eqref{u_H3_estimate_temp} yields
\begin{align}
\|\u\|_{H^3(\Omega)}
&\preceq1 + \|\rD_t \u\|_{H^1(\Omega)} + \|\vartheta\|_{H^2(\Omega)} + \|\f\|_{H^1(\Omega)} + \|\vartheta\|_{L^\infty(\Omega)} \|\nabla^2 \varrho\|_{L^2(\Omega)},\label{u_H3_estimate_temp2}
\end{align}
where the constant depends on the data, $\|\nabla \u\|_{L^1(0,T;L^\infty(\Omega))}$, $\|\vartheta\|_{L^{5+\delta}(0,T;L^\infty(\Omega))}$, \\ $\|\f\|_{H^1(Q_T)}$, and $\|g\|_{L^2(Q_T)}$.
Using \eqref{u_H3_estimate_temp2} in \eqref{D2rho_L2_diff_ineq}, we obtain that
\begin{align*}
\frac{d}{dt} \|\nabla^2 \varrho\|^2_{L^2(\Omega)} &\preceq1 + \|\f\|^2_{H^1(\Omega)} + \|\nabla \u\|^4_{L^4(\Omega)} + \|\rD_t \u\|^2_{H^1(\Omega)} + \|\vartheta\|^2_{H^2(\Omega)}  \\
&\qquad + \big(1 + \|\nabla \u\|_{L^\infty(\Omega)} + \|\vartheta\|^2_{L^\infty(\Omega)} \big) \|\nabla^2 \varrho\|^2_{L^2(\Omega)}.
\end{align*}
Using \eqref{closed_estimate}, we apply the Gronwall's inequality to conclude that
\begin{align*}
\|\nabla^2 \varrho&\|^2_{L^2(\Omega)} \\&\preceq\|\varrho_{{}_0}\|^2_{H^2(\Omega)} + \int_0^t \big(1 + \|\f\|^2_{H^1(\Omega)} + \|\nabla \u\|^4_{L^4(\Omega)} + \|\rD_t \u\|^2_{H^1(\Omega)} + \|\vartheta\|^2_{H^2(\Omega)}\big) dt' \\
&\preceq \mathrm{1}\Big(\text{data},\|\nabla \u\|_{L^1(0,T;L^\infty(\Omega))}, \|\vartheta\|_{L^{5+\delta}(0,T;L^\infty(\Omega))}, \|\f\|_{H^1(Q_T)}, \|g\|_{L^2(Q_T)}\Big).
\end{align*}
It then follows from \eqref{Linfty_bound_for_rho} that
\begin{align}\label{rho_LinftyH2_estimate}
\sup_{t\in [0,T]}& \|\varrho(t)\|_{H^2(\Omega)}\\& \preceq\mathrm{1}\Big(\text{data},\|\nabla \u\|_{L^1(0,T;L^\infty(\Omega))}, \|\vartheta\|_{L^{5+\delta}(0,T;L^\infty(\Omega))}, \|\f\|_{H^1(Q_T)}, \|g\|_{L^2(Q_T)}\Big).
\end{align}

\subsubsection{\texorpdfstring{$L^2_t H^3_x$}{L²ₜH³ₓ} and \texorpdfstring{$L^\infty_t H^2_x$}{LᵒᵒₜH²ₓ}-estimates for $\u$}
Now we proceed to the $L^2_t H^3_x$-estimate for $\u$.
Using \eqref{closed_estimate} and \eqref{u_H3_estimate_temp2}, we find that
\begin{align*}
\int_0^T &\|\u\|^2_{H^3(\Omega)} \dt\\
&\preceq \mathrm{1}\big(\text{data},\|\nabla \u\|_{L^1(0,T;L^\infty(\Omega))}, \|\vartheta\|_{L^{5+\delta}(0,T;L^\infty(\Omega))}, \|\f\|_{H^1(Q_T)}, \|g\|_{L^2(Q_T)}\big),
\end{align*}
and the momentum equation \eqref{momentum_eq} further implies that
\begin{align*}
\int_0^T &\|\varrho \rD_t \u\|^2_{H^1(\Omega)}\dt\\ &\preceq\int_0^T \Big[\|\Delta \u\|^2_{H^1(\Omega)} + \|\nabla p\|^2_{H^1(\Omega)} + \|\varrho \f\|^2_{H^1(\Omega)}\Big] dt  \nonumber\\
&\preceq\mathrm{1}\big(\text{data},\|\nabla \u\|_{L^1(0,T;L^\infty(\Omega))}, \|\vartheta\|_{L^{5+\delta}(0,T;L^\infty(\Omega))}, \|\f\|_{H^1(Q_T)}, \|g\|_{L^2(Q_T)}\big).
\end{align*}
As a consequence,
\begin{align*}
& \int_0^T \|\u_t\|^2_{H^1(\Omega)} \dt \preceq \int_0^T \Big[\|\rD_t \u\|^2_{H^1(\Omega)} + \|\u \cdot \nabla \u\|^2_{L^2(\Omega)}\Big] dt \nonumber\\
& \preceq \int_0^T \Big[\|\varrho \rD_t \u\|^2_{L^2(\Omega)} + \|\varrho \nabla \rD_t \u\|^2_{L^2(\Omega)}\Big] dt + C \int_0^T \Big[\|\nabla^2 \u\|^2_{L^2(\Omega)} + \|\nabla \u\|^4_{L^4(\Omega)}\Big] dt \nonumber\\
&\preceq\int_0^T \Big[\|\varrho \rD_t \u\|^2_{L^2(\Omega)} + \|\nabla (\varrho \rD_t \u) - \rD_t \u \otimes \nabla \varrho\|^2_{L^2(\Omega)} + \|\nabla^2 \u\|^2_{L^2(\Omega)} + \|\nabla \u\|^4_{L^4(\Omega)}\Big] dt \nonumber\\
& \preceq \int_0^T \Big[\|\varrho \rD_t \u\|^2_{H^1(\Omega)} + \|\rD_t \u\|^2_{L^6(\Omega)} \|\nabla \varrho\|^2_{L^3(\Omega)} + \|\nabla^2 \u\|^2_{L^2(\Omega)} + \|\nabla \u\|^4_{L^4(\Omega)}\Big] dt \nonumber\\
&\preceq\mathrm{1}\big(\text{data},\|\nabla \u\|_{L^1(0,T;L^\infty(\Omega))}, \|\vartheta\|_{L^{5+\delta}(0,T;L^\infty(\Omega))}, \|\f\|_{H^1(Q_T)}, \|g\|_{L^2(Q_T)}\big).
\end{align*}
By interpolation in time,
\begin{align*}
\sup_{t\in [0,T]}& \|\u(t)\|^2_{H^2(\Omega)}\\ &\preceq\|\u_{{}_0}\|^2_{H^2(\Omega)} + \|\u\|^2_{L^2(0,T;H^3(\Omega))} + \|\u_t\|^2_{L^2(0,T;H^1(\Omega))} \nonumber\\
&\preceq\mathrm{1}\big(\text{data},\|\nabla \u\|_{L^1(0,T;L^\infty(\Omega))}, \|\vartheta\|_{L^{5+\delta}(0,T;L^\infty(\Omega))}, \|\f\|_{H^1(Q_T)}, \|g\|_{L^2(Q_T)}\big),
\end{align*}
and using the momentum equation \eqref{momentum_eq} we also have
\begin{align*}
& \sup_{t\in [0,T]} \|\u_t(t)\|^2_{L^2(\Omega)} \nonumber\\
&\preceq\sup_{t\in[0,T]} \Big[\|\u(t)\|^2_{H^2(\Omega)} + \|\u(t)\|^2_{L^\infty(\Omega)} \|\nabla \u(t)\|^2_{L^2(\Omega)} + \|\nabla (\varrho \vartheta)(t)\|^2_{L^2(\Omega)} + \|\f(t)\|^2_{L^2(\Omega)} \Big]\nonumber\\
&\preceq\mathrm{1}\big(\text{data},\|\nabla \u\|_{L^1(0,T;L^\infty(\Omega))}, \|\vartheta\|_{L^{5+\delta}(0,T;L^\infty(\Omega))}, \|\f\|_{H^1(Q_T)}, \|g\|_{L^2(Q_T)}\big).
\end{align*}
Finally, all the estimates above yield the $L^2_tH^{-1}_x$-estimate for $\u_{tt}$. Differentiating the momentum equation \eqref{momentum_eq} in time and estimating in $H^{-1}(\Omega)$, we obtain
\begin{align*}
&\|\u_{tt}\|_{H^{-1}(\Omega)}
\preceq
\|\varrho_t\|_{H^1(\Omega)}
\Big(
\|\u_t\|_{L^2(\Omega)}
+\|\u\|_{L^\infty(\Omega)}\|\nabla\u\|_{L^2(\Omega)}
+\|\nabla(\varrho\vartheta)\|_{L^2(\Omega)}
\\
&\qquad+\|\Delta\u\|_{L^2(\Omega)}+\|\f\|_{L^2(\Omega)}
\Big)
+\|\u_t\|_{H^1(\Omega)}\|\u\|_{H^2(\Omega)}
+\|\u\|_{L^\infty(\Omega)}\|\nabla\u_t\|_{L^2(\Omega)}
 \\
&\qquad+\|\nabla(\varrho\vartheta)_t\|_{H^{-1}(\Omega)}
+\|\Delta\u_t\|_{H^{-1}(\Omega)}
+\|\f_t\|_{H^{-1}(\Omega)}.
\end{align*}
Here
\[
\|\nabla(\varrho\vartheta)_t\|_{H^{-1}(\Omega)}
\le \|(\varrho\vartheta)_t\|_{L^2(\Omega)}
\preceq
\|\varrho_t\|_{H^1(\Omega)}\|\vartheta\|_{H^1(\Omega)}
+\|\varrho\|_{L^\infty(\Omega)}\|\vartheta_t\|_{L^2(\Omega)},
\]
and
\[
\|\vartheta_t\|_{L^2(\Omega)}
\le \|\rD_t\vartheta\|_{L^2(\Omega)}
+\|\u\|_{L^\infty(\Omega)}\|\nabla\vartheta\|_{L^2(\Omega)}.
\]
Consequently,
\begin{align*}
\int_0^T \|\u_{tt}&\|^2_{H^{-1}(\Omega)} \dt\\& \preceq\mathrm{1}\big(\text{data},\|\nabla \u\|_{L^1(0,T;L^\infty(\Omega))}, \|\vartheta\|_{L^{5+\delta}(0,T;L^\infty(\Omega))}, \|\f\|_{H^1(Q_T)}, \|g\|_{L^2(Q_T)}\big).
\end{align*}
Combining all the estimates above, we conclude that
\begin{equation}\label{u_full_estimate}
\|\u\|^2_{X_T} \preceq\mathrm{\mathrm{1}}\big(\text{data},\|\nabla \u\|_{L^1(0,T;L^\infty(\Omega))}, \|\vartheta\|_{L^{5+\delta}(0,T;L^\infty(\Omega))}, \|\f\|_{H^1(Q_T)}, \|g\|_{L^2(Q_T)}\big).
\end{equation}
The estimate above also yields the $L^\infty_t H^1_x$-estimate for $\varrho_t$:
$$
\|\varrho_t\|_{H^1(\Omega)} \le \|\div (\varrho \u)\|_{H^1(\Omega)} \preceq \|\varrho\|_{H^2(\Omega)} \|\u\|_{H^2(\Omega)}
$$
so we have
\begin{align}\label{rhot_LinftyH1_estimate}
\sup_{t\in [0,T]} &\|\varrho_t(t)\|_{H^1(\Omega)} \\&\preceq\mathrm{1}\big(\text{data},\|\nabla \u\|_{L^1(0,T;L^\infty(\Omega))}, \|\vartheta\|_{L^{5+\delta}(0,T;L^\infty(\Omega))}, \|\f\|_{H^1(Q_T)}, \|g\|_{L^2(Q_T)}\big).
\end{align}

\subsubsection{\texorpdfstring{$L^\infty_t L^\infty_x$}{LᵒᵒₜLᵒᵒₓ} and \texorpdfstring{$L^2_t H^3_x$}{L²ₜH³ₓ}-estimates for \texorpdfstring{$\vartheta$}{θ}}
Having obtained \eqref{u_full_estimate}, by the Sobolev embedding $H^2(\Omega) \contsubset L^\infty(\Omega)$ we find that instead of $\nabla \u \in L^1(0,T;L^\infty(\Omega))$ we indeed have $\nabla \u \in L^2(0,T;L^\infty(\Omega))$ and we have
\begin{align*}
\int_0^T &\|\nabla \u\|^2_{L^\infty(\Omega)} \dt
\\&\preceq\mathrm{1}\big(\text{data},\|\nabla \u\|_{L^1(0,T;L^\infty(\Omega))}, \|\vartheta\|_{L^{5+\delta}(0,T;L^\infty(\Omega))}, \|\f\|_{H^1(Q_T)}, \|g\|_{L^2(Q_T)}\big).
\end{align*}
Following \cite{Hu2009}, by testing the temperature equation \eqref{temperature_eq} against $\vartheta^{q+1}$ we obtain that
\begin{align*}
\frac{1}{q+2} \frac{d}{dt} \int_\Omega \varrho\hpt\vartheta^{q+2} \dx &+ \frac{1}{q+2} \int_{\partial\Omega} \varrho (\u\cdot \bN) \vartheta^{q+2} \dS - \kappa \int_\Omega \Delta \vartheta \cdot \vartheta^{q+1} \dx \\
& + \int_\Omega \varrho\hpt\vartheta^{q+2} \div \u \dx= - \int_\Omega |\nabla \u|^2 \vartheta^{q+1} \dx + \int_\Omega \varrho\hpt g \vartheta^{q+1}\dx\pd
\end{align*}
Let $f(t) = \smallint{\Omega}{} \varrho\hpt\vartheta^{q+2}\dx$. Since the additional boundary integral can be estimated by
$$
- \smallint{\Gin}{} \varrho(\u\cdot \bN) |\vartheta|^{q+2} dS \preceq\|\rB\|_{L^\infty(\Omega)} \|\uB\|_{L^\infty(\Omega)} \|\tB\|_{L^\infty(\Omega)}^{q+2}\scl
$$
we have
\begin{equation}\label{differential_ineq_for_theta}
\frac{d}{dt} f(t) \preceq(q+2)\big(1 + \|\u\|^2_{L^\infty(\Omega)}\big) f(t) + \|\rB\|_{L^\infty(\Omega)} \|\uB\|_{L^\infty(\Omega)} \|\tB\|_{L^\infty(\Omega)}^{q+2}
\end{equation}
for the constant depending on $\Omega$ but independent of $q$. Therefore, the differential inequality \eqref{differential_ineq_for_theta} implies that
\begin{align*}
\frac{d}{dt}& \Big[e^{-C (q+2) \text{\tiny$\displaystyle{}\int_0^t$}\,(1 + \|\nabla \u\|^2_{L^\infty(\Omega)}) dt'} f(t)\Big] \\&\qquad\qquad\preceq e^{-C (q+2) \text{\tiny$\displaystyle{}\int_0^t$}\,(1 + \|\nabla \u\|^2_{L^\infty(\Omega)}) dt'}  \|\rB\|_{L^\infty(\Omega)} \|\uB\|_{L^\infty(\Omega)} \|\tB\|_{L^\infty(\Omega)}^{q+2}
\end{align*}
so
$$
f(t)^\frac{1}{q+2} \preceq e^{C \text{\tiny$\displaystyle{}\int_0^t$}\,(1 + \|\nabla \u\|^2_{L^\infty(\Omega)}) dt'} \Big[f(0)^\frac{1}{q+2} + \|\tB\|_{L^\infty(\Omega)} \Big(\|\rB\|_{L^\infty(\Omega)} \|\uB\|_{L^\infty(\Omega)}\Big)^\frac{1}{q+2}\Big].
$$
By Lemma \ref{lem:lower_upper_bound_for_rho}, it follows that
\begin{align*}
\min\big\{\rol,\rBl\big\}& \exp\Big(\!-\! \int_0^t \|\div \u\|_{L^\infty(\Omega)}\dt'\Big)\\& \le \varrho(t,x) \le \max\big\{\rou,\rBu\big\} \exp\Big(\int_0^t \|\div \u\|_{L^\infty(\Omega)}\dt'\Big)\cm
\end{align*}
so passing to the limit as $q \to \infty$, we conclude that
\begin{equation}\label{theta_Linfty_estimate}
\|\vartheta\|_{L^\infty(Q_T)} \preceq\|\vartheta_0\|_{L^\infty(\Omega)} + \|\tB\|_{L^\infty(\Omega)}
\end{equation}
for the constant depending on the data, $\|\nabla \u\|_{L^1(0,T;L^\infty(\Omega))}$, $\|\vartheta\|_{L^{5+\delta}(0,T;L^\infty(\Omega))}$,\\ $\|\f\|_{H^1(Q_T)}$, and $\|g\|_{L^2(Q_T)}$.

We also apply the parabolic estimate (Theorem \ref{thm:linear_parabolic_blackbox}) for the case $\v = \vartheta$, $f = -\div \u$ and $\sbF = |\nabla \u|^2 + \varrho g$ to obtain the full regularity for $\vartheta$. Clearly $f\in L^\infty(0,T;H^1(\Omega))$ with $f_t \in L^\infty(0,T;H^{-1}(\Omega))$ and we have
$$
\|f\|^2_{L^\infty(0,T;H^1(\Omega))} + \|f_t\|^2_{L^\infty(0,T;H^{-1}(\Omega))} \preceq \|\u\|^2_{X_T}.
$$
Moreover, $\|\sbF\|_{H^1(\Omega)} \preceq\|\u\|_{H^2(\Omega)} \|\u\|_{H^3(\Omega)} + \|g\|_{H^1(\Omega)}$ and
\begin{align*}
\|\sbF_t\|_{H^{-1}(\Omega)} &\preceq\|\nabla \u_t\|_{L^6(\Omega)} \|\nabla \u\|_{L^3(\Omega)} + \|\varrho_t\|_{L^6(\Omega)} \|g\|_{L^3(\Omega)} + \|g_t\|_{H^{-1}(\Omega)} \\
&\preceq\|\|\u\|_{H^2(\Omega)} \|\u_t\|_{H^2(\Omega)} + \|\varrho_t\|_{H^1(\Omega)} \|g\|_{H^1(\Omega)} + \|g_t\|_{H^{-1}(\Omega)}\\
\end{align*}
so we have
\begin{align*}
& \int_0^T \Big[\|\sbF\|^2_{H^1(\Omega)} + \|\sbF_t\|^2_{H^{-1}(\Omega)} \Big] dt \\
&\qquad \preceq1 + \|\u\|^4_{X_T} + \big(1 + \|\varrho_t\|^2_{L^\infty(0,T;H^1(\Omega))}\big) \int_0^T \big(\|g\|^2_{H^1(\Omega)} + \|g_t\|^2_{H^{-1}(\Omega)}\big) dt.
\end{align*}
Therefore, using \eqref{rhot_LinftyH1_estimate} we conclude from Theorem \ref{thm:linear_parabolic_blackbox} that $\vartheta$ satisfies
\begin{equation}\label{theta_full_estimate}
\|\vartheta\|^2_{X_T}\preceq\mathrm{1}\big(\text{data},\|\nabla \u\|_{L^1(0,T;L^\infty(\Omega))}, \|\vartheta\|_{L^{5+\delta}(0,T;L^\infty(\Omega))}, \|\f\|_{H^1(Q_T)}, \|g\|_{L^2(Q_T)}\big).
\end{equation}
This finishes all the required estimate and establishes Theorem \ref{thm:main}.

\appendix
\section{Estimates of \texorpdfstring{$\displaystyle{}\int_\Gin (\u\cdot \bN) |\nabla^\ell \varrho|^k dS$ for various $k$ and $\ell$}{boundary integrals}}\label{sec:bdy_integral_estimate}
\subsection{The case \texorpdfstring{$\ell=1$}{ℓ=1} given boundary data with good regularity}\label{sec:bdy_integral_estimate1}
In this part of the appendix, we give a detailed derivation of the following
\begin{theorem}\label{thm:rho_boundary_integral_estimate}
Let $\partial\Omega$ be of class ${\mathscr{C}}^3$, and $\varrho$, $\u$ satisfies the equation of continuity \eqref{density_eq} together with the boundary condition \eqref{boundary_condition2}. Then
\begin{equation}\label{rho_boundary_integral_estimate}
\begin{array}{l}
\displaystyle{}- \int_{\Gin} (\u \cdot \bN) |\nabla \varrho|^k \dS \vspace{.15cm}\\
\displaystyle{}\qquad \preceq\big(1 + \|\uB\|^k_{L^\infty(\Gin)}\big) \|\rB\|^k_{W^{1,k}(\bdy\Omega)} + \|\uB_t\|^k_{L^k(\bdy\Omega)} + \|\rB\|^k_{L^\infty(\Gin)} \|\div \u\|^k_{L^k(\Gin)}
\end{array}
\end{equation}
for the constant depending on $k$, $\uBl \equiv \inf\limits_{\Gin} (-\uB\cdot \bN)$ and $\Omega$.
\end{theorem}
\begin{proof}
By our assumption on $\partial\Omega$, $\Gin$ is a compact manifold without boundary. Therefore, for there exist a collection of open sets $\{\U_m\}_{m=0}^M$ with each $\U_m \subseteq  \bbR^\n$, a collection of ${\mathscr{C}}^3$-maps $\{\phi_m\}_{m=1}^M$, and a collection of positive numbers $\{r_m\}_{m=1}^M$ such that
$$
\Gin \subseteq \bigcup\limits_{m=1}^M \U_m\,,
$$
and for each $1 \le m\le M$,
\begin{enumerate}
\item $\phi_m: B(0, r_m) \to \U_m \text{ is a } {\mathscr{C}}^3\text{-diffeomorphism}$; that is, $\phi_m: B(0, r_m) \to \U_m$ is bijective, $\phi_m \in {\mathscr{C}}^3(B(0,r_m))$, and $\phi_m^{-1} \in {\mathscr{C}}^3(\U_m)$;
\item $\phi_m: B(0,r_m) \cap \{y_n=0\} \to \U_m \cap \bdy\Omega $;
\item $\phi_m: B^+_m \equiv B(0,r_m) \cap \{y_\n>0\} \to \U_m \cap \Omega $;
\item $\det (\nabla \phi_m) =1 $.
\end{enumerate}
Let $\{\zeta_m\}_{m=1}^M$ be a partition-of-unity of $\Gin$ subordinate to $\{\U_m\}_{m=1}^M$. Then
$$
\int_\Gin (\u\cdot \bN) |\nabla \varrho|^k \dS = \sum_{m=1}^M \int_\Gin \zeta_m (\u\cdot \bN) |\nabla \varrho|^k \dS = \sum_{m=1}^M \int_{\U_m \cap \Gin} \zeta_m (\u\cdot \bN) |\nabla \varrho|^k \dS\pd
$$
For each $m$, let $\v_m = \u \circ \phi_m$, $\rho_m = \varrho \circ \phi_m$, where the composition is done in the spatial variables; that is,
$$
\v_m(\y,t) = \u(\phi_m(\y),t)\cm\qquad \rho_m(\y,t) = \varrho(\phi_m(\y),t)\pd
$$
Define the $\rg_m = \|\phi_m,_1 \times \phi_m,_2\|^2$ and $A = (\nabla \phi_m)^{-1}$. Then
$$
\|A\|_{L^\infty(B(0,r_m))} < \infty \qquad\text{and}\qquad \|\rg_m\|_{L^\infty(B(0,r_m) \cap \{y_\n = 0\})} < \infty\pd
$$
Since $\bN\circ \phi_m = A^\top \bfe_\n$, we have
\begin{align*}
& - \int_{\U_m \cap \Gin} \zeta_m (\u\cdot \bN) |\nabla \varrho|^k \dS \\
& = - \sum_{i=1}^\n \int_{B(0,r_m) \cap \{y_n = 0\}} (\zeta_m \circ \phi_m) \v_m^i A^\n_i \left(\sum_{j,r,s=1}^\n A_j^r A_j^s \rho_m,_r \rho_m,_s\right)^{\!\frac{k}{2}} \sqrt{\rg_m} \dy' \\
& \preceq - 3^{\frac{k}{2}-1} \sum_{i,j=1}^\n \int_{B(0,r_m) \cap \{y_\n = 0\}} (\zeta_m \circ \phi_m) \v_m^i A^\n_i \left(\sum_{r,s=1}^\n A_j^r A_j^s \rho_m,_r \rho_m,_s\right)^{\!\frac{k}{2}} \sqrt{\rg_m} \dy'\cm
\end{align*}
where $dy' = dy_1 dy_2$. When $r \ne \n$, we have $\rho_m,_r = (\rB \circ \phi_m),_r$, so the Jensen inequality implies that
\begin{alignat*}{2}
& - \int_{B(0,r_m) \cap \{y_n = 0\}} (\zeta_m \circ \phi_m) \v_m^i A^\n_i \left(\sum_{r,s=1}^\n A_j^r A_j^s \rho_m,_r \rho_m,_s\right)^{\!\frac{k}{2}} \sqrt{\rg_m} \dy' \\
&\qquad \preceq - 3^{\frac{k}{2}-1} \int_{B(0,r_m) \cap \{y_n = 0\}} (\zeta_m \circ \phi_m) \v_m^i A^\n_i \left(\sum_{r,s=1}^2 A_j^r A_j^s \rho_m,_r \rho_m,_s\right)^{\!\frac{k}{2}} \sqrt{\rg_m} \dy'  \\
&\qquad\quad - 3^{\frac{k}{2}-1} \int_{B(0,r_m) \cap \{y_n = 0\}} (\zeta_m \circ \phi_m) \v_m^i A^\n_i \left(\sum_{r=1}^2 A_j^r A_j^3 \rho_m,_r \rho_m,_3\right)^{\!\frac{k}{2}} \sqrt{\rg_m} \dy' ) \\
&\qquad\quad - 3^{\frac{k}{2}-1} \int_{B(0,r_m) \cap \{y_n = 0\}} (\zeta_m \circ \phi_m) \v_m^i A^\n_i \left(A_j^3 A_j^3 \rho_m,_3 \rho_m,_3\right)^{\!\frac{k}{2}} \sqrt{\rg_m} \dy' \\
&\qquad:=\sum_{j=1}^3 I_j.
\end{alignat*}
We remark here that for the case $k=2$ the estimates of $I_1$ and $I_3$ corresponds to the estimates for the integral
$$
\smallint{\bbT^2\times \{x_\n = 1\}}{} (\u\cdot \bN) |\varrho,_k\!|^2 dS, \qquad k =1,2
$$
and the integral
$$
\smallint{\bbT^2\times \{x_\n = 1\}}{} (\u\cdot \bN) |\varrho,_\n\!|^2 dS
$$
in Section \ref{sec:H1_estimate_of_rho}, respectively. In the general case, there is one additional term, the cross term $I_2$, that is required to be estimated.

Since $\varrho$ satisfy the equation of continuity \eqref{density_eq}, we find that $\rho_m$ satisfies
$$
\partial_t \rho_m + \rho_m (\div \u)\circ\phi_m + \sum_{j,k=1}^3 \v_m^j A_j^k \rho_m,_r  = 0 \qquad\text{in}\quad B(0,r_m) \scl
$$
thus on $B(0,r_m) \cap \{y_\n = 0\}$,
\begin{align*}
\rho_m,_\n \hspace{-7pt}&= - \frac{1}{\sum\limits_{j=1}^3 \v_m^j A_j^3} \Big(\partial_t \rho_m + \rho_m (\div \u)\circ\phi_m + \sum_{j=1}^3 \sum_{k=1}^2 \v_m^j A_j^k \rho_m,_r\Big) \vspace{-.2cm}\\
&=- \frac{1}{\sum\limits_{j=1}^3 \v_m^j A_j^3} \Big[(\rB_t \circ\phi_m) + (\rB\circ \phi_m) (\div \u)\circ\phi_m \\&\qquad\qquad+ \sum_{j=1}^3 \sum_{k=1}^2 (\uB \circ \phi_m)^j A_j^k (\rB \circ \phi_m),_r \Big],
\end{align*}

here we have used the fact that $\sum\limits_{j=1}^\n \v_m^j A^\n_j = (\u \cdot \bN) \cdot \phi_m \le - \uBl < 0$ to show the identity.
Therefore, using the fact that
\begin{equation}\label{rho,m_on_boundary}
|\rho_m,_\n\!|  \preceq\Big[|\rB_t| + |\uB||\nabla \rB \times \bN| + |\rB| |(\div \u)|\Big]\circ \phi_m \qquad\text{on}\quad B(0,r_m) \cap \{y_\n =0\}
\end{equation}
for the constant depending on $\uBl$, where $\nabla \rB \times \bN$ denotes the tangential derivative of $\rB$. By H\"older's inequality,
\begin{align*}
\rI_1 + \rI_2 + \rI_3 &\preceq \sum_{r=1}^\n \int_{B(0,r_m) \cap \{y_\n = 0\}} \|\uB\|_{L^\infty(\Gin)} \big|\varrho_m,_r\!\big|^k \dS \\
&\preceq \|\uB\|_{L^\infty(\Gin)}  \Big[\|\rB\|^k_{W^{1,k}(\Gin)} + \big\|\varrho_m,_\n\big\|^k_{L^k(B(0,r_m) \cap \{y_\n = 0\})}\Big],
\end{align*}
and the desire estimate \eqref{rho_boundary_integral_estimate} following from \eqref{rho,m_on_boundary}.
\end{proof}

\subsection{The case \texorpdfstring{$k=\ell=2$}{k=ℓ=2} given estimates from Section \ref{sec:L2H2_estimate_of_utheta}}\label{sec:bdy_integral_estimate2}
We only demonstrate the idea by setting $\Omega = \bbT^2 \times (0,1)$, and the general estimate is done by partition-of-unity and straighten the inflow boundary.\vspace{.1cm}

The estimate of $- \smallint{\Gin}{} (\u \cdot \bN) |\varrho,_{jk}\!|^2\dS$ is proceeded as follows. First we note that for $j,k \ne 3$,\vspace{-.1cm}
\begin{equation}\label{D2_inflow_boundary_estimate_temp1}
- \int_\Gin (\u \cdot \bN) |\varrho,_{jk}\!|^2\dS \le \uBu \|\rB\|^2_{H^2(\Gin)}.
\end{equation}
Suppose that $j = 3$, we differentiate the equation \eqref{normal_derivative_in_terms_of_tangential_derivative1} w.r.t. $x_k$ and obtain
\begin{equation}\label{two_spatial_derivative_of_rho}
\begin{array}{rl}
\varrho,_{3k}
\hspace{-7pt}&=\displaystyle{} - \frac{1}{u_3} \big(\partial_t \varrho,_k + u_1,_k \varrho,_{1} + u_2,_k \varrho,_{2} + \varrho,_k \div \u + u_3 \varrho,_{1k} + u_2 \varrho,_{2k} + \varrho\, \div \u,_k \big) \vspace{.1cm}\\
&\quad\displaystyle{} + \frac{u_3,_k}{u_3^2} \big(\partial_t \varrho + u_3 \varrho,_{1} + u_2 \varrho,_{2} + \varrho\, \div \u\big).
\end{array}
\end{equation}
Therefore, for $k=1,2$ we have
\begin{align*}
& - \int_\Gin (\u \cdot \bN) |\varrho,_{3k}\!|^2\dS \\
&\qquad \le - \int_\Gin \frac{2}{u_3} \big(\partial_t \varrho,_k + u_1,_k \varrho,_{1} + u_2,_k \varrho,_{2} + \varrho,_k \div \u + u_3 \varrho,_{1k} + u_2 \varrho,_{2k} + \varrho\, \div \u,_k \big)^2 dS \\
&\qquad\quad - \int_\Gin \frac{2 |u_3,_k\!|^2}{u^3_3} \big(\partial_t \varrho + u_3 \varrho,_{1} + u_2 \varrho,_{2} + \varrho\, \div \u\big)^2 dS
\end{align*}
and the Jensen inequality further shows that
\begin{align*}
& - \int_\Gin (\u \cdot \bN) |\varrho,_{3k}\!|^2\dS \\
&\qquad \preceq- \int_\Gin \frac{14}{u_3} \Big[|\partial_t \varrho,_k\!|^2 + |u_1,_k\!|^2 |\varrho,_{1}|^2 + |u_2,_k\!|^2 |\varrho,_{2}|^2 + |\varrho,_k\!|^2 |\div \u|^2 + |u_3|^2 |\varrho,_{1k}|^2 \\
&\qquad\qquad\qquad\quad + |u_2|^2 |\varrho,_{2k}|^2 + |\varrho|^2 |\div \u,_k\!|^2 \Big] dS \\
&\qquad\quad - \int_\Gin \frac{8 |u_3,_k\!|^2}{u^3_3} \Big[|\partial_t \varrho|^2 + |u_3|^2 |\varrho,_{1}|^2 + |u_2|^2 |\varrho,_{2}|^2 + |\varrho|^2 |\div \u|^2 \Big] dS\\
&\qquad \preceq \frac{14}{\uBl} \Big(\|\partial_t \rB\|^2_{H^1(\Gin)} + 2 \|\uB\|^2_{W^{1,4}(\Gin)} \|\rB\|^2_{W^{1,4}(\Gin)} + 2\u^2_{{}_B} \|\rB\|^2_{H^2(\Gin)}  \\&\qquad\qquad\qquad\qquad + \rBu^2 \|\nabla^2 \u\|^2_{L^2(\partial\Omega)} \Big)\\
&\qquad\qquad + \frac{8}{\underline{\u^3_{{}_B}}} \|\uB\|^2_{W^{1,4}(\Gin)} \Big(\|\partial_t \rB\|^2_{L^4(\Gin)} + 2 \|\uB\|^2_{L^\infty(\Gin)} \|\rB\|^2_{W^{1,4}(\Gin)}\\& \qquad\qquad\qquad\qquad+ \rBu^2 \|\div \u\|^2_{L^4(\partial\Omega)}\Big)
\end{align*}
and \eqref{interpolation_ineq} further implies that for $k=1,2$ and any $\varepsilon > 0$,
\begin{align}
- \int_\Gin (\u \cdot \bN) |\varrho,_{3k}\!|^2\dS &\preceq1 + \|\nabla^2 \u\|^2_{L^2(\partial\Omega)}
\preceq1 + \|\nabla^2 \u\|_{L^2(\Omega)} \|\nabla^2 \u\|_{H^1(\Omega)} \nonumber\\
&\preceq1 + \|\nabla\u\|^\frac{1}{2}_{L^2(\Omega)} \|\u\|^\frac{3}{2}_{H^3(\Omega)}\nonumber\\&\preceq1 + \|\nabla \u\|^2_{L^2(\Omega)} + \|\u\|^2_{H^3(\Omega)} \nonumber\\
&\preceq1 + \|\f\|^2_{H^1(Q_T)} + \|g\|^2_{L^2(Q_T)} + \|\u\|^2_{H^3(\Omega)} \label{D2_inflow_boundary_estimate_temp2}
\end{align}
for some constant $C$ depending on data.\vspace{.1cm}

We now estimate $- \smallint{\Gin}{} (\u \cdot \bN) |\varrho,_{33}\!|^2\dS$. Using \eqref{two_spatial_derivative_of_rho} and the Jensen inequality,
\begin{align*}
& - \int_\Gin (\u \cdot \bN) |\varrho,_{33}\!|^2\dS \\
&\preceq - \int_\Gin \frac{2 (\u \cdot \bN)}{u^2_3} \big(\partial_t \varrho,_3 + u_1,_3 \varrho,_{1} + u_2,_3 \varrho,_{2} + \varrho,_3 \div \u + u_1 \varrho,_{13} + u_2 \varrho,_{23} \\
&\qquad\quad+ \varrho\, \div \u,_3 \big)^2 dS  - \int_\Gin \frac{2 (\u \cdot \bN) |u_3,_3|^2}{u^4_3} \big(\partial_t \varrho + u_1 \varrho,_{1} + u_2 \varrho,_{2} + \varrho\, \div \u\big)^2 dS\\
&\preceq- \frac{14}{\underline{\u^2_{{}_B}}} \int_\Gin (\u \cdot \bN) \Big[|\partial_t \varrho,_3|^2 + |\varrho,_3|^2 |\div \u|^2 + |u_1|^2 |\varrho,_{13}|^2 + |u_2|^2 |\varrho,_{23}|^2\Big] dS \\
&\qquad+ \frac{14}{\uBl} \Big(2 \|\nabla \u\|^2_{L^4(\partial\Omega)} \|\rB\|^2_{W^{1,4}(\Gin)} + \rBu^2 \|\nabla^2 \u\|^2_{L^2(\partial\Omega)} \Big) \\
&\qquad + \frac{8 \|\nabla \u\|^2_{L^4(\partial\Omega)}}{\underline{\u^3_{{}_B}}} \Big(\|\partial_t \rB\|^2_{L^4(\Gin)} + 2 \|\uB\|^2_{L^\infty(\Gin)} \|\rB\|^2_{W^{1,4}(\Gin)} + \rBu^2 \|\div\u\|^2_{L^4(\partial\Omega)}\Big) \\
&\preceq- \frac{14}{\underline{\u^2_{{}_B}}} \int_\Gin (\u \cdot \bN) \Big[|\partial_t \varrho,_3|^2 + |\varrho,_3|^2 |\div \u|^2 + \uBu^2 |\varrho,_{13}|^2 + \uBu^2 |\varrho,_{23}|^2\Big] dS \\
&\qquad\quad + 1 + \|\nabla \u\|^4_{L^4(\partial\Omega)}.
\end{align*}
Using \eqref{normal_derivative_in_terms_of_tangential_derivative1},
\begin{align*}
\partial_t \varrho,_3 &= \frac{1}{u_3} \big(\varrho_{tt} + u_{1t} \varrho,_{1} + u_{2t} \varrho,_{2} + \varrho_t\, \div \u + u_1 \varrho_t,_{1} + u_2 \varrho_t,_{2} + \varrho\, \div \u_t \big) \\
&\quad - \frac{u_{3t}}{u_3^2} \big(\varrho_t + u_1 \varrho,_{1} + u_2 \varrho,_{2} + \varrho\, \div \u\big)\scl
\end{align*}
thus the fact that $\partial_t \uB = {\bf 0}$
shows that
\begin{align}
& -\int_\Gin (\u \cdot \bN) |\partial_t \varrho,_3|^2 dS = -\int_\Gin \frac{1}{u_3} \big(\varrho_{tt} + \varrho_t\, \div \u + u_1 \varrho_t,_{1} + u_2 \varrho_t,_{2} + \varrho\, \div \u_t \big)^2 dS \nonumber\\
&\qquad \preceq - \frac{5}{\uBu} \int_\Gin \big(|\rB_{tt}|^2 + |\rB_t|^2 |\div \u|^2 + |u_1|^2 |\rB_t,_1|^2 + |u_2|^2 |\rB_t,_2|^2  \nonumber\\
&\qquad\qquad\qquad+  \rBu^2 |\div \u_t|^2 \big) dS\nonumber\\
&\qquad\preceq1 + \|\div \u\|^2_{L^2(\partial\Omega)} + \|\div \u_t\|^2_{L^2(\partial\Omega)} \nonumber\\
&\qquad \preceq1 + \|\div \u\|_{L^2(\Omega)} \|\div \u\|_{H^1(\Omega)} + \|\div \u_t\|_{L^2(\Omega)} \|\div \u_t\|_{H^1(\Omega)}\nonumber\\
&\qquad \preceq1 + \|\u\|^2_{H^2(\Omega)} + \|\nabla \u\|^4_{L^4(\Omega)} + \|\rD_t \u\|^2_{H^1(\Omega)}, \label{D2_inflow_boundary_estimate_temp3}
\end{align}
where we have used
$$
\|\u_t\|_{H^1(\Omega)} \preceq \|\rD_t \u\|_{H^1(\Omega)} + \|\u\|_{L^\infty(\Omega)} \|\u\|_{H^2(\Omega)} + \|\nabla \u\|^2_{L^4(\Omega)}
$$
to deduct the last inequality.
Moreover, \eqref{interpolation_ineq} implies that
\begin{align}
& -\int_\Gin (\u \cdot \bN) |\varrho,_3|^2 |\div \u|^2 \dS\nonumber\\&\qquad = - \int_\Gin (\u \cdot \bN) \Big[\frac{1}{u_3^2} \big(\partial_t \varrho + u_1 \varrho,_{1} + u_2 \varrho,_{2} + \varrho\, \div \u\big)^2 |\div \u|^2\Big] dS \nonumber\\
&\qquad\preceq - \frac{4}{\uBl} \int_\Gin \Big[|\rB_t|^2 + \uBu^2 \big(|\rB,_{1}|^2 + |\rB,_{2}|^2\big) + |\rB|^2 |\div \u|^2\Big] |\div \u|^2 dS \nonumber\\
&\qquad \preceq1 + \|\nabla \u\|^4_{L^4(\partial\Omega)}.
\label{D2_inflow_boundary_estimate_temp4}
\end{align}
Therefore, using \eqref{D2_inflow_boundary_estimate_temp2}, \eqref{D2_inflow_boundary_estimate_temp3} and \eqref{D2_inflow_boundary_estimate_temp4} we conclude that
\begin{align*}
- \int_\Gin (\u \cdot \bN) |\varrho,_{33}\!|^2\dS &\preceq1 + \|\nabla \u\|^4_{L^4(\Omega)} + \|\rD_t \u\|^2_{H^1(\Omega)} + \|\u\|^2_{H^2(\Omega)} + \|\nabla \u\|^4_{L^4(\partial\Omega)} \\
&\qquad + \|\f\|^2_{H^1(Q_T)} + \|g\|^2_{L^2(Q_T)} + \|\u\|^2_{H^3(\Omega)}.
\end{align*}
By the continuous embedding $H^1(\Omega) \contsubset L^4(\partial\Omega)$ and interpolation,
$$
\|\nabla \u\|^4_{L^4(\partial\Omega)} \preceq \|\nabla \u\|^4_{H^1(\Omega)} \preceq \|\nabla \u\|^2_{L^2(\Omega)} \|\nabla \u\|^2_{H^2(\Omega)} \cm
$$
so using \eqref{closed_estimate} we obtain that
\begin{align}\label{D2_inflow_boundary_estimate_temp5}
&- \int_\Gin (\u \cdot \bN) |\varrho,_{33}\!|^2\dS\nonumber\\&\qquad\qquad \preceq1 + \|\nabla \u\|^4_{L^4(\Omega)} + \|\rD_t \u\|^2_{H^1(\Omega)} + \|\f\|^2_{H^1(Q_T)} + \|g\|^2_{L^2(Q_T)} + \|\u\|^2_{H^3(\Omega)} .
\end{align}
Combining \eqref{D2_inflow_boundary_estimate_temp1}, \eqref{D2_inflow_boundary_estimate_temp2} and \eqref{D2_inflow_boundary_estimate_temp5}, we conclude \eqref{D2_inflow_boundary_estimate}.

\section{Proof of the boundary-adapted material-derivative estimate}
\label{app:linearized_material_derivative_estimate}
This appendix records the algebra leading to Lemma~\ref{lem:linearized_material_derivative_estimate}.  It is kept outside the main text because it consists mainly of integrations by parts and product estimates.

\subsection{Derivation of the estimate}
Suppose that $\v \in X_T$ satisfies
\begin{subequations}\label{parabolic_eq_w_inhom_bdy_condition_with_f}
\begin{alignat}{2}
\brho \bD_t \v &= \nu \Delta \v + \sbF \qquad&&\text{in}\quad\Omega\times (0,T), \label{parabolic_eq_w_inhom_bdy_condition_with_f.1}\\
\v &= \v_0 &&\text{on}\quad\Omega \times \{t=0\},\\
\v &= \sbV &&\text{on}\quad\bdy\Omega \times (0,T),
\end{alignat}
\end{subequations}
where $\sbF$ and $\sbV$ have the regularity described in
Theorem~\ref{thm:linear_parabolic_blackbox}.
Then $\v$ also satisfies
\begin{align*}
\brho \bD_t^2 \v + [\brho_t + \div (\brho \bu)\big] \bD_t \v = \nu \Delta \v_t + \nu \div (\Delta \v \otimes \bu) + \bD_t \sbF + (\div \bu) \sbF
\end{align*}
in the $H^{-1}$-sense.
To see this, we need to show that
\begin{equation}\label{linear_Hopf_eq_weak_form}
\begin{array}{rl}
\langle \bD_t^2 \v, \brho \Varphi\rangle \hspace{-7pt}&+\displaystyle{} \int_\Omega \big[\brho_t + \div (\brho \bu)\big] (\bD_t \v \cdot \Varphi) dx + \nu \int_\Omega \nabla \v_t : \nabla \Varphi \dx \vspace{.15cm}\\
&=\displaystyle{} \nu \int_\Omega \div (\Delta \v \otimes \bu) \cdot \Varphi \dx + \int_\Omega (\div \bu) \sbF \cdot \Varphi\dx + \langle \bD_t\sbF, \Varphi\rangle
\end{array}
\end{equation}
for any $ \Varphi\in H^1_0(\Omega),$
where
$\bD_t^2 \v$ and $\bD_t \sbF$ are bounded linear functional on $H^1_0(\Omega)$ defined by
\begin{align}
&\langle \bD_t^2 \v, \Varphi\rangle = \langle \v_{tt}, \Varphi \rangle + \int_\Omega (\bu\cdot\nabla \v_t) \cdot \Varphi \dx + \int_\Omega \big[\bD_t (\bu \cdot \nabla \v)\big]\cdot \Varphi \dx \nonumber\\
&= \langle \v_{tt}, \Varphi \rangle + \int_\Omega (\bu\cdot\nabla \v_t) \cdot \Varphi \dx + \int_\Omega \brho (\bu_t \cdot \nabla \v) \cdot \Varphi \dx + \int_\Omega \brho (\bu \cdot \nabla \bD_t \v) \cdot \Varphi \dx \label{defn:Dt2v}
\end{align}
and
$$
\langle \bD_t\sbF, \Varphi\rangle = \langle \sbF_t, \Varphi\rangle + \int_\Omega (\bu \cdot \nabla \sbF) \cdot \Varphi\dx\pd
$$
The starting point is the weak formulation
\begin{equation}\label{time_derivative_of_weak_form}
\langle \v_{tt}, \brho \Varphi\rangle + \int_\Omega \brho_t \v_t \cdot \Varphi \dx + \int_\Omega (\brho \bu \cdot \nabla \v)_t \cdot \Varphi \dx = \nu \int_\Omega \nabla \v_t : \nabla \Varphi \dx + \langle \sbF_t, \Varphi\rangle
\end{equation}
for any $ \Varphi \in H^1_0(\Omega)$,
which is a direct consequence of \eqref{eq:linear_parabolic_weak}, applied to \eqref{parabolic_eq_w_inhom_bdy_condition_with_f} with $a_s=0$. From \eqref{time_derivative_of_weak_form}, to establish \eqref{linear_Hopf_eq_weak_form} we note that both sides of \eqref{parabolic_eq_w_inhom_bdy_condition_with_f.1} belongs to $H^1(\Omega)$ so we are allowed to take the spatial weak derivative and obtain that
$$
\bu \cdot \nabla (\brho \bD_t \v) = \nu \bu \cdot \nabla \Delta \v + \bu \cdot \nabla \sbF\qquad \text{in}\quad L^2(\Omega), \text{ for a.a. $t\in (0,T)$}\scl
$$
thus
$$
\int_\Omega \big[\bu \cdot \nabla (\brho \bD_t \v)\big]\cdot \Varphi \dx = \nu \int_\Omega (\bu \cdot \nabla \Delta \v) \cdot \Varphi \dx + \int_\Omega (\bu \cdot \nabla \sbF)\cdot \Varphi \dx \qquad\Forall \Varphi\in H^1_0(\Omega).
$$
Moreover,
$$
(\div \bu) \brho \bD_t \v = \nu (\div \bu)\Delta \v + (\div \bu) \sbF \qquad \text{in}\quad H^1(\Omega), \text{ for a.a. $t\in (0,T)$}\scl
$$
thus
$$
\int_\Omega (\div \bu) \brho \bD_t \v \cdot \Varphi \dx = \nu \int_\Omega (\div \bu) \Delta \v \cdot \Varphi \dx + \int_\Omega (\div \bu) \sbF \cdot \Varphi \dx \qquad\Forall \Varphi\in H^1_0(\Omega).
$$
Adding the two integral identities to \eqref{time_derivative_of_weak_form} and using \eqref{defn:Dt2v}, we obtain \eqref{linear_Hopf_eq_weak_form}.

Suppose that $\w \in L^2(0,T;H^1_0(\Omega))$ with $\w_t \in L^2(0,T;H^{-1}(\Omega))$. Then
$$
\frac{1}{2} \frac{d}{dt} \int_\Omega \brho |\w|^2 \dx - \frac{1}{2} \int_\Omega \brho_t |\w|^2 \dx = \langle \w_t, \brho\w \rangle = \langle \bD_t \w, \rho \w\rangle - \int_\Omega \rho (\bu \cdot \nabla \w) \cdot \w \dx\pd
$$
Since $\w = {\bf 0}$ on $\bdy\Omega$,
$$
\int_\Omega \brho (\bu \cdot \nabla \w) \cdot \w \dx = \frac{1}{2} \int_\Omega \brho \u \cdot \nabla |\w|^2 \dx = - \frac{1}{2} \int_\Omega \div (\brho \bu) |\w|^2 \dx\scl
$$
thus we conclude that if $\w \in L^2(0,T;H^1_0(\Omega))$ with $\w_t \in L^2(0,T;H^{-1}(\Omega))$, then
\begin{equation}\label{integration_of_meterial_derivative}
\langle \bD_t \w, \brho \w\rangle = \frac{1}{2} \frac{d}{dt} \int_\Omega \brho |\w|^2 \dx - \frac{1}{2} \int_\Omega \big[\brho_t + \div (\brho \bu)\big] |\w|^2 \dx\pd
\end{equation}

Let $\bfu$, $\bfv$, $\bfw$ be chosen so that $\w = \bD_t \v - \bfu - \bfv \cdot \nabla \bfw \in L^2(0,T;H^1_0(\Omega))$ with $\w_t \in L^2(0,T;H^{-1}(\Omega))$. Then \eqref{linear_Hopf_eq_weak_form} implies that $\w$ satisfies

\begin{align*}
& \langle \bD_t \w, \brho \Varphi\rangle + \nu \int_\Omega \nabla \v_t : \nabla \Varphi \dx + \int_\Omega \big[\brho_t + \div (\brho \bu)\big] (\bD_t \v \cdot \Varphi) \dx \\
& = \nu \int_\Omega \div (\Delta \v \otimes \bu) \cdot \Varphi \dx - \int_\Omega \brho \big[(\bD_t \bfv) \cdot \nabla \bfw + (\bu \otimes \bfv): \nabla^2 \bfw + \bfv \cdot \nabla \bfw_t\big]\cdot \Varphi \dx \\
&\qquad\quad + \langle \bD_t \sbF, \Varphi\rangle + \int_\Omega (\div \bu) \sbF \cdot \Varphi \dx - \langle \bD_t \bfu, \brho \Varphi\rangle \qquad\quad \Forall \Varphi \in H^1_0(\Omega).
\end{align*}
Letting $\Varphi = \w$ as a test function, using \eqref{integration_of_meterial_derivative} the fact that $\w = {\bf 0}$ on $\bdy\Omega$ implies that
\begin{equation}\label{two_time_differentiated_problem_weak_form_linear_eq_temp}
\begin{array}{l}
\displaystyle{} \frac{1}{2} \frac{d}{dt} \int_\Omega \brho |\w|^2 \dx + \nu \|\nabla \w\|^2_{L^2(\Omega)} - \nu \int_\Omega \nabla \big[\bu \cdot \nabla \v + \bfu + \bfv \cdot \nabla \bfw \big]\cdot \nabla \w \dx \vspace{.15cm}\\
\qquad\displaystyle{} + \int_\Omega \big[\brho_t + \div (\brho \bu)\big] \Big(\bD_t \v \cdot \w - \frac{1}{2} |\w|^2\Big) \dx = \nu \int_\Omega \div (\Delta \v \otimes \bu) \cdot \w \dx \vspace{.15cm}\\
\qquad\qquad\displaystyle{} - \int_\Omega \brho \big[(\bD_t \bfv) \cdot \nabla \bfw + (\bu \otimes \bfv): \nabla^2 \bfw + \bfv \cdot \nabla \bfw_t\big]\cdot \w \dx \vspace{.15cm}\\
\qquad\qquad\qquad\displaystyle{} + \langle \bD_t \sbF, \w\rangle + \int_\Omega (\div \bu) \sbF \cdot \w \dx - \langle \bD_t \bfu, \brho \w\rangle \pd
\end{array}
\end{equation}
Since $\w \in H^1_0(\Omega)$ and $\v,\bu \in X_{T^*}$,
\begin{align*}
&\int_\Omega \div (\Delta \v \otimes \bu) \cdot \w \dx = - \int_\Omega (\Delta \v \otimes \bu) : \nabla \w \dx \nonumber\\&\qquad= - \int_\Omega \big[(\v^i,_k \bu^j),_k - \v^i,_k \bu^j,_k\big] \w^i,_j \dx \nonumber\\
&\qquad = - \int_\Omega (\v^i,_k \bu^j),_j \w^i,_k \dx + \int_\Omega \v^i,_k \bu^j,_k \w^i,_j \dx \nonumber \\
&\qquad = - \int_\Omega (\div \bu) (\nabla \v : \nabla \w) \dx - \int_\Omega  \bu^j \v^i,_{jk} \w^i,_k \dx + \int_\Omega \v^i,_j \bu^k,_j \w^i,_k \dx \pd
\end{align*}
Moreover,
\begin{align*}
& \int_\Omega \nabla \big[\bu \cdot \nabla \v + \bfu + \bfv \cdot \nabla \bfw \big]\cdot \nabla \w \dx = \int_\Omega \big[\bu^j \v^i,_j + \bfu^i + \bfv^j \bfw^i,_j\big],_k \w^i,_k \dx \\
&\qquad = \int_\Omega \big[\bu^j,_k \v^i,_j + \bu^j \v^i,_{jk} + \bfu^i,_k + \bfv^j,_k \bfw^i,_j + \bfv^j \bfw^i,_{jk}\big] \w^i,_k \dx \pd
\end{align*}
Therefore, \eqref{two_time_differentiated_problem_weak_form_linear_eq_temp} implies that
\begin{equation}\label{two_time_differentiated_problem_weak_form_linear_eq}
\begin{array}{l}
\displaystyle{} \frac{1}{2} \frac{d}{dt} \int_\Omega \brho |\w|^2 \dx + \nu \|\nabla \w\|^2_{L^2(\Omega)} + \int_\Omega \big[\brho_t + \div (\brho \bu)\big] \Big(\bD_t \v \cdot \w - \frac{1}{2} |\w|^2\Big) \dx \vspace{.15cm}\\
\qquad\displaystyle{} = - \nu \int_\Omega (\div \bu) (\nabla \v : \nabla \w) \dx + \nu \int_\Omega (\bu^k,_j + \bu^j,_k) \v^i,_j \w^i,_k \dx \vspace{.15cm} \\
\qquad\quad\displaystyle{} + \nu \int_\Omega \big[\bfu^i,_k + \bfv^j,_k \bfw^i,_j + \bfv^j \bfw^i,_{jk}\big] \w^i,_k \dx\Big] \vspace{.15cm}\\
\qquad\quad\displaystyle{} - \int_\Omega \brho \big[(\bD_t \bfv) \cdot \nabla \bfw + (\bu \otimes \bfv): \nabla^2 \bfw + \bfv \cdot \nabla \bfw_t\big]\cdot \w \dx \vspace{.15cm}\\
\qquad\quad\displaystyle{} + \langle \bD_t \sbF, \w\rangle + \int_\Omega (\div \bu) \sbF \cdot \w \dx - \langle \bD_t \bfu, \brho \w \rangle \pd
\end{array}
\end{equation}
Using H\"older's inequality and the continuous embedding $H^1(\Omega) \contsubset L^6(\Omega)$, we find that
\begin{align*}
& \frac{1}{2} \frac{d}{dt} \int_\Omega \brho |\w|^2 \dx + \nu \|\nabla \w\|^2_{L^2(\Omega)}\!+\hspace{-1pt} \int_\Omega \big[\brho_t + \div (\brho \bu)\big] \Big(\bD_t \v \cdot \w - \frac{1}{2} |\w|^2\Big) \dx \\
&\quad \preceq \Big[\|(\div \bu) (\nabla \v)\|_{L^2(\Omega)}\!+\hspace{-1pt} \|(\Def \bu) (\nabla \v)\|_{L^2(\Omega)} \!+\hspace{-1pt} \|\bfu\|_{H^1(\Omega)}\!+\hspace{-1pt} \big(\|\nabla \bfv\|_{L^3(\Omega)}\\
&\qquad\quad +\hspace{-1pt} \|\bfv\|_{L^\infty(\Omega)}\big) \|\bfw\|_{H^2(\Omega)} + \|\brho\|_{L^\infty(\Omega)} \|\bD_t \bfv\|_{L^2(\Omega)} \|\nabla \bfw\|_{L^3(\Omega)} + \big(\|\nabla \brho\|_{L^3(\Omega)} \\
&\qquad  + \|\brho\|_{L^\infty(\Omega)}\big) \|\bD_t \bfu\|_{H^{-1}(\Omega)} \Big] \|\nabla \w\|_{L^2(\Omega)}\\&\qquad+  \|\brho\|_{L^\infty(\Omega)} \|\bfv\|_{L^\infty(\Omega)} \|\bu\|_{L^\infty(\Omega)} \|\bfw\|_{H^2(\Omega)} \|\w\|_{L^2(\Omega)} \\
&\qquad - \int_\Omega (\bfv \cdot \nabla \bfw_t) \cdot \w \dx + \langle \bD_t \sbF, \w\rangle + \int_\Omega (\div \bu) (\sbF \cdot \w) \dx
\end{align*}
and Young's inequality further shows that
\begin{align}
& \frac{1}{2} \frac{d}{dt} \int_\Omega \brho |\w|^2 \dx + \frac{3\nu}{4} \|\nabla \w\|^2_{L^2(\Omega)}\!+\hspace{-1pt} \int_\Omega \big[\brho_t + \div (\brho \bu)\big] \Big(\bD_t \v \cdot \w - \frac{1}{2} |\w|^2\Big) \dx \nonumber\\
&\quad \preceq\|(\div \bu) (\nabla \v)\|^2_{L^2(\Omega)}\!+\hspace{-1pt} \|(\Def \bu) (\nabla \v)\|^2_{L^2(\Omega)} \!+\hspace{-1pt} \|\bfu\|^2_{H^1(\Omega)}\nonumber\\&\qquad+\hspace{-1pt} \big(\|\nabla \bfv\|^2_{L^3(\Omega)}\!+\hspace{-1pt} \|\bfv\|^2_{L^\infty(\Omega)}\big) \|\bfw\|^2_{H^2(\Omega)} \nonumber\\
&\qquad+ \|\brho\|^2_{L^\infty(\Omega)} \|\bD_t \bfv\|^2_{L^2(\Omega)} \|\nabla \bfw\|^2_{L^3(\Omega)} + \big(\|\nabla \brho\|^2_{L^3(\Omega)} + \|\brho\|^2_{L^\infty(\Omega)}\big) \|\bD_t \bfu\|^2_{H^{-1}(\Omega)} \label{key_estimate}\\
&\qquad +  \|\brho\|^2_{L^\infty(\Omega)} \|\bu\|^2_{L^\infty(\Omega)} \|\w\|^2_{L^2(\Omega)} - \int_\Omega (\bfv \cdot \nabla \bfw_t) \cdot \w \dx + \langle \bD_t \sbF, \w\rangle \nonumber\\&\qquad+ \int_\Omega (\div \bu) (\sbF \cdot \w) \dx, \nonumber
\end{align}
where the constant depends only on $\Omega$ and $\nu$. We note that here we do not perform further estimates on $\|(\div \bu) (\nabla \v)\|^2_{L^2(\Omega)}$ and $\|(\Def \bu) (\nabla \v)\|^2_{L^2(\Omega)}$ because different $\v$ might require different approach.

There is one particular case of interest. In the following, we assume that $\brho_t + \div(\brho \bu) = 0$, and $\sbF = - \nabla \bp + \brho \f$, where $\bp \in L^2(0,T;H^2(\Omega))$, $\bp_t \in L^2(0,T;H^1(\Omega))$, $\f \in L^2(0,T;H^1(\Omega))$, and $f_t \in L^2(0,T;H^{-1}(\Omega))$. In this case, we have
$$
\langle \bD_t \sbF, \w\rangle + \int_\Omega (\div \bu) \sbF \cdot \w \dx = - \int_\Omega \big[\nabla \bp_t + \div (\nabla \bp \otimes \bu) \big] \cdot \w \dx + \langle \bD_t \f, \brho \w\rangle \pd
$$
Since $\w \in H^1_0(\Omega)$ and $\bu = \sbV$ on $\bdy\Omega$,
\begin{align*}
& \int_\Omega \big[\nabla \bp_t \hspace{-1pt}+\hspace{-1pt} \div (\nabla \bp \otimes \bu)\big]\cdot \w \dx \\&= -\hspace{-1pt} \int_\Omega \bp_t\hpt\div \w \dx \hspace{-1pt}-\hspace{-1pt} \int_\Omega \bp,_i (\bu-\sbV\,)^j \w^i\!,_j dx \hspace{-1pt}-\hspace{-1pt} \int_\Omega \bp,_i\hspace{-1pt}\sbV^{\hpt\,j} \w^i\!,_j dx \\
& = -\hspace{-1pt} \int_\Omega \bp_t\hpt\div \w \dx \hspace{-1pt}+\hspace{-1pt} \int_\Omega \bp (\bu \hspace{-1pt}-\hspace{-1pt} \sbV\,)^j\!,_i \w^i\!,_j dx \hspace{-1pt}+\hspace{-1pt} \int_\Omega \bp (\bu \hspace{-1pt}-\hspace{-1pt} \sbV\,)^j (\div \w),_j \dx \\&\qquad\qquad\qquad\hspace{-1pt}-\hspace{-1pt} \int_\Omega \bp,_i\hspace{-1pt}\sbV^{\hpt\,j} \w^i\!,_j dx \\
& = -\hspace{-1pt} \int_\Omega \bp_t\hpt\div \w \dx \hspace{-1pt}+\hspace{-1pt} \int_\Omega \bp (\bu \hspace{-1pt}-\hspace{-1pt} \sbV\,)^j\!,_i \w^i\!,_j dx \hspace{-1pt}-\hspace{-1pt} \int_\Omega \div \big[\bp (\bu \hspace{-1pt}-\hspace{-1pt} \sbV\,)\big] (\div \w)\dx \\&\qquad\qquad\qquad\hspace{-1pt}-\hspace{-1pt} \int_\Omega \bp,_i\hspace{-1pt}\sbV^{\hpt\,j} \w^i\!,_j dx \\
& = -\hspace{-1pt} \int_\Omega \big[\bp_t \hspace{-1pt}+\hspace{-1pt} \div (\bp\bu)\big] \div \w \dx \hspace{-1pt}+\hspace{-1pt} \int_\Omega \bp \bu^j\!,_i \w^i\!,_j dx \hspace{-1pt}+\hspace{-1pt} \int_\Omega \div (\bp\hpt\sbV\,) (\div \w)\dx\\&\qquad\qquad\qquad \hspace{-1pt}-\hspace{-1pt} \int_\Omega (\bp\hpt\sbV^{\hpt\,j}),_i \w^i\!,_j dx
\end{align*}
and the fact that $\smallint{\Omega}{}\, (\bp\hpt\sbV^{\hpt\,j}),_i \w^i,_j dx = - \smallint{\Omega}{}\, (\bp\hpt\sbV^{\hpt\,j}),_{ij} \w^i \dx = \smallint{\Omega}{}\, \div (\bp\hpt\sbV\,) (\div \w) \dx$ shows that

$$
\int_\Omega \big[\nabla \bp_t + \div (\nabla \bp \otimes \bu)\big]\cdot \w \dx =  - \int_\Omega \big[\bp_t + \div (\bp\bu)\big] (\div \w) \dx + \int_\Omega \bp \u^j\!,_i \w^i\!,_j dx \pd
$$

Therefore, for $\sbF = -\nabla \bp + \brho \f$ with given regularity on $\bp$ and $\f$, \eqref{key_estimate} implies that

\begin{align}
& \frac{1}{2} \frac{d}{dt} \int_\Omega \brho |\w|^2 \dx + \frac{\nu}{2} \|\nabla \w\|^2_{L^2(\Omega)} \le C \Big[\big(\|\div \bu\|^2_{L^4(\Omega)} + \|\Def \bu\|^2_{L^4(\Omega)}\big) \|\nabla \v\hpt\|^2_{L^4(\Omega)}  \nonumber\\
&\quad\ + \|\bfu\|^2_{H^1(\Omega)} +\big(\|\nabla \bfv\|^2_{L^3(\Omega)}+\|\bfv\|^2_{L^\infty(\Omega)}\big) \|\bfw\|^2_{H^2(\Omega)}\nonumber\\&\qquad\qquad\qquad + \|\brho\|^2_{L^\infty(\Omega)} \|\bD_t \bfv\|^2_{L^2(\Omega)} \|\nabla \bfw\|^2_{L^3(\Omega)} \nonumber\\
&\quad\ + \big(\|\nabla \brho\|^2_{L^3(\Omega)} + \|\brho\|^2_{L^\infty(\Omega)}\big) \big(\|\bD_t \bfu\|^2_{H^{-1}(\Omega)} + \|\f_t\|^2_{H^{-1}(\Omega)} \,+ \|\f\|^2_{H^1(\Omega)} \big) \label{key_estimate_for_u_appendix}\\
&\quad\ + \big\|\bp_t + \div (\bp\bu)\big\|^2_{L^2(\Omega)} + \|\bp\|^2_{L^\infty(\Omega)} \|\nabla \u\|^2_{L^2(\Omega)} + \|\brho\|^2_{L^\infty(\Omega)} \|\bu\|^2_{L^\infty(\Omega)} \|\w\|^2_{L^2(\Omega)} \Big] \nonumber\\
&\quad\ \,- \int_\Omega (\bfv \cdot \nabla \bfw_t) \cdot \w \dx \pd \nonumber
\end{align}

\section{The self-map estimate in the fixed-point construction}
\label{app:fixed_point_self_map_estimate}
This appendix proves Proposition~\ref{prop:fixed_point_self_map_estimate}.  We
keep only the estimates needed for the self-map property; constants denoted by
$C$ may change from line to line but are independent of $M$ after $T^*$ has been
chosen sufficiently small.

\begin{proof}[Proof of Proposition~\ref{prop:fixed_point_self_map_estimate}]
Let $\bu\in C_{T^*}(M)$ be fixed.  By Lemma~\ref{lem:lower_upper_bound_for_rho}
and the inflow boundary condition, choosing $T^*=T^*(M)$ sufficiently small gives
\begin{equation}\label{rho_bounded_by_a_generic_constant}
 c_0:=\frac12\min\{\rol,\rBl\}
 \le \brho \le
 \frac32\max\{\rou,\rBu\}=:c_1
 \qquad\text{in }\Omega\times(0,T^*).
\end{equation}

We next record the density estimates.  Differentiating
\eqref{linear_density_eq.1}, testing by $|\brho,_k|\brho,_k$, and using the
inflow boundary estimate from Appendix~\ref{sec:bdy_integral_estimate} yield
\begin{equation}\label{Drho_LinftyL3_linear_estimate}
\sup_{t\in[0,T^*]}\|\nabla\brho(t)\|_{L^3(\Omega)}^3
\le
C\bigl(\|\nabla\varrho_0\|_{L^3(\Omega)}^3+1\bigr),
\end{equation}
provided $T^*=T^*(M)$ is small enough.  The same second-order transport estimate
as in Section~\ref{sec:H2_estimate_of_rho}, together with
\eqref{rho_bounded_by_a_generic_constant}, gives
\begin{equation}\label{rho_H2_H1t_linear_estimate}
\sup_{t\in[0,T^*]}
\Bigl(
\|\brho(t)\|_{H^2(\Omega)}^2
+
\|\brho_t(t)\|_{H^1(\Omega)}^2
\Bigr)
\le C(M).
\end{equation}

For the temperature equation \eqref{linear_temperature_eq}, we apply
Theorem~\ref{thm:linear_parabolic_blackbox} with
\[
 a_s=-\div\bu,
 \qquad
 F=|\nabla\bu|^2+\brho g,
 \qquad
 V=\tB .
\]
The assumptions on $\bu\in C_{T^*}(M)$ imply
$|\nabla\bu|^2\in L^2(0,T^*;H^1(\Omega))$.  Moreover,
\[
\|\partial_t|\nabla\bu|^2\|_{H^{-1}(\Omega)}
\le
C\|\bu_t\|_{L^2(\Omega)}\|\bu\|_{H^2(\Omega)}^{1/2}
      \|\bu\|_{H^3(\Omega)}^{1/2},
\]
so
\[
\int_0^{T^*}\|\partial_t|\nabla\bu|^2\|_{H^{-1}(\Omega)}^2\,dt
\le C\sqrt{T^*}\,M^2.
\]
Similarly, using \eqref{rho_H2_H1t_linear_estimate},
\[
\|\partial_t(\brho g)\|_{H^{-1}(\Omega)}
\le
C\bigl(\|\brho_t\|_{L^2(\Omega)}\|g\|_{H^1(\Omega)}
+
\|\brho\|_{H^2(\Omega)}\|g_t\|_{H^{-1}(\Omega)}\bigr).
\]
Thus the parabolic regularity theorem gives
\begin{equation}\label{theta_XTstar_linear_estimate}
\|\btheta\|_{X_{T^*}}^2\le F(M)
\end{equation}
for some nondecreasing function $F$.

It remains to estimate the image $\u=\Phi(\bu)$.  Since
$\bu\in C_{T^*}(M)$, after decreasing $T^*$ if necessary we have
\begin{equation}\label{Dbu_LinftyL2_linear_estimate}
\sup_{t\in[0,T^*]}\|\bu(t)\|_{H^1(\Omega)}^2
\le 2\bigl(\|\u_0\|_{H^1(\Omega)}^2+1\bigr).
\end{equation}
Let
\[
 \w:=\bD_t^s\u-\sbV_t-\sbV\cdot\nabla\u .
\]
Then $\w\in H^1_0(\Omega)$, and Lemma~\ref{lem:linearized_material_derivative_estimate},
applied with $(\bfu,\bfv,\bfw)=(\sbV_t,\sbV,\u)$ and
$\bp=\brho\btheta$, gives the boundary-adapted material-derivative estimate.
After estimating the remaining term
$\int_\Omega(\sbV\cdot\nabla\u_t)\cdot\w\,dx$ by integration by parts and using
\eqref{Dbu_LinftyL2_linear_estimate}, one obtains
\begin{equation}\label{Dtu_LinftyL2_L2H1_linear_estimate}
\sup_{t\in[0,T^*]}\|\bD_t^s\u(t)\|_{L^2(\Omega)}^2
+
\int_0^{T^*}\|\nabla\bD_t^s\u\|_{L^2(\Omega)}^2\,dt
\le C_0,
\end{equation}
provided $T^*(M+F(M))\le1$.  Here $C_0$ depends only on the data, not on $M$.

The elliptic estimate applied to \eqref{linear_momentum_eq} yields
\begin{align}
\|\u\|_{H^2(\Omega)}
&\le C\bigl(
\|\bD_t^s\u\|_{L^2(\Omega)}+
\|\btheta\|_{H^1(\Omega)}+
\|\f\|_{L^2(\Omega)}+
\|\sbV\|_{H^2(\Omega)}
\bigr), \label{H2_elliptic_estimate_for_linear_eq}\\
\|\u\|_{H^3(\Omega)}
&\le C\bigl(
\|\bD_t^s\u\|_{H^1(\Omega)}+
\|\btheta\|_{H^2(\Omega)}+
\|\f\|_{H^1(\Omega)}+
\|\brho\|_{H^2(\Omega)}\|\btheta\|_{L^\infty(\Omega)}
\bigr). \label{H3_elliptic_estimate_for_linear_eq}
\end{align}
Combining these estimates with \eqref{Dtu_LinftyL2_L2H1_linear_estimate} and
\eqref{theta_XTstar_linear_estimate}, and using the smallness of $T^*$, gives
\begin{equation}\label{u_LinftyH2_L2H3_linear_estimate}
\sup_{t\in[0,T^*]}\|\u(t)\|_{H^2(\Omega)}^2
+
\int_0^{T^*}\|\u\|_{H^3(\Omega)}^2\,dt
\le C_0.
\end{equation}
Since $\u_t=\bD_t^s\u-\bu\cdot\nabla\u$, we also obtain
\begin{equation}\label{ut_LinftyL2_L2H1_linear_estimate}
\sup_{t\in[0,T^*]}\|\u_t(t)\|_{L^2(\Omega)}^2
+
\int_0^{T^*}\|\nabla\u_t\|_{L^2(\Omega)}^2\,dt
\le C_0.
\end{equation}

Finally, differentiating \eqref{linear_momentum_eq} in time in the material form
implies, in $H^{-1}(\Omega)$,
\[
\brho(\bD_t^s)^2\u+
\nabla\bp_t+\div(\nabla\bp\otimes\bu)
=
\mu\Delta\u_t+
\mu\div(\Delta\u\otimes\bu)+\brho\bD_t^s\f .
\]
The estimates already obtained imply
\begin{equation}\label{utt_L2H-1_linear_estimate}
\int_0^{T^*}\|\u_{tt}\|_{H^{-1}(\Omega)}^2\,dt\le C_0.
\end{equation}
Combining \eqref{u_LinftyH2_L2H3_linear_estimate},
\eqref{ut_LinftyL2_L2H1_linear_estimate}, and
\eqref{utt_L2H-1_linear_estimate}, we conclude that
\[
 \|\u\|_{X_{T^*}}\le C_*,
\]
where $C_*$ depends only on the data and is independent of $M$.  This proves the
self-map estimate.
\end{proof}

\section*{Acknowledgments}
The authors used ChatGPT and Perplexity to assist with translating this
manuscript into English and polishing its language. The authors carefully
reviewed and revised all resulting text and assume responsibility for the
entire content.

\bibliographystyle{plain}
\bibliography{bibfile}

\end{document}